\documentclass{amsart}
\usepackage{ucs}
\usepackage[utf8x]{inputenc}
\usepackage{verbatim}
\usepackage{paralist}
\usepackage{leftidx}
\usepackage{dsfont}
\usepackage{amsthm}
\usepackage{amsmath}
\usepackage{amssymb}
\usepackage{amscd}
\usepackage{graphicx}
\usepackage{setspace}
\usepackage[all]{xy}
\usepackage{mathrsfs} 
\usepackage{hyperref}
\usepackage{mathtools}
\title{On the Irreducibility and Distribution of Arithmetic Divisors}
\author{Robert Wilms}
\address{Robert Wilms\\
	Department of Mathematics and Computer Science\\
	University of Basel\\
	Spiegelgasse 1\\
	4051 Basel\\
	Switzerland}
\email{robert.wilms@unibas.ch}
\subjclass[2010]{14G40.}
\begin{document}
	\numberwithin{equation}{section}
	\newtheorem{Def}{Definition}
	\numberwithin{Def}{section}
	\newtheorem*{Con}{Condition}
	\newtheorem{Rem}[Def]{Remark}
	\newtheorem{Lem}[Def]{Lemma}
	\newtheorem{Que}[Def]{Question}
	\newtheorem{Cor}[Def]{Corollary}
	\newtheorem{Exam}[Def]{Example}
	\newtheorem{Thm}[Def]{Theorem}
	\newtheorem*{clm}{Claim}
	\newtheorem{Pro}[Def]{Proposition}
	\newcommand\gf[2]{\genfrac{}{}{0pt}{}{#1}{#2}}

\begin{abstract}
We introduce the notion of $\epsilon$-irreducibility for arithmetic cycles meaning that the degree of its analytic part is small compared to the degree of its irreducible classical part. We will show that for every $\epsilon>0$ any sufficiently high tensor power of an arithmetically ample hermitian line bundle can be represented by an $\epsilon$-irreducible arithmetic divisor. Our methods of proof also allow us to study the distribution of divisors of small sections of an arithmetically ample hermitian line bundle $\overline{\mathcal{L}}$. We will prove that for increasing tensor powers $\overline{\mathcal{L}}^{\otimes n}$ the normalized Dirac measures of these divisors almost always converge to $c_1(\overline{\mathcal{L}})$ in the weak sense. Using geometry of numbers we will deduce this result from a distribution result on divisors of random sections of positive line bundles in complex analysis. As an application, we will give a new equidistribution result for the zero sets of integer polynomials. Finally, we will express the arithmetic intersection number of arithmetically ample hermitian line bundles as a limit of classical geometric intersection numbers over the finite fibers.
\end{abstract}
\maketitle
\section{Introduction}
The property of irreducibility of a divisor or more general of a cycle is a very fundamental notion in algebraic geometry. While it is most often useful to restrict from general cycles to irreducible cycles, there are also results and constructions only working for irreducible cycles. For example, Newton--Okounkov bodies \cite{Oko96, Oko03, LM09, KK12} rely on the choice of a flag of irreducible subvarieties and they have recently been shown to be very useful in intersection theory, see for example \cite{JL21}. 

In arithmetic intersection theory one studies an analogue of classical intersection theory on varieties flat and projective over $\mathrm{Spec}(\mathbb{Z})$, where additional analytic information playing the role of a hypothetical fiber at infinity. One aim of this paper is to introduce and study a new notion of irreducibility for (horizontal) arithmetic cycles (of Green type) as the naive definition does not make sense. Our main result can be seen as an analogue of Bertini's theorem. It states that for any $\epsilon>0$ and any arithmetically ample hermitian line bundle $\overline{\mathcal{L}}$ a sufficiently high tensor power $\overline{\mathcal{L}}^{\otimes n}$ can be represented by an arithmetic divisor which is irreducible up to $\epsilon$. For a more precise formulation we refer to Section \ref{sec_intro-irreducibility}.

The other aim of this paper is to study the distribution of the divisors of small sections of any arithmetically ample hermitian line bundle $\overline{\mathcal{L}}$. This turns out as an application of our methods of proof of our main result. We will show that the normalized Dirac measures of the divisors in a sequence of sections of $\overline{\mathcal{L}}^{\otimes n}$ almost always converge to $c_1(\overline{\mathcal{L}})$ in the weak sense for $n\to\infty$. As applications, we will discuss an analogue of the generalized Bogomolov conjecture for the set of global sections and we will study the distribution of zero sets of integer polynomials in a sequence. Finally, we will show how to compute arithmetic intersection numbers of arithmetically ample hermitian line bundles as a limit of classical geometric intersection numbers over the finite fibers.

We will explain our result in more detail in several subsections.

\subsection{Irreducibility of Arithmetic Cycles}\label{sec_intro-irreducibility}
By definition an arithmetic cycle (of Green type) $\overline{\mathcal{Z}}=(\mathcal{Z},g_{\mathcal{Z}})$ consists of a cycle $\mathcal{Z}$ of an arithmetic variety $\mathcal{X}$ defined over $\mathrm{Spec}~\mathbb{Z}$ and a Green current $g_{\mathcal{Z}}$ on $\mathcal{X}_\mathbb{C}$. It is called effective, if $\mathcal{Z}$ is effective and $g_{\mathcal{Z}}\ge 0$ on $\mathcal{X}_\mathbb{C}\setminus\mathrm{Supp}(\mathcal{Z}_\mathbb{C})$. Naively, one would define an arithmetic cycle $\overline{\mathcal{Z}}$ to be irreducible if for any decomposition $\overline{\mathcal{Z}}=\overline{\mathcal{Z}}_1+\overline{\mathcal{Z}}_2$ for two effective arithmetic cycles $\overline{\mathcal{Z}}_1$ and $\overline{\mathcal{Z}}_2$ we have $\overline{\mathcal{Z}}_1=0$ or $\overline{\mathcal{Z}}_2=0$. But with this notion every irreducible arithmetic cycle will lie in a fiber over a closed point $p\in\mathrm{Spec}~\mathbb{Z}$. Indeed, if $(\mathcal{D},g_\mathcal{D})$ is an effective arithmetic divisor such that $\mathcal{D}$ is horizontal, then $g_\mathcal{D}$ is a non-trivial Green function on $\mathcal{X}(\mathbb{C})$ and we can always find a non-zero, non-negative $C^\infty$-function $f$ on $\mathcal{X}(\mathbb{C})$ such that $g_\mathcal{D}=g'_\mathcal{D}+f$ for another non-negative Green function $g'_\mathcal{D}$ associated to the divisor $\mathcal{D}$. Thus $(\mathcal{D},g_\mathcal{D})=(\mathcal{D},g'_\mathcal{D})+(0,f)$ is not irreducible. As a consequence an analogue of Bertini's theorem is not true: An arithmetically ample hermitian line bundle is not representable by an irreducible arithmetic divisor. To avoid this issue, we introduce a new notion of irreducibility for arithmetic cycles, which measures how far an arithmetic cycle is from being irreducible in the above sense.

Before we present our definition, we clarify some notions. By a \emph{generically smooth projective arithmetic variety} $\mathcal{X}$ we mean an integral scheme which is projective, separated, flat and of finite type over $\mathrm{Spec}~\mathbb{Z}$ and such that $\mathcal{X}_{\mathbb{Q}}$ is smooth. The group of arithmetic cycles of dimension $p$ on $\mathcal{X}$ is denoted by $\widehat{Z}_p(\mathcal{X})$. Its quotient by rational equivalence is denoted by $\widehat{\mathrm{CH}}_p(\mathcal{X})$. We refer to Section \ref{sec_intersection} for more details on these groups.
\begin{Def}
	Let $\mathcal{X}$ be any generically smooth projective arithmetic variety.
	\begin{enumerate}[(i)]
		\item
		Let $(\mathcal{Z},g_{\mathcal{Z}})\in \widehat{Z}_{p}(\mathcal{X})$ be any effective arithmetic cycle on $\mathcal{X}$. For any positive real number $\epsilon>0$ and any arithmetically ample hermitian line bundle $\overline{\mathcal{M}}$ we say that $(\mathcal{Z},g_{\mathcal{Z}})$ is \emph{$(\epsilon,\overline{\mathcal{M}})$-irreducible} if $\mathcal{Z}$ is irreducible, we have
		\begin{align}\label{equ_air}
			\int_{\mathcal{X}(\mathbb{C})}g_{\mathcal{Z}}\wedge c_1(\overline{\mathcal{M}})^{p}< \epsilon\cdot (\overline{\mathcal{M}}^p\cdot(\mathcal{Z},0))
		\end{align}
		and $\frac{i}{2\pi}\partial\overline{\partial}g_{\mathcal{Z}}+\delta_{\mathcal{Z}(\mathbb{C})}$ is represented by a semi-positive form.
		\item We say that a class $\alpha\in\widehat{\mathrm{CH}}_p(\mathcal{X})$ in the arithmetic Chow group of $\mathcal{X}$ is \emph{air (asymptotically irreducibly representable)} if for any positive real number $\epsilon>0$ and any arithmetically ample hermitian line bundle $\overline{\mathcal{M}}$ there exists an $n\in\mathbb{Z}_{\ge 1}$ such that $n\cdot \alpha$ can be represented by an $(\epsilon,\overline{\mathcal{M}})$-irreducible arithmetic cycle $(\mathcal{Z},g_{\mathcal{Z}})$. 
		We call $\alpha$ \emph{generically smoothly air} if the arithmetic cycle $(\mathcal{Z},g_{\mathcal{Z}})$ can always be chosen such that $\mathcal{Z}$ is horizontal and $\mathcal{Z}_\mathbb{Q}$ is smooth.
	\end{enumerate}
\end{Def}
	
In other words, $(\mathcal{Z},g_\mathcal{Z})$ is $(\epsilon,\overline{\mathcal{M}})$-irreducible, if the $\overline{\mathcal{M}}$-degree of its part in the fiber at $\infty$ is sufficiently small compared to the $\overline{\mathcal{M}}$-degree of its irreducible classical part. One may ask, why we do not just bound the integral by $\epsilon$. The answer is, that this would give a too strong condition to prove the following theorem, which can be considered as an analogue of Bertini's theorem in this setting. For an even more Bertini-like result we refer to Remark \ref{rem_dimension1} (ii).
\begin{Thm}\label{thm_air}
	Let $\mathcal{X}$ be any generically smooth projective arithmetic variety of dimension $d\ge 2$. Every arithmetically ample hermitian line bundle $\overline{\mathcal{L}}$ on $\mathcal{X}$ is generically smoothly air. Moreover, if $\alpha\in\widehat{\mathrm{CH}}_p(\mathcal{X})$ is any generically smoothly air class for some $p\ge 2$, then also the class $\overline{\mathcal{L}}\cdot \alpha\in\widehat{\mathrm{CH}}_{p-1}(\mathcal{X})$ is generically smoothly air.
\end{Thm}
To prove the theorem, we have to show that the integral in (\ref{equ_air}) is asymptotically small. We will use an equidistribution result by Bayraktar, Coman and Marinescu \cite{BCM20} on the set of holomorphic sections of the tensor powers of $\overline{\mathcal{L}}(\mathbb{C})$. In particular, they studied the asymptotic behavior of the integral in question in the average over all sections with respect to a probability measure satisfying a certain condition. Using some geometry of numbers we will deduce from their result an equidistribution result on the discrete subset $H^0(\mathcal{X},\mathcal{L})$ in $H^0(\mathcal{X}(\mathbb{C}),\mathcal{L}(\mathbb{C}))$.
To state our result, we denote 
$$\widehat{H}_{\le r}^0(\mathcal{X},\overline{\mathcal{L}})=\{s\in H^0(\mathcal{X},\mathcal{L})~|~\|s\|_{\sup}\le r\}$$
for the global sections of $\sup$-norm at most $r$, where $\|s\|_{\sup}=\sup_{x\in\mathcal{X}(\mathbb{C})}|s(x)|_{\overline{\mathcal{L}}}$. We also write $\widehat{H}^0(\mathcal{X},\overline{\mathcal{L}})=\widehat{H}_{\le 1}^0(\mathcal{X},\overline{\mathcal{L}})$.
Then our equidistribution result can be expressed in the following way.
\begin{Thm}\label{thm_equidistribution}
	Let $\mathcal{X}$ be any generically smooth projective arithmetic variety and $\mathcal{Y}\subseteq\mathcal{X}$ any generically smooth arithmetic subvariety of dimension $e\ge 1$. Let $\overline{\mathcal{L}}$ be any arithmetically ample hermitian line bundle on $\mathcal{X}$ and $(r_p)_{p\in \mathbb{Z}_{\ge 1}}$ any sequence of positive real numbers satisfying $\lim_{p\to \infty}r_p^{1/p}=\tau\in[1,\infty)$. If we write $$S_p=\left\{\left. s\in \widehat{H}^0_{\le r_{p}}\left(\mathcal{X},\overline{\mathcal{L}}^{\otimes p}\right)~\right|~s|_{\mathcal{Y}}\neq 0\right\},$$
	then it holds
	$$\lim_{p\to \infty} \frac{1}{\# S_p}\sum_{s\in S_p}\frac{1}{p}\int_{\mathcal{Y}(\mathbb{C})}\left|\log|s|_{\overline{\mathcal{L}}^{\otimes p}}-p\log\tau\right|c_1(\overline{\mathcal{L}})^{e-1}=0.$$
\end{Thm}
This implies that for all $\epsilon>0$ the arithmetic divisors $\widehat{\mathrm{div}}(\overline{\mathcal{L}}^{\otimes p},s)$ are $(\epsilon,\overline{\mathcal{M}})$-irreducible for almost all $s\in \widehat{H}^0\left(\mathcal{X},\overline{\mathcal{L}}^{\otimes p}\right)$ for $p\to \infty$.
Theorem \ref{thm_air} follows from Theorem \ref{thm_equidistribution} using the fact, that the sections $s$ in $\widehat{H}^0\left(\mathcal{X},\overline{\mathcal{L}}^{\otimes p}\right)$ with irreducible and generically smooth divisors $\mathrm{div}(s)$ are dense for $p\to\infty$. This has recently been proved by Charles \cite{Cha21}.

One of the main difficulties in the proof of Theorem \ref{thm_equidistribution} is to conclude from the vanishing of the integral in the theorem for a sequence of sections in the real subspace  $H^0(\mathcal{X},\mathcal{L}^{\otimes n})_{\mathbb{R}}$ of $H^0(\mathcal{X}({\mathbb{C}}),\mathcal{L}({\mathbb{C}})^{\otimes n})$ to the vanishing of the integral for some sequence of sections in the lattice $H^0(\mathcal{X},\mathcal{L}^{\otimes n})$ in $H^0(\mathcal{X},\mathcal{L}^{\otimes n})_{\mathbb{R}}$. In Proposition \ref{pro_integral} we will prove a general result that if the integral vanishes for a sequence in $H^0(\mathcal{X}({\mathbb{C}}),\mathcal{L}({\mathbb{C}})^{\otimes n})$ then it also vanishes after a small change of this sequence. By a result of Moriwaki \cite{Mor11} we can always reach a sequence of lattice points by such a small change.

\subsection{Distribution of Divisors}
Although we presented Theorem \ref{thm_equidistribution} as a step in the proof of Theorem \ref{thm_air}, it is of its own interest and has many applications on the distribution of divisors of sections. We will discuss some equidistribution results for the divisors of sections in $\widehat{H}_{\le {r_p}}^0(\mathcal{X},\overline{\mathcal{L}}^{\otimes p})$. Let us first define the normalized height of any pure dimensional cycle $\mathcal{Z}$ of $\mathcal{X}$ by the arithmetic intersection number
$$h_{\overline{\mathcal{L}}}=\frac{\left(\overline{\mathcal{L}}|_{\mathcal{Z}}^{\dim \mathcal{Z}}\right)}{\mathcal{L}_{\mathbb{C}}|_{\mathcal{Z}_{\mathbb{C}}}^{\dim \mathcal{Z}_{\mathbb{C}}}}$$
of a fixed arithmetically ample hermitian line bundle $\overline{\mathcal{L}}$ restricted to $\mathcal{Z}$, where we set $(\mathcal{L}_{\mathbb{C}}|_{\mathcal{Z}_{\mathbb{C}}})^{\dim \mathcal{Z}_{\mathbb{C}}}=1$ if $\mathcal{Z}_{\mathbb{C}}=\emptyset$. We can relate the height of $\mathcal{X}$ to the height of any section $s\in H^0(\mathcal{X},\mathcal{L}^{\otimes p})\setminus \{0\}$ by the formula
\begin{align}\label{equ_induction-formula}
h_{\overline{\mathcal{L}}}(\mathcal{X})=h_{\overline{\mathcal{L}}}(\mathrm{div}(s)) -\frac{1}{p}\int_{\mathcal{X}(\mathbb{C})}\log |s|_{\overline{\mathcal{L}}^{\otimes p}}c_1(\overline{\mathcal{L}})^{d-1}.
\end{align}
We will discuss this formula in Section \ref{sec_restriction}. 
If the integral on the right hand side tends to $\frac{1}{p}\log\|s_p\|_{\sup}$ for a sequence of sections $s_p\in \widehat{H}^0_{\le r_p}(\mathcal{X},\overline{\mathcal{L}}^{\otimes p})$, one can show by Stokes' theorem that the normalized Dirac measure of $\mathrm{div}(s_p)({\mathbb{C}})$ weakly converges to $c_1(\overline{\mathcal{L}})$. Thus, we get the following proposition. 
\begin{Pro}\label{pro_equidistribution-individual}
	Let $\mathcal{X}$ be any generically smooth projective arithmetic variety of dimension $d\ge 2$. Let $\overline{\mathcal{L}}$ be any arithmetically ample hermitian line bundle on $\mathcal{X}$. For any sequence $(s_p)_{p\in\mathbb{Z}_{\ge 1}}$ of sections $s_p\in H^0(\mathcal{X},\overline{\mathcal{L}}^{\otimes p})$ satisfying 
	$$\lim_{p\to \infty}\left(h_{\overline{\mathcal{L}}}(\mathrm{div}(s_p))-\tfrac{1}{p}\log\|s_p\|_{\sup}\right)=h_{\overline{\mathcal{L}}}(\mathcal{X})$$
	and any $(d-2,d-2)$ $C^0$-form $\Phi$ on $\mathcal{X}(\mathbb{C})$ it holds
	$$\lim_{p\to \infty}\frac{1}{p}\int_{\mathrm{div}(s_p)(\mathbb{C})}\Phi=\int_{\mathcal{X}(\mathbb{C})}\Phi\wedge c_1(\overline{\mathcal{L}}).$$	
\end{Pro}
As a consequence of Theorem \ref{thm_equidistribution} and Equation (\ref{equ_induction-formula}) we get that the condition in Proposition \ref{pro_equidistribution-individual} is generically satisfied in $\widehat{H}^0_{\le r_p}(\mathcal{X},\overline{\mathcal{L}}^{\otimes p})$ if $\lim_{p\to \infty} r_p^{1/p}=\tau\ge 1$. This can again be used to obtain an equidistribution result on the average over all sections in $\widehat{H}^0_{\le r_p}(\mathcal{X},\overline{\mathcal{L}}^{\otimes p})\setminus\{0\}$. By arguing with subsequences we may not even assume that $r_p^{1/p}$ converges.
\begin{Cor}\label{cor_height-converges}
		Let $\mathcal{X}$ be any generically smooth projective arithmetic variety and $\mathcal{Y}\subseteq \mathcal{X}$ any generically smooth arithmetic subvariety of dimension $e\ge 2$. Let $\overline{\mathcal{L}}$ be any arithmetically ample hermitian line bundle on $\mathcal{X}$ and $(r_p)_{p\in \mathbb{Z}_{\ge 1}}$ any sequence with $r_p\in \mathbb{R}_{>0}$. We write $$S_p=\left\{\left.s\in\widehat{H}^0_{\le r_{p}}\left(\mathcal{X},\overline{\mathcal{L}}^{\otimes p}\right)~\right|~s|_{\mathcal{Y}}\neq 0\right\}.$$
	\begin{enumerate}[(i)]
		\item If $r_p^{1/p}$ converges to a value $\tau\in [1,\infty)$, then it holds
		$$\lim_{p\to \infty} \frac{1}{\# S_p}\sum_{s\in S_p}\left| h_{\overline{\mathcal{L}}}(\mathcal{Y})+\log\tau-h_{\overline{\mathcal{L}}}(\mathrm{div}(s)\cdot\mathcal{Y})\right|=0.$$
		\item If $1\le\liminf_{p\to \infty}r_p^{1/p}\le \limsup_{p\to \infty}r_p^{1/p}<\infty$, then for every $(e-2,e-2)$ $C^0$-form $\Phi$ on $\mathcal{Y}(\mathbb{C})$ it holds
		$$\lim_{p\to \infty}\frac{1}{\# S_p}\sum_{s\in S_p}\left|\frac{1}{p}\int_{\mathrm{div}(s|_{\mathcal{Y}})(\mathbb{C})}\Phi-\int_{\mathcal{Y}(\mathbb{C})}\Phi\wedge c_1(\overline{\mathcal{L}})\right|=0.$$
		In particular, we have
		$$\lim_{p\to \infty}\frac{1}{\# S_p}\sum_{s\in S_p}\frac{1}{p}\int_{\mathrm{div}(s|_{\mathcal{Y}})(\mathbb{C})}\Phi=\int_{\mathcal{Y}(\mathbb{C})}\Phi\wedge c_1(\overline{\mathcal{L}}).$$
	\end{enumerate}
\end{Cor}

\subsection{An Analogue of the Generalized Bogomolov Conjecture}
We want to discuss an analogue of the generalized Bogomolov conjecture for $\bigcup_{p\ge 1}H^0(\mathcal{X},\mathcal{L}^{\otimes p})$. First, let us recall the generalized Bogomolov conjecture. Let $A$ be an abelian variety defined over a number field $K$ and $L$ a symmetric ample line bundle on $A$. The height $h_{L}$ associated to $L$ is a height on the subvarieties of $A$ and it coincides with the N\'eron--Tate height for points. The generalized Bogomolov conjecture, proven by Zhang \cite{Zha98} based on an idea by Ullmo \cite{Ull98}, states that for any subvariety $X\subseteq A$, which is not a translate of an abelian 
subvariety by a torsion point, there exists an $\epsilon>0$, such that the geometric points $P\in X(\overline{K})$ of height $h_L(P)<\epsilon$ are not Zariski dense in $X$. But here, we focus on the previous result by Zhang in \cite[Theorem 1.10]{Zha95} that this property holds for a subvariety $X\subseteq A$ if $h_L(X)>h_L(A)$. Such subvarieties are called \emph{non-degenerate}.

We want to state an analogue result for $\bigcup_{p\ge 1}H^0(\mathcal{X},\mathcal{L}^{\otimes p})$. We may consider a global section $s\in H^0(\mathcal{X},\mathcal{L}^{\otimes p})$ as an analogue of a geometric point $P\in A(\overline{K})$, where $p$ is the analogue of the degree $[K(P):K]$. For a more detailed translation between the situation in the generalized Bogomolov conjecture and the situation in this paper we refer to Table \ref{tab_bogomolov}. Note that our situation is in some sense dual to the generalized Bogomolov conjecture, as we are not interested in embedding points of a subvariety $X$ into $A$, but in restricting global sections $s\in H^0(\mathcal{X},\mathcal{L}^{\otimes p})$ to an arithmetic subvariety $\mathcal{Y}\subseteq \mathcal{X}$.
The following corollary may be considered as an analogue of the generalized Bogomolov conjecture in our situation and is a consequence of Corollary \ref{cor_height-converges}.

\begin{Cor}\label{cor_bogomolov}
	Let $\mathcal{X}$ be any generically smooth projective arithmetic variety and $\overline{\mathcal{L}}$ any arithmetically ample hermitian line bundle on $\mathcal{X}$. Let $\mathcal{Y}\subseteq \mathcal{X}$ be any generically smooth arithmetic subvariety of dimension $e\ge 2$ with $h_{\overline{\mathcal{L}}}(\mathcal{Y})<h_{\overline{\mathcal{L}}}(\mathcal{X})$. For any $\epsilon\in (0,h_{\overline{\mathcal{L}}}(\mathcal{X})-h_{\overline{\mathcal{L}}}(\mathcal{Y}))$ and any sequence $(r_p)_{p\in\mathbb{Z}}$ of positive real numbers with $\lim_{p\to\infty}r_p^{1/p}=\tau\in [1,\infty)$ it holds
	$$\lim_{p\to\infty}\frac{\#\left\{s\in \widehat{H}^0_{\le r_p}(\mathcal{X},\overline{\mathcal{L}}^{\otimes p})~|~h_{\overline{\mathcal{L}}}(\mathrm{div}(s))-h_{\overline{\mathcal{L}}|_{\mathcal{Y}}}(\mathrm{div}(s|_{\mathcal{Y}}))\le\epsilon\right\}}{\#\widehat{H}^0_{\le r_p}(\mathcal{X},\overline{\mathcal{L}}^{\otimes p})}=0.$$
\end{Cor}

\begin{table}
\caption{An analogue of the generalized Bogomolov conjecture.}\label{tab_bogomolov}
\begin{tabular}{|p{5.9cm}|p{5.9cm}|}
	\hline
	\underline{Generalized Bogomolov conjecture} & \underline{Situation in this paper}\\[0.2cm] \hline
	Abelian variety $A$ defined over a number field $K$ & Generically smooth projective arithmetic variety $\mathcal{X}$\\[0.2cm] \hline
	Symmetric ample line bundle $L$& Arithmetically ample hermitian line bundle $\overline{\mathcal{L}}$\\[0.2cm]\hline
	Geometric points $P\in A(\overline{K})$& Sections $s\in R(\mathcal{L})=\bigcup_{p\ge 1} H^0(\mathcal{X},\mathcal{L}^{\otimes p})$\\[0.2cm]\hline
	Degree $[K(P):K]$ of $P\in A(\overline{K})$ & $p$ such that $s\in H^0(\mathcal{X},\mathcal{L}^{\otimes p})$\\[0.2cm] \hline
	N\'eron--Tate height $h_L(P)$ & Height $h_{\overline{\mathcal{L}}}(\mathrm{div}(s))$\\[0.2cm]\hline
	Northcott property: For $M\in\mathbb{R}$, $p\in\mathbb{Z}$ there are only finitely many $P\in A(\overline{K})$ with $h_L(P)\le M$ and $[K(P):K]=p$& Northcott property: For $M\in \mathbb{R}$ and $p\in\mathbb{Z}$ there are only finitely many $s\in H^0(\mathcal{X},\mathcal{L}^{\otimes p})/H^0(\mathcal{X},\mathcal{O}_{\mathcal{X}})^*$ with $h_{\overline{\mathcal{L}}}(\mathrm{div}(s))\le M$. \\[0.2cm]\hline
	Subvariety $X\subseteq A$ & Generically smooth arithmetic subvariety $\mathcal{Y}\subseteq \mathcal{X}$\\[0.2cm] \hline
	$X$ non-degenerate: $h_L(X)>h_L(A)$ & $\mathcal{Y}$ \emph{non-degenerate}: $h_{\overline{\mathcal{L}}}(\mathcal{Y})<h_{\overline{\mathcal{L}}}(\mathcal{X})$\\[0.2cm] \hline
	Geometric point $P\in X(\overline{K})$ & Restriction $s|_{\mathcal{Y}}\in H^0(\mathcal{Y},\mathcal{L}|_{\mathcal{Y}}^{\otimes p})$ of a global section $s\in H^0(\mathcal{X},\mathcal{L}^{\otimes p})$ to $\mathcal{Y}$\\[0.2cm] \hline
	N\'eron--Tate height $h_L(P)$ on $X(\overline{K})$ & $d_{\mathcal{Y}}(s)=h_{\overline{\mathcal{L}}}(\mathrm{div}(s))-h_{\overline{\mathcal{L}}|_{\mathcal{Y}}}(\mathrm{div}(s|_{\mathcal{Y}}))$.\\[0.2cm] \hline
	Zariski density of a subset $M\subseteq X(\overline{K})$ in $X$ &
	We say that a subset of sections $M\subseteq R(\mathcal{L})$ is of \emph{positive density}, if
	$\limsup_{p\to \infty}\frac{\#M\cap \widehat{H}_{\le r^p}^0(\mathcal{X},\overline{\mathcal{L}}^{\otimes p})}{\#\widehat{H}_{\le r^p}^0(\mathcal{X},\overline{\mathcal{L}}^{\otimes p})}>0$ for some $r\ge 1$.\\[0.2cm] \hline
	Generalized Bogomolov conjecture: If $X$ is non-degenerate, then $\exists\epsilon>0$ such that
	$\{P\in X(\overline{K})~|~h_L(P)<\epsilon\}$
	is not Zariski dense in $X$.& If $\mathcal{Y}$ is non-degenerate, then $\exists \epsilon>0$, such that
	$\{s\in R(\mathcal{L})~|~d_\mathcal{Y}(s)<\epsilon\}$ is not of positive density.\\[0.2cm]\hline
	\end{tabular}
\end{table}

\subsection{Distribution of Zeros of Polynomials}
Next, we discuss results on the distribution of zeros of integer polynomials by applying the above results to $\mathcal{X}=\mathbb{P}^1_\mathbb{Z}$ and the line bundle $\mathcal{L}=\mathcal{O}(1)$ equipped with the Fubini--Study metric multiplied by some $e^{-\epsilon}$. This may be also seen as a result on the distribution of algebraic numbers. For any polynomial $P\in \mathbb{C}[X]$ of the form $P=a_n\prod_{k=1}^{n}(X-\alpha_k)=\sum_{k=0}^n a_nX^n$ with $a_n\neq 0$ we define the values
$$h_{\mathrm{FS}}(P)=\tfrac{1}{n}\log|a_n|+\tfrac{1}{2n}\sum_{k=1}^n\log(1+|\alpha_k|^2) \quad \text{and}\quad h_{B}(P)=\tfrac{1}{n}\log \max_{0\le k\le n}\frac{|a_k|}{\sqrt{\tbinom{n}{k}}},$$
which may be considered as heights associated to the Fubini--Study metric and to the Bombieri $\infty$-norm. First, we state the analogue of Proposition \ref{pro_equidistribution-individual} in this case, which can be even proved for polynomials with complex coefficients.
\begin{Pro}\label{pro_equidistribution-polynomials}
	For any sequence $(P_n)_{n\in\mathbb{Z}_{\ge 1}}$ of polynomials
	$P_n\in \mathbb{C}[X]$ of degree $\deg P_n=n$ satisfying 
	$\limsup_{n\to\infty} (h_{B}(P_n)+\frac{1}{2}-h_{\mathrm{FS}}(P_n))\le0$
	and any continuous function $f\colon \mathbb{C}\to\mathbb{C}$, such that $\lim_{|z|\to\infty}f(z)$ is well-defined and finite, it holds
	$$\lim_{n\to \infty}\frac{1}{n}\sum_{\gf{z\in\mathbb{C}}{P_n(z)=0}}f(z)=\frac{i}{2\pi}\int_\mathbb{C}f(z)\frac{dzd\overline{z}}{\left(1+|z|^2\right)^2}.$$
\end{Pro}
Note that we always take the sum over the zeros of a polynomial counted with their multiplicities. We remark that under the condition in Proposition \ref{pro_equidistribution-polynomials} we automatically get
$$\lim_{n\to \infty}(h_B(P_n)+\tfrac{1}{2}-h_{\mathrm{FS}}(P_n))=0.$$
In this special situation Corollary \ref{cor_height-converges} can be explicitly formulated for polynomials as in the following corollary. In particular, we get that the condition in Proposition \ref{pro_equidistribution-polynomials} to the sequence $(P_n)_{n\in\mathbb{Z}_{\ge 1}}$ is generically satisfied in sets of polynomials of bounded height $h_B$.
\begin{Cor}\label{cor_distribution-algebraic-numbers}
	For any $n\in\mathbb{Z}_{\ge 0}$ and any $r\in\mathbb{R}_{>0}$ we define
	$$\mathcal{P}_{n,r}=\left\{\left.P\in \mathbb{Z}[X]~\right|~\deg P=n,~ h_B(P)\le r\right\}.$$
	Let $(r_n)_{n\in\mathbb{Z}_{\ge 1}}$ be any  sequence of real numbers.
	\begin{enumerate}[(i)]
		\item If $r_n$ converges to a value $\tau\in (0,\infty)$, then it holds
		$$\lim_{n\to \infty} \frac{1}{\#\mathcal{P}_{n,r_n}}\sum_{P\in \mathcal{P}_{n,r_n}} \left|\tau+\tfrac{1}{2}-h_{\mathrm{FS}}(P)\right|=0.$$
		\item If $0<\liminf_{n\to \infty} r_n\le \limsup_{n\to \infty} r_n<\infty$, then for any continuous function $f\colon \mathbb{C}\to\mathbb{C}$, such that $\lim_{|z|\to\infty}f(z)$ is well-defined and finite, it holds
		$$\lim_{n\to \infty} \frac{1}{\#\mathcal{P}_{n,r_n}}\sum_{P\in \mathcal{P}_{n,r_n}} \left|\frac{1}{n}\sum_{\gf{z\in\mathbb{C}}{P(z)=0}}f(z)-\frac{i}{2\pi}\int_\mathbb{C}f(z)\frac{dzd\overline{z}}{\left(1+|z|^2\right)^2}\right|=0.$$
		In particular, we have
		$$\lim_{n\to \infty} \frac{1}{\#\mathcal{P}_{n,r_n}}\sum_{P\in \mathcal{P}_{n,r_n}} \frac{1}{n}\sum_{\gf{z\in\mathbb{C}}{P(z)=0}}f(z)=\frac{i}{2\pi}\int_\mathbb{C}f(z)\frac{dzd\overline{z}}{\left(1+|z|^2\right)^2}.$$
	\end{enumerate}
\end{Cor}

Let us compare this result with other known results on the distribution of the zero sets in sequences of polynomials. First, it has been worked out by Erdős--Turán \cite{ET50}, see also \cite{Gra07}, that if a sequence of polynomials
$$P_n(X)=\sum_{k=0}^na_{n,k}X^k=a_{n,n}\prod_{k=1}^n(X-\alpha_{n,k})\in\mathbb{C}[X]$$
satisfies $\lim_{n\to \infty}h_{\mathrm{ET}}(P_n)=0$ for
$$h_{\mathrm{ET}}(P_n)=\tfrac{1}{n}\log\frac{\sum_{j=0}^n|a_n|}{\sqrt{|a_{n,n}a_{n,0}|}},$$
then the measures $\mu_n=\frac{1}{n}\sum_{k=1}^n\delta_{\alpha_{n,k}}$ weakly converge to the Haar measure on the unit circle $S^1\subseteq \mathbb{C}$. Bilu \cite{Bil97} proved a stronger version of this result for integer polynomials. To state his result, recall that the height associated to the Mahler measure $M(P_n)$ of $P_n$ is given by $$h_M(P_n)=\frac{1}{n}\log M(P_n)=\frac{1}{n}\log|a_{n,n}|+\frac{1}{n}\sum_{k=1}^n\log\max\{1,|\alpha_{n,k}|\}.$$
Then for all sequences $(P_n)_{n\in\mathbb{Z}_{\ge 1}}$ with $P_n\in\mathbb{Z}[X]$, $a_{n,n}\neq 0$ and $a_{n,0}\neq 0$ for all $n$ and $\lim_{n\to \infty}h_M(P_n)=0$
the measures $\mu_n=\frac{1}{n}\sum_{k=1}^n\delta_{\alpha_{n,k}}$ weakly converge to the Haar measure on $S^1$. If $P_n$ is irreducible in $\mathbb{Z}[X]$, then $h(\alpha_{n,k})=h_M(P_n)$ is the classical height of $\alpha_{n,k}$ for all $1\le k\le n$.
Let us also mention, that Pritsker \cite{Pri11} proved that $\mu_n$ weakly converges  to the Haar measure on $S^1$ if all $P_n\in\mathbb{Z}[X]$ have only simple zeros and we have $|\alpha_{n,k}|\le 1$, $a_{n,n}\neq 0$ and $|a_{n,k}|\le M$ for some $M\in \mathbb{R}$ for all $n$ and all $k\le n$. 

While the results by Erdős--Turán und Bilu yield equidistribution on $S^1$ for the zero sets of sequences of polynomials if the non-negative heights $h_{\mathrm{ET}}$ or $h_{M}$ tend to zero, our result gives equidistribution on the whole complex plane $\mathbb{C}$ weighted by the Fubini--Study measure if the difference of heights $h_B-h_{\mathrm{FS}}$ tends to its minimal limit point $-\frac{1}{2}$. Even more, the latter condition is generically satisfied in the sets of polynomials of bounded height $h_B$.

\subsection{Applications to Arithmetic Intersections Numbers}
Let us now discuss computations of arithmetic intersection numbers as an application of air classes.
For any $\alpha\in \widehat{\mathrm{CH}}_p(\mathcal{X})$ we write $(\alpha)_{\mathrm{Eff}}\subseteq \widehat{Z}_p(\mathcal{X})$ for the subset of arithmetic cycles of Green type which are effective and represent $\alpha$. One may consider $(\alpha)_{\mathrm{Eff}}$ as an analogue of a complete linear system in classical algebraic geometry.
\begin{Pro}\label{pro_intersection-finite}
	Let $\mathcal{X}$ be any generically smooth projective arithmetic variety. 
	For any arithmetic cycle $\alpha\in\widehat{\mathrm{CH}}_p(\mathcal{X})$, which is air, and any arithmetically ample hermitian line bundles $\mathcal{L}_1,\dots,\mathcal{L}_{p}$ we can compute their arithmetic intersection number in the following way
	$$(\overline{\mathcal{L}}_1\cdots\overline{\mathcal{L}}_{p}\cdot\alpha)=\limsup_{n\to\infty}\sup_{(\mathcal{Z},g_{\mathcal{Z}}) \in (n\alpha)_{\mathrm{Eff}}} \tfrac{1}{n}(\overline{\mathcal{L}}_1\cdots\overline{\mathcal{L}}_{p}\cdot (\mathcal{Z},0)).$$
\end{Pro}
By Theorem \ref{thm_air} and induction one can completely compute the arithmetic intersection number $(\overline{\mathcal{L}}_1\cdots\overline{\mathcal{L}}_d)$ of arithmetically ample hermitian line bundles $\overline{\mathcal{L}}_1,\dots,\overline{\mathcal{L}}_d$ by a limit of classical geometric intersection numbers. With some more effort we even get the following more explicit formula.
\begin{Thm}\label{thm_intersection-line-bundles-geometric}
	Let $\mathcal{X}$ be any generically smooth projective arithmetic variety. Further, let $\mathcal{Y}\subseteq \mathcal{X}$ be any generically smooth arithmetic subvariety of dimension $e\ge 1$.	
	For any arithmetically ample hermitian line bundles $\overline{\mathcal{L}}_1,\dots,\overline{\mathcal{L}}_e$ and any $n\in\mathbb{Z}_{\ge 1}$ we define
	$$H_n=\left\{\left.(s_1,\dots,s_e)\in\prod_{i=1}^e \widehat{H}^0\left(\mathcal{X},\overline{\mathcal{L}}_i^{\otimes n}\right)~\right|~\dim\left(\mathcal{Y}\cap\bigcap_{i=1}^e \mathrm{Supp}(\mathrm{div}(s_i))\right)=0\right\}.$$
	Then the arithmetic intersection number $(\overline{\mathcal{L}}_1|_{\mathcal{Y}}\cdots\overline{\mathcal{L}}_e|_{\mathcal{Y}})$ can be computed by
	\begin{align*}
		(\overline{\mathcal{L}}_1|_{\mathcal{Y}}\cdots\overline{\mathcal{L}}_e|_{\mathcal{Y}})=\lim_{n\to\infty}\frac{1}{n^e}\max_{(s_1,\dots,s_e)\in H_n}\sum_{p\in|\mathrm{Spec}(\mathbb{Z})|}i_p(\mathcal{Y}\cdot\mathrm{div}(s_1)\cdots\mathrm{div}(s_e))\log p
	\end{align*}
	where $i_p(\mathcal{Y}\cdot\mathrm{div}(s_1)\cdots\mathrm{div}(s_e))$ denotes the degree of the $0$-cycle $\mathcal{Y}\cdot\mathrm{div}(s_1)\cdots\mathrm{div}(s_e)$ in the fiber $\mathcal{X}_p$ of $\mathcal{X}$ over $p\in |\mathrm{Spec}(\mathbb{Z})|$.
\end{Thm}
As every hermitian line bundle can be expressed as the difference of two arithmetically ample hermitian line bundles, we can by multi-linearity compute the arithmetic intersection number of any hermitian line bundles as a limit of classical geometric intersection numbers over the finite fibers.
\subsection{Outline}
In Section \ref{sec_geometry-of-numbers} we define some notions in geometry of numbers and we discuss a recent result by Freyer and Lucas \cite{FL21}. In Section \ref{sec_complex-analysis}
we establish the complex analytic tools needed in this paper. After recalling some preliminaries from complex analysis, we discuss a distribution result on sequences of divisors of sections of an ample line bundle by Bayraktar, Coman and Marinescu \cite{BCM20}. Subsequently, we will show that the condition for such a sequence to equidistribute stays true after a small change of the sections.
We recall some definitions and basic facts from arithmetic intersection theory in Section \ref{sec_arithmetc-intersection-theory}. As an example, we will consider the projective line $\mathbb{P}_{\mathbb{Z}}^1$ and the hermitian line bundle $\overline{\mathcal{O}(1)}$ equipped with the Fubini--Study metric. We will discuss the behavior of arithmetic intersection numbers under restrictions in more details. In particular, we will discuss heights and prove Proposition \ref{pro_equidistribution-individual}.

We start Section \ref{sec_distribution} with the proof of Theorem \ref{thm_equidistribution}. Actually, we will prove a much more general version of it. Further, we will discuss the distribution of divisors in sequences of small sections. In particular, we will prove Corollaries \ref{cor_height-converges} and \ref{cor_bogomolov}. As an example, we will consider the distribution of divisors of sections of the line bundle $\overline{\mathcal{O}(1)}^{\otimes n}$ on $\mathbb{P}_{\mathbb{Z}}^1$. We will interpret this as an equidistribution result on the zero sets of integer polynomials. Especially, we give proofs of Proposition \ref{pro_equidistribution-polynomials} and Corollary \ref{cor_distribution-algebraic-numbers}. In the last section we discuss the notion of $(\epsilon,\overline{\mathcal{M}})$-irreducible arithmetic cycles and of air classes. By proving Theorem \ref{thm_air} we will show that every arithmetically ample hermitian line bundle is air. Finally, we give the proof of Theorem \ref{thm_intersection-line-bundles-geometric}.

\section{Geometry of Numbers}\label{sec_geometry-of-numbers}
In this section we recall some notions and results from geometry of numbers. For the basics of the theory of geometry of numbers we refer to Moriwaki's book \cite[Chapter 2]{Mor14}.

Let $V$ be an euclidean vector space of dimension $n$. By a \emph{lattice} $\Lambda\subseteq V$ we mean a free $\mathbb{Z}$-submodule of rank $n$ which spans $V$ as an $\mathbb{R}$-vector space. We denote $B_t=\{w\in V~|~\|w\|<t\}$ for the open ball of radius $t\in\mathbb{R}$ around the origin. For any lattice $\Lambda\subseteq V$ and any compact, convex, and symmetric subset $K\subseteq V$ of positive volume we define the successive minima by
$$\lambda_{j}(K,\Lambda)=\min\{\lambda>0~|~\dim(\mathrm{span}(\lambda K\cap \Lambda))=j\},$$
where $1\le j\le n$. Moreover, we set 
$$\lambda_j(\Lambda)=\lambda_j\left(\overline{B_1},\Lambda\right).$$
Further, we define
$$\lambda_\mathbb{Z}(\Lambda)=\min\{\lambda>0~|~\exists~ \mathbb{Z}\text{-basis } x_1,\dots,x_n\text{ of } \Lambda\text{ with } \|x_j\|\le \lambda \text{ for all } j.\}$$
By \cite[Lemma 2.18]{Mor14} it holds 
$$\lambda_{n}(\Lambda)\le \lambda_{\mathbb{Z}}(\Lambda)\le n\lambda_{n}(\Lambda).$$

Freyer and Lucas \cite{FL21} recently proved that for all compact, convex, and symmetric subsets $K\subseteq \mathbb{R}^n$ of positive volume it holds
\begin{align}\label{equ_latticebound}
	\# K\cap \mathbb{Z}^n\le \mathrm{Vol}(K)\prod_{j=1}^n\left(1+\frac{n\lambda_j(K,\mathbb{Z}^n)}{2}\right).
\end{align}
If $K$ moreover satisfies $\lambda_n(K,\mathbb{Z}^n)\le \frac{2}{n}$, then they also showed
\begin{align}\label{equ_latticebound2}
	\#\mathrm{int}(K)\cap\mathbb{Z}^n\ge \mathrm{Vol}(K)\prod_{j=1}^n\left(1-\frac{n\lambda_j(K,\mathbb{Z}^n)}{2}\right),
\end{align}
where $\mathrm{int}(K)$ denotes the interior of $K$.
Let us rewrite this result for arbitrary lattices $\Lambda\subseteq V$. We choose a $\mathbb{Z}$-basis $v_1,\dots, v_{n}$ of $\Lambda$. The determinant of $\Lambda$ is defined by
$$\det(\Lambda)=|{\det}(v_1,\dots,v_{n})|.$$
It is independent of the choice of the basis. Denote $\varphi\colon V\to \mathbb{R}^{n}$ for the isomorphism associated to the basis $v_1,\dots,v_n$. It reduces to an isomorphism $\Lambda\cong\varphi(\Lambda)=\mathbb{Z}^{n}$. For any compact, convex, and symmetric subset $K\subseteq V$ the image $\varphi(K)\subseteq \mathbb{R}^{n}$ is again compact, convex, and symmetric and its volume satisfies
$$\mathrm{Vol}(\varphi(K))=\frac{\mathrm{Vol}(K)}{\det(\Lambda)}.$$
As $\varphi$ identifies $j$-tuples of linearly independent $\Lambda$-lattice points in $K$ with $j$-tuples of linearly independent integral points in $\varphi(K)$ for every $1\le j\le n$, the successive minima satisfy
$$\lambda_j(K,\Lambda)=\lambda_j(\varphi(K),\mathbb{Z}^{n}).$$
We also have $\#K\cap \Lambda=\#\varphi(K\cap\Lambda)=\#\varphi(K)\cap \mathbb{Z}^{n}$.
Applying (\ref{equ_latticebound}) to $\varphi(K)$ and using the conclusions above we get
\begin{align}\label{equ_latticebound-general}
	\# (K\cap \Lambda) \le \frac{\mathrm{Vol}(K)}{\det(\Lambda)}\left(1+\frac{n\lambda_{n}(K,\Lambda)}{2}\right)^{n},
\end{align}
where we additionally used the trivial bounds $\lambda_j(K,\Lambda)\le \lambda_{n}(K,\Lambda)$ for all $j\le n$.
Under the condition $\lambda_n(K,\Lambda)\le \frac{2}{n}$ we similarly deduce from (\ref{equ_latticebound2}) that
\begin{align}\label{equ_latticebound-general2}
	\#(\mathrm{int}(K)\cap\Lambda)\ge\frac{\mathrm{Vol}(K)}{\det(\Lambda)}\left(1-\frac{n\lambda_{n}(K,\Lambda)}{2}\right)^{n}.
\end{align}

One directly checks that for all $\mu>0$ it holds 
$$\lambda_j(\mu K,\Lambda)=\mu^{-1}\lambda_j(K,\Lambda),\qquad \mathrm{Vol}(\mu K)=\mu^n\mathrm{Vol}(K)$$
for all $j\le n$.
If $K_1\subseteq K_2$ are two compact, convex and symmetric sets of positive volume in $V$, then it holds
$$\lambda_j(K_2,\Lambda)\le \lambda_j(K_1,\Lambda)$$
for all $j\le n$. For later use we prove the following upper bound for the quotient of the numbers of lattice points in $K$ and $\mu K$.

\begin{Lem}\label{lem_geometry-of-numbers}
	Let $\Lambda$ be a lattice in an euclidean vector space $V$ of dimension $n$.
	Let $r>0$ and $\mu>0$ be two real numbers satisfying $r\mu\ge n\lambda_{\mathbb{Z}}(\Lambda)$ and $K\subseteq V$ be a compact, convex and symmetric space satisfying $\overline{B_r}\subseteq K$.
	Then it holds
	$$\frac{\#( K\cap \Lambda)}{\#(\mu K\cap \Lambda)}\le\mu^{-n}\left(1+n(1+\mu^{-1})r^{-1}\lambda_{\mathbb{Z}}(\Lambda)\right)^n $$
\end{Lem}
\begin{proof}
	Since $r\overline{B_1}=\overline{B_r}\subseteq K$, we have 
	$$\lambda_n(K,\Lambda)\le \lambda_n(r\overline{B_1},\Lambda)=r^{-1}\lambda_n(\Lambda)\le r^{-1}\lambda_{\mathbb{Z}}(\Lambda).$$
	In particular, we have $\lambda_n(\mu K,\Lambda)\le \mu^{-1} r^{-1}\lambda_{\mathbb{Z}}(\Lambda)\le \frac{1}{n}$. Hence, we can apply (\ref{equ_latticebound-general}) to $K$ and (\ref{equ_latticebound-general2}) to $\mu K$ to obtain
	\begin{align*}
		\frac{\#(K\cap \Lambda)}{\#(\mu K\cap \Lambda)}&\le \frac{(1+\frac{n}{2}\lambda_n(K,\Lambda))^n}{\mu^n(1-\frac{n}{2}\mu^{-1}\lambda_n(K,\Lambda))^n}\le\mu^{-n}\left(\frac{1+\frac{n}{2}r^{-1}\lambda_{\mathbb{Z}}(\Lambda)}{1-\frac{n}{2}\mu^{-1}r^{-1}\lambda_{\mathbb{Z}}(\Lambda)}\right)^n\\
		&=\mu^{-n}\left(1+\frac{\frac{n}{2}(1+\mu^{-1})r^{-1}\lambda_{\mathbb{Z}}(\Lambda)}{1-\frac{n}{2}\mu^{-1}r^{-1}\lambda_{\mathbb{Z}}(\Lambda)}\right)^n\\
		&\le \mu^{-n}\left(1+n(1+\mu^{-1})r^{-1}\lambda_{\mathbb{Z}}(\Lambda)\right)^n.
	\end{align*}
	This proves the lemma.
\end{proof}

\section{Complex Analysis}\label{sec_complex-analysis}
In this section we establish the complex analytic tools needed in our study of arithmetic intersection theory. After collecting some basic facts in Section \ref{sec_preliminaries} we will discuss results on the distribution of the divisors of global sections by Bayraktar, Coman and Marinescu \cite{BCM20} in Section \ref{sec_distribution-complex}. As the divisors in a sequence of sections $(s_{p})_{p\in\mathbb{Z}_{\ge 1}}$ of a positive line bundle $\overline{L}^{\otimes p}$ tends to equidistribute with respect to $c_1(\overline{L})$ for $p\to\infty$ if the integral of $|\log|s_p|^{1/p}|$ tends to $0$, their studies concentrate on this integral. 
In Section \ref{sec_integral-higher-dim} we will show, that if the integral tends to $0$, it will still do so after a small change of the sequence $(s_{p})_{p\in\mathbb{Z}}$. We study the distribution result for certain real vector subspaces of the global sections in Section \ref{sec_real-space}.
\subsection{Preliminaries}\label{sec_preliminaries}
We recall some preliminaries on complex analysis especially applied to smooth projective complex varieties.

Let $X$ be any smooth projective complex variety of dimension $n\ge 0$. Note that we do not assume that varieties	are connected.
Let $\mathscr{A}_X^{p,q}$ be the sheaf of $C^{\infty}$ $(p,q)$-forms on $X$ and set $A^{p,q}(X)=H^0(X,\mathscr{A}_X^{p,q})$. We denote the usual Dolbeault operators
$$\partial\colon A^{p,q}(X)\to A^{p+1,q}(X),\qquad \overline{\partial}\colon A^{p,q}(X)\to A^{p,q+1}(X)$$
and $d=\partial+\overline{\partial}$ on $A^r(X)=\bigoplus_{p+q=r}A^{p,q}(X)$. A \emph{current of type} $(p,q)$ on $X$ is a continuous $\mathbb{C}$-linear map $T\colon A^{n-p,n-q}(X)\to \mathbb{C}$. We write $D^{p,q}(X)$ for the vector space of currents of type $(p,q)$ on $X$. We can consider $A^{p,q}$ as a subspace of $D^{p,q}(X)$ by associating to $\omega\in A^{p,q}(X)$ the current of type $(p,q)$
$$[\omega]\colon A^{n-p,n-q}(X)\to \mathbb{C},\qquad \eta\mapsto\int_X\omega\wedge \eta.$$
We will often omit the brackets and also write $\omega$ for the associated current.
Another important family of currents is associated to subvarieties of $X$. If $Y\subseteq X$ is a reduced subvariety of codimension $p$, then the \emph{current of Dirac type} $\delta_Y\in D^{p,p}(X)$ is defined by
$$\delta_Y\colon A^{n-p,n-p}(X)\to\mathbb{C},\qquad  \eta\mapsto \int_Y\eta:= \int_{Y_{\mathrm{reg}}}\eta,$$
where $Y_{\mathrm{reg}}$ denotes the non-singular locus of $Y$. If $Y=\sum_{j=1}^r n_j Y_j$ is any cycle of pure codimension $p$, we also write $\delta_Y=\sum_{j=1}^r n_j \delta_{Y_j}\in D^{p,p}(X)$. 

We say that $T\in D^{p,p}(X)$ is \emph{real} if for all $\eta\in A^{p,p}(X)$ it holds $T(\overline{\eta})=\overline{T(\eta)}$. We say that a current $T\in D^{p,p}(X)$ is \emph{semi-positive}, written as $T\ge 0$, if it is real and for all positive real forms $\eta\in A^{p,p}(X)$ we have $T(\eta)\ge 0$. For any current $T\in D^{n,n}(X)$ we use the notation
$$\int_X T=T(1).$$
In particular, this implies $\int_X[\omega]=\int_X\omega$ for any $\omega\in A^{n,n}(X)$.

If $T\in D^{p,q}(X)$ is a current, we define $\partial T\in D^{p+1,q}$ to be the linear map $T\circ \partial\colon A^{n-p-1,n-q}(X)\to \mathbb{C}$. Analogously, we can define $\overline{\partial}T$. If $Z$ is a codimension $p$ cycle of $X$, we call $g\in D^{p-1,p-1}(X)$ a \emph{Green current} of $Z$ if $$\frac{i}{2\pi}\partial\overline{\partial}g+\delta_Z=[\omega]$$ for some $(p,p)$-form $\omega\in A^{p,p}(X)$.

If $f\colon Y\to X$ is a projective morphism of smooth projective complex varieties, we define the \emph{push-forward} $f_*T\in D^{p-\dim Y+\dim X,q-\dim Y+\dim X}(X)$ of any current $T\in D^{p,q}(Y)$ by 
$$f_*T\colon A^{\dim Y-p,\dim Y-q}(X)\to \mathbb{C},\qquad \eta\mapsto T(f^*\eta).$$

Now let $L$ be any line bundle on $X$. By a \emph{hermitian metric} $h$ on $L$ we mean a family $(h_x)_{x\in X}$ of hermitian metrics on the fiber $L_x$ of $L$ over $x\in X$ for any point $x$ of $X$. For any Zariski open subset $U\subseteq X$ and any section $s\in H^0(U,L)$ we write $|s(x)|=|s(x)|_h=\sqrt{h_x(s(x),s(x))}$ for the norm of the section at $x$.
We call a hermitian metric $h$ \emph{smooth}, if for any Zariski open subset $U\subseteq X$ and any section $s\in H^0(U,L)$ the map $U\to \mathbb{R},~x\mapsto |s(x)|^2$ is smooth. We call the pair $\overline{L}=(L,h)$ a \emph{hermitian line bundle} if $L$ is a line bundle and $h$ a smooth hermitian metric on $L$.

The hermitian line bundle $(L^{\otimes p},h^{\otimes p})$ obtained by taking the tensor power of $(L,h)$ will play an important role in the upcoming sections. 
If $f\colon Y\to X$ is a projective morphism of smooth projective complex varieties and $\overline{L}=(L,h)$ a hermitian line bundle on $X$, we define its pullback $f^*\overline{L}$ by the pair $(f^*L,f^*h)$, where the hermitian metric $f^*h$ is determined by $$|(f^*s)(y)|_{f^*h}=|s(f(y))|_{h}$$ for any open subset $U\subseteq X$, any section $s\in H^0(U,L)$ and any point $y\in f^{-1}(U)$.

If $(L,h)$ is a hermitian line bundle and $s$ a rational and non-zero section of $L$, we may think of $-\log|s|^2$ as the current sending $\omega\in A^{n,n}$ to $-\int_{X}\log|s|^2\omega$.
It turns out that $-\log |s|^2$ is a Green current of type $\mathrm{div}(s)$.  Indeed, by the Poincaré--Lelong formula we have
\begin{align}\label{equ_poincare-lelong}
\frac{i}{2\pi}\partial\overline{\partial}\left(-\log |s|^2\right)+\delta_{\mathrm{div}(s)}=[c_1(L,h)],
\end{align}
where $c_1(L,h)$ denotes the first Chern form of the metrized line bundle $(L,h)$. We call the hermitian metric $h$ \emph{positive} if $c_1(L,h)$ is a positive form.
If $L$ is ample and $h$ is a positive hermitian metric on $L$, the Chern form $c_1(L,h)$ is a Kähler form on $X$. 

In the following we fix a Kähler form $\omega$ on $X$. Then we obtain a hermitian form on the set of holomorphic sections $H^0(X,L)$ of any hermitian line bundle $(L,h)$ by setting
\begin{align}\label{equ_inner-product}
\langle s_1,s_2\rangle =\left(\int_X\omega^n\right)^{-1}\int_X h(s_1(x),s_2(x))\omega^n
\end{align}
for any $s_1,s_2\in H^0(X,L)$. We write $\|s\|=\sqrt{\langle s,s\rangle}$ for any $s\in H^0(X,L)$ and we call $\|s\|$ the $L^2$-norm of $s$.
There is another natural norm on $H^0(X,L)$, called the \emph{sup-norm}, defined by
$$\|s\|_{\sup}=\sup_{x\in X}|s(x)|$$
for any section $s\in H^0(X,L)$. 
We recall, that there exists a constant $C_1\ge 1$ depending only on $(X,\omega)$ and $(L,h)$ but not on $p$, such that
\begin{align}\label{equ_supl2}
\|s\|\le \|s\|_{\sup}\le C_1 p^{n}\|s\|
\end{align}
for all sections $s\in H^0(X,L^{\otimes p})$. While the first inequality is trivial, the second has been proved by Gillet and Soulé \cite[Lemma 30]{GS92} based on ideas by Gromov.

Finally in this section, we deduce the following application of Stokes' theorem in our setting.
\begin{Lem}\label{lem_stokes}
	Let $(L,h)$, $(L_1,h_1)$ and $(L_2,h_2)$ be any hermitian line bundles. If $s_1\in H^0(X,L_1^{\otimes p_1})$ and $s_2\in H^0(X,L_2^{\otimes p_2})$ are sections, such that $$\dim (\mathrm{div}(s_1)\cap\mathrm{div}(s_2))=n-2,$$ then it holds
	\begin{align*}
	&p_2\int_X \log|s_1|c_1(L,h)^{n-1}c_1(L_2,h_2)-\int_{\mathrm{div}(s_2)}\log|s_1|c_1(L,h)^{n-1}\\
	&=p_1\int_X \log|s_2|c_1(L,h)^{n-1}c_1(L_1,h_1)-\int_{\mathrm{div}(s_1)}\log|s_2|c_1(L,h)^{n-1}.
	\end{align*}
\end{Lem}
\begin{proof}
	Note, that $c_1(L_i^{\otimes p},h_i^{\otimes p})=pc_1(L_i,h_i)$ for all $p\in\mathbb{Z}$. Thus, by the Poincaré--Lelong formula (\ref{equ_poincare-lelong}) it is enough to show that
	\begin{align}\label{equ_stokes-integrals0}
	\int_X\log|s_1| \left(\partial\overline{\partial}\log|s_2|\right)c_1(L,h)^{n-1}=\int_X \left(\partial\overline{\partial}\log|s_1|\right)\log|s_2|c_1(L,h)^{n-1}.
	\end{align}
	We consider the current $\log|s_1|\left(\overline{\partial}\log|s_2|\right)c_1(L,h)^{n-1}$. Since this current is of type $(n-1,n)$, the operator $d=\partial+\overline{\partial}$ operates like $\partial$. Thus we get
	\begin{align*}
	&d\left(\log|s_1|\left(\overline{\partial}\log|s_2|\right)c_1(L,h)^{n-1}\right)\\
	&=\left(\partial\log|s_1|\right)\left(\overline{\partial}\log|s_2|\right)c_1(L,h)^{n-1}+\log|s_1|\left(\partial\overline{\partial}\log|s_2|\right)c_1(L,h)^{n-1}.
	\end{align*}
	Note, that $c_1(L,h)$ is $d$-closed, such that also $c_1(L,h)^{n-1}$ is $d$-closed. In particular, the $(n,n-1)$-part $\partial \left(c_1(L,h)^{n-1}\right)$ of $d\left(c_1(L,h)^{n-1}\right)$ vanishes.
	Since $X$ has no boundary, we obtain by Stokes' theorem
	\begin{align}\label{equ_stokes-integrals1}
	\int_X\left(\partial\log|s_1|\right)\left(\overline{\partial}\log|s_2|\right)c_1(L,h)^{n-1}+\int_X\log|s_1|\left(\partial\overline{\partial}\log|s_2|\right)c_1(L,h)^{n-1}=0.
	\end{align}
	Applying the same argument to the $(n,n-1)$-current $\left(\partial\log|s_1|\right)\log|s_2|c_1(L,h)^{n-1}$, we obtain the identity
	\begin{align}\label{equ_stokes-integrals2}
	\int_X\left(\overline{\partial}\partial\log|s_1|\right)\log|s_2|c_1(L,h)^{n-1}-\int_X\left(\partial\log|s_1|\right)\left(\overline{\partial}\log|s_2|\right)c_1(L,h)^{n-1}=0.
	\end{align}
 	Since $\overline{\partial}\partial=-\partial\overline{\partial}$, we obtain Equation (\ref{equ_stokes-integrals0}) by summing the identities (\ref{equ_stokes-integrals1}) and (\ref{equ_stokes-integrals2}).
\end{proof}

\subsection{Distribution of Zeros of Random Sections}\label{sec_distribution-complex}
In \cite{SZ99} Shiffman and Zelditch proved an equidistribution property for the divisors of randomly chosen sections of increasing tensor powers of positive line bundles. Their result has been improved and generalized by many people like Bayraktar, Coman, Dinh, Ma, Marinescu, Nguyên, and Sibony. In this section we will recall the universality result by Bayraktar, Coman and Marinescu \cite{BCM20}, adjusted to our setup, and we will deduce some auxiliary lemmas for the next section.

Let $X$ be any smooth projective complex variety of dimension $n\ge 0$, $\omega$ a Kähler form on $X$ and $(L,h)$ an ample line bundle on $X$ equipped with a positive hermitian metric. Let us recall Condition (B) from \cite{BCM20} for sequences of probability measures on $H^0(X,L^{\otimes p})$.
\begin{Con}[B]
	Let $(p_j)_{j\in\mathbb{Z}_{\ge 1}}$ be an increasing sequence of positive integers and $\sigma=(\sigma_{p_j})_{j\in\mathbb{Z}_{\ge 1}}$ a sequence of probability measures $\sigma_{p_j}$ on $H^0(X,L^{\otimes p_j})$. We say that $\sigma$ satisfies Condition (B) if for every $j\ge 1$ there exists a constant $C_{p_j}>0$ such that
	\begin{align}\label{equ_condition-B}
		\int_{H^0(X,L^{\otimes p_j})} \left| \log|\langle s,u\rangle|\right|d\sigma_{p_j}(s)\le C_{p_j} \qquad \forall u\in H^0(X, L^{\otimes p_j}) \text{ with } \|u\|=1.
	\end{align}
	
\end{Con}
In \cite{BCM20} Bayraktar, Coman and Marinescu proved in a much broader setting that if $\lim \inf _{j\to\infty}\frac{C_{p_j}}{p_j}=0$, then there exists an increasing subsequence $(p'_j)_{j\in\mathbb{Z}_{\ge 1}}$ of $(p_j)_{j\in\mathbb{Z}_{\ge 1}}$, such that
\begin{align}\label{equ_Bayraktar-Coman-Marinescu}
\sum_{j=1}^\infty\int_{s\in H^0(X,L^{\otimes p'_j})}\frac{1}{p'_j}\int_X\left|\log |s|\right|\omega^n d\sigma_{p'_j}(s)<\infty.
\end{align}
They also showed that in this case $\sigma$-almost all sequences $(s_{p'_j})_{j\in\mathbb{Z}_{\ge1}}$ of sections $s_{p'_j}\in H^0(X,L^{\otimes p'_j})$ satisfy the equidistribution
$$\lim_{j\to \infty} \frac{1}{p'_j}([s_{p'_j}=0]-c_1(L^{\otimes p'_j},h^{\otimes p'_j}))=0$$
in the weak sense of currents on $X$ if $n\ge 1$.

Our first goal in this section is to give a more general result in our special setting using the same methods as in \cite{BCM20}. First, let us recall the definition of the Bergman kernel function. For any $k\in \mathbb{Z}_{\ge 1}$ write $d_k=\dim H^0(X,L^{\otimes k})$ and let $S_1^k,\dots,S_{d_k}^k\in H^0(X,L^{\otimes k})$ be an orthonormal basis of $H^0(X,L^{\otimes k})$. The Bergman kernel function is defined by
$$P_k(x)=\sum_{j=1}^{d_k}\left|S_j^{k}(x)\right|^2_{k}.$$
Next, we associate a vector $U^k(x)\in H^0(X,L^{\otimes k})$ with $\|U^k(x)\|=1$ to every point $x\in X$. Let $e_k$ be a local holomorphic frame of $L^{\otimes k}$ in some open neighborhood $U\subseteq X$ of $x$ and write $S_j^k=s_j^ke_k$ with $s_j^k\in \mathcal{O}_X(U)$. We set
$$u^k_j(x)=\frac{s_j^k(x)}{\sqrt{\sum_{l=1}^{d_k}|s_l^k(x)|^2}},\qquad U^k(x)=\sum_{j=1}^{d_k}u_j^k(x)S_j^k(x).$$
If $t\in H^0(X,L^{\otimes k})$ is any section and we write $t=\sum_{j=1}^k t_j S_j^k$ for some $t_j\in\mathbb{C}$, then we get
\begin{align*}
	\frac{|t(x)|}{\sqrt{P_k(x)}}=\frac{\left|\sum_{j=1}^{d_k} t_j S_j^k(x)\right|}{\sqrt{\sum_{j=1}^{d_k}\left|S_j^{k}(x)\right|^2}}=\frac{\left|\sum_{j=1}^{d_k} t_j s_j^k(x)\right|}{\sqrt{\sum_{j=1}^{d_k}\left|s_j^{k}(x)\right|^2}}=\left|\sum_{j=1}^{d_k}t_ju_j^k(x)\right|=|\langle t, U^k(x)\rangle |.
\end{align*}
Taking logarithms we get
\begin{align}\label{equ_norm-inner-product}
	\log|t(x)|=\log |\langle t, U^k(x)\rangle|+\frac{1}{2}\log P_k(x).
\end{align}
In other words, we have constructed a continuous map
$$U^k\colon X\to S^{2d_k-1},\qquad x\mapsto U^k(x),$$
where $S^{2d_k-1}\subseteq H^0(X,L^{\otimes k})$ denotes the unit sphere. 
Our first lemma generalizes Equation (\ref{equ_Bayraktar-Coman-Marinescu}) to subvarieties of $X$. 
\begin{Lem}\label{lem_integral-zero}
	Let $(p_j)_{j\in \mathbb{Z}_{\ge 1}}$ be any increasing sequence of positive integers and $\sigma=(\sigma_{p_j})_{j\in\mathbb{Z}_{\ge 1}}$ any sequence of probability measures on $H^0(X,L^{\otimes p_j})$ satisfying Condition (B) with $\lim_{j\to \infty} \frac{C_{p_j}}{p_j}=0$. Further, let $Y\subseteq X$ be any smooth projective subvariety of dimension $d_Y$. Then we have
	$$\lim_{j\to \infty}\int_{s\in H^0(X,L^{\otimes p_j})}\frac{1}{p_j}\int_Y\left|\log |s|\right|\omega^{d_Y} d\sigma_{p_j}(s)=0.$$
\end{Lem}
\begin{proof}
	First, we recall that by a Theorem of Tian \cite{Tia90}, see also \cite[Remark 3.3]{CMM17} there exists a constant $C$ not depending on $p\in\mathbb{Z}_{\ge 1}$ and $x\in X$, such that
	$$\left|\frac{P_p(x)}{p^n}\cdot\frac{\omega_x^n}{c_1(L,h)_x^n}-1\right|\le \frac{C}{p}.$$
	This implies that $\frac{1}{e}\le \frac{P_p(x)}{p^n}\cdot\frac{\omega_x^n}{c_1(L,h)_x^n}\le e$ for sufficiently large $p$. Thus,
	$$|\log P_p(x)|\le 1+\log p^n+\left|\log\frac{\omega_x^n}{c_1(L,h)_x^n}\right|$$
	for sufficiently large $p$. Since $\omega$ and $c_1(L,h)$ are positive and $X$ is compact, there is a constants $A>1$ such that
	$$\frac{1}{A}\omega\le c_1(L,h)\le A\omega.$$
	Hence, for sufficiently large $p$ we get
	$$|\log P_p(x)|\le 1+\log p^n+\log A^n.$$
		
	By Equation (\ref{equ_norm-inner-product}) we can compute for sufficiently large $j$
	\begin{align*}
		&\int_{Y}\int_{H^0(X,L^{p_j})}|\log|s(x)|| d\sigma_{p_j}\omega^{d_Y}\\
		&\le \int_{Y}\int_{H^0(X,L^{p_j})}\left(\left|\log|\langle s,U^k(x)\rangle|\right|+\frac{1}{2}|\log P_{p_j}(x)|\right) d\sigma_{p_j}\omega^{d_Y}\\
		&\le \int_Y\omega^{d_Y}\cdot\left(C_{p_j}+\frac{1}{2}(1+\log p_j^n+\log A^n)\right)<\infty.
	\end{align*}
	Thus, we can apply Tonelli's theorem to get
	\begin{align*}
		\int_{H^0(X,L^{\otimes p_j})}\int_Y\left|\log |s|\right|\omega^{d_Y} d\sigma_{p_j}&=\int_{Y}\int_{H^0(X,L^{p_j})}|\log|s|| d\sigma_{p_j}\omega^{d_Y}\\
		&\le \int_Y\omega^{d_Y}\cdot\left(C_{p_j}+\frac{1}{2}(1+\log p_j^n+\log A^n)\right).
	\end{align*}
	Finally we can compute the limit by
	\begin{align*}
		0&\le \lim_{j\to \infty}\int_{H^0(X,L^{\otimes p_j})}\frac{1}{p_j}\int_Y\left|\log |s|\right|\omega^{d_Y} d\sigma_{p_j}\\
		&\le \lim_{j\to\infty}\int_Y\omega^{d_Y}\cdot\left(\frac{C_{p_j}}{p_j}+\frac{1}{2}\left(\frac{1}{p_j}+\frac{\log p_j^n}{p_j}+\frac{\log A^n}{p_j}\right)\right)=0.
	\end{align*}
	The assertion of the lemma follows.
\end{proof}
Next, we show that the vanishing of the limit of the inner integral in Lemma \ref{lem_integral-zero} implies equidistribution of the divisors even if we only consider one fixed sequence of sections.
\begin{Lem}\label{lem_equidistribution} Assume $n\ge 1$.
	Let $(p_j)_{j\in\mathbb{Z}_{\ge 1}}$ be an increasing sequence of positive integers and $(s_{p_j})_{j\in\mathbb{Z}_{\ge 1}}$ a sequence of sections $s_{p_j}\in H^0(X, L^{\otimes p_j})$.
	\begin{enumerate}[(i)]
		\item\label{assertion1} If it holds
		$$\lim_{j\to\infty}\frac{1}{p_j}\int_X\left|\log|s_{p_j}|\right|\omega^n=0,$$
		then for every $(n-1,n-1)$ $C^0$-form $\Phi$ on $X$ we have
		$$\lim_{j\to\infty}\frac{1}{p_j}\int_{\mathrm{div}(s_{p_j})}\Phi=\int_X \Phi\wedge c_1(L,h).$$
		\item\label{assertion2} If it holds
		$$\lim_{j\to\infty}\frac{1}{p_j}\left(\log\|s_{p_j}\|_{\sup}\int_X\omega^n-\int_X\log|s_{p_j}|\omega^n\right)=0,$$
		then for every $(n-1,n-1)$ $C^0$-form $\Phi$ on $X$ we have
		$$\lim_{j\to\infty}\frac{1}{p_j}\int_{\mathrm{div}(s_{p_j})}\Phi=\int_X \Phi\wedge c_1(L,h).$$
 		\item\label{assertion3} If it holds
		$$\sum_{j=1}^\infty \frac{1}{p_j}\int_X\left|\log|s_{p_j}|\right|\omega^n<\infty,$$
		then for every $(n-1,n-1)$ $C^0$-form $\Phi$ on $X$ we have
		$$\sum_{j=1}^\infty\left|\int_X \Phi\wedge c_1(L,h)-\frac{1}{p_j}\int_{\mathrm{div}(s_{p_j})}\Phi\right|<\infty.$$
	\end{enumerate}
\end{Lem}
\begin{proof}
	As every $C^0$-form on $X$ can be approximated from above and from below by a $C^{\infty}$-form, we may assume, that $\Phi$ is $C^{\infty}$. Similarly as in the proof of Lemma \ref{lem_stokes} we obtain by the Poincaré--Lelong formula and by Stokes' theorem
	\begin{align}\label{equ_equidistribution}
		&\int_X \Phi\wedge c_1(L,h)-\frac{1}{p_j}\int_{\mathrm{div}(s_{p_j})}\Phi=\frac{1}{p_j}\left(\int_X \Phi\wedge c_1(L^{\otimes p_j},h^{\otimes p_j})-\int_{\mathrm{div}(s_{p_j})}\Phi\right)\\
		&=\frac{1}{p_j}\int_X\Phi\wedge \frac{\partial\overline{\partial}}{\pi i}\log|s_{p_j}|=\frac{1}{p_j}\int_X\log|s_{p_j}|\frac{\partial\overline{\partial}}{\pi i}\Phi.\nonumber
	\end{align}
	Since $X$ is compact and $\omega$ is positive, there is an $A\in\mathbb{R}$ such that $\left|\frac{\partial\overline{\partial}}{\pi i}\Phi\right|\le A\omega^n$. Thus, in (\ref{assertion1}) we have
	$$0\le\lim_{j\to\infty}\left|\frac{1}{p_j}\int_X\log|s_{p_j}|\frac{\partial\overline{\partial}}{\pi i}\Phi\right|\le A\cdot\lim_{j\to\infty}\frac{1}{p_j}\int_X\left|\log|s_{p_j}|\right|\omega^n=0.$$
	This means, that the value in Equation (\ref{equ_equidistribution}) tends to $0$ for $j\to\infty$. This proves (\ref{assertion1}). 
	Part (\ref{assertion2}) follows from (\ref{assertion1}) by replacing $s_{p_j}$ by $\frac{s_{p_j}}{\|s_{p_j}\|_{\sup}}$ as
	$$\log\|s_{p_j}\|_{\sup}\int_X\omega^n-\int_X\log|s_{p_j}|\omega^n=\int_X\left|\log\left|\frac{s_{p_j}}{\|s_{p_j}\|_{\sup}}\right|\right|\omega^n.$$
	To prove (\ref{assertion3}), we sum Equation (\ref{equ_equidistribution}) over $j$ to obtain
	$$\sum_{j=1}^\infty\left|\int_X \Phi\wedge c_1(L,h)-\frac{1}{p_j}\int_{\mathrm{div}(s_{p_j})}\Phi\right|\le A\cdot\sum_{j=1}^\infty\frac{1}{p_j} \int_{X}\left|\log|s_{p_j}|\right|\omega^n<\infty.$$
\end{proof}
In the following we restrict to special types of probability measures which are obtained by normalizations of restrictions of the Haar measure. We make the following definition.
\begin{Def}
	Let $\sigma_k$ be a probability measure on $H^0(X,L^{\otimes k})$.
	\begin{enumerate}[(i)]
		\item We say that $\sigma_k$ is of type $\mathbf{L}_{\mathbb{C}}$ if there is a compact, symmetric and convex subset $K\subseteq H^0(X,L^{\otimes k})$ with non-empty interior such that
		$$\sigma_k=\frac{1}{\mathrm{Vol}(K)}\lambda|_K$$
		for the Haar measure $\lambda$ on $H^0(X,L^{\otimes k})$.
		\item Let $V\subset H^0(X,L^{\otimes k})$ be a real vector subspace spanning $H^0(X,L^{\otimes k})$ as a complex vector space and such that $\langle v,w\rangle\in\mathbb{R}$ for all $v,w\in V$. We equip $V$ with the induced euclidean structure.
		We say that $\sigma_k$ is of type $\mathbf{L}_{\mathbb{R}}$ if there is a compact, symmetric and convex subset $K\subseteq V$ with non-empty interior such that
		$$\sigma_k=\frac{1}{\mathrm{Vol}(K)}\lambda|_K$$
		for the Haar measure $\lambda$ on $V$.
	\end{enumerate}
	We say that $\sigma_k$ is of type $\mathbf{L}$ if it is of type $\mathbf{L}_\mathbb{C}$ or $\mathbf{L}_{\mathbb{R}}$.
\end{Def}
To any probability measure $\sigma_k$ on $H^0(X,L^{\otimes})$ we associate the function
$$F_{\sigma_k}(x)\colon X\to \mathbb{R}, \qquad x\mapsto\int_{H^0(X,L^{\otimes k})}\log |t(x)|d\sigma_k(t)$$
and also the map
$$\widetilde{F}_{\sigma_k}\colon S^{2d_k-1}\to \mathbb{R},\qquad u\mapsto \int_{H^0(X,L^{\otimes k})}\log|\langle t,u\rangle|d\sigma_k(t)$$
Let us check, that these maps are continuous if $\sigma_k$ is of type $\mathbf{L}$.
\begin{Lem}\label{lem_bergman}
	Let $\sigma_k$ be a probability measure on $H^0(X,L^{\otimes k})$ of type $\mathbf{L}$. Then $F_{\sigma_k}$ and $\widetilde{F}_{\sigma_k}$ are continuous functions. Moreover, $\sigma_k$ satisfies (\ref{equ_condition-B}) for some $C_k$.
\end{Lem}
\begin{proof}
	Since $F_{\sigma_k}=\widetilde{F}_{\sigma_k}\circ U^k+\frac{1}{2}\log P_k$ by Equation (\ref{equ_norm-inner-product}), it is enough to prove that $\widetilde{F}_{\sigma_k}$ is continuous.
	If $\sigma_k$ is of type $\mathbf{L}_{\mathbb{C}}$ we also write $V=H^0(X,L^{\otimes k})$ as in the real case for the real subspace $V$.
	It is enough to show $\lim_{j\to\infty}\widetilde{F}_{\sigma_k}(u_j)=\widetilde{F}_{\sigma_k}(u)$ if $\lim_{j\to \infty}u_j=u$. Note that $\log|\langle t,u_j\rangle|$ is continuous  for  $t\in \{u_j\neq 0\}$. In particular, it is continuous on a dense open subset of $V$ as $V$ generates $H^0(X,L^{\otimes k})$ over the complex numbers. Hence, by the Vitali convergence theorem it is enough to show that 
	$$\lim_{N\to\infty}\sup_{u\in S^{2d_k-1}}\int_{\{\log|\langle t,u\rangle|<-N\}}\log|\langle t,u\rangle|d\sigma_k(t)=0.$$
	Note that we only have to consider $\log|\langle t,u\rangle|<-N$ instead of $|\log|\langle t,u\rangle||>N$ as the support of $\sigma_k$ is bounded. 
	
	First, we consider the complex case. We fix an $u\in S^{2d_k-1}$. We can choose another orthonormal basis $\widetilde{S}^k_1,\dots, \widetilde{S}^k_{d_k}$ of $V$ such that $\widetilde{S}^k_1=u$. We write $t=(t_1,\dots,t_{d_k})$ for the coefficients with respect to this basis. Then $\langle t,u\rangle= t_1$. As the support $K$ of $\sigma_k$ is compact, there is a ball $B_r$ of radius $r$ such that $K\subseteq B_r$. Thus, we can compute for $N\ge 0$ by a coordinate change $\tau=|t_1|$
	\begin{align*}
		0&\ge \int_{\{\log|\langle t,u\rangle|<-N\}}\log|\langle t,u\rangle|d\sigma_k(t)\ge 2\pi\int_0^{e^{-N}}\log\tau d\tau\cdot \frac{(2r)^{2(d_k-1)}}{\mathrm{Vol}(K)}\\
		&=2\pi e^{-N}(-N-1)\frac{(2r)^{2(d_k-1)}}{\mathrm{Vol}(K)}
	\end{align*}
	The last term goes to $0$ for $N\to \infty$ and is independent of $u$. Thus, $\widetilde{F}_{\sigma_k}$ is continuous.
	
	Now we consider the real case. We fix again an $u\in S^{2d_k-1}$. As $V$ generates $H^0(X,L^{\otimes k})$ over $\mathbb{C}$ there are two vectors $u_1,u_2\in V$ with $u=u_1+iu_2$. As $$1=\|u\|^2=\|u_1\|^2+\|u_2\|^2,$$
	we have $\max\{\|u_1\|,\|u_2\|\}\ge \frac{1}{\sqrt{2}}$. We assume $\|u_1\|\ge  \frac{1}{\sqrt{2}}$, the other case works in the same way. As $\langle t,u_1\rangle$ and $\langle t, u_2\rangle$ are real-valued, we get
	$$\log|\langle t,u\rangle|=\log \sqrt{|\langle t,u_1\rangle|^2+|\langle t,u_2\rangle|^2}\ge\log|\langle t,u_1\rangle|$$
	By this, we can bound for $N\ge 0$
	\begin{align*}
		0&\ge \int_{\{\log|\langle t,u\rangle|<-N\}}\log|\langle t,u\rangle|d\sigma_k(t)\ge\int_{\{|t,u_1|<e^{-N}\}}\log|\langle t,u_1\rangle|d\sigma_k(t)
	\end{align*}	
	As in the complex case, we choose an orthonormal basis $v_1,\dots,v_d$ of $V$ with $u_1=v_1$ and write $t=(t_1,\dots,t_d)$ with respect to this basis. Then with $\tau=|t_1|$
	$$\int_{\{|t,u_1|<e^{-N}\}}\log|\langle t,u_1\rangle|d\sigma_k(t)\ge 2\int_0^{e^{-N}}\log\tau d\tau\cdot \frac{(2r)^{d-1}}{\mathrm{Vol}(K)}=e^{-N}(-N-1)\frac{2(2r)^{d-1}}{\mathrm{Vol}(K)}.$$
	The continuity of $\widetilde{F}_{\sigma_k}$ follows in the same way as in the complex case.
	
	Now we show the second assertion of the lemma. Let $K$ denote the support of $\sigma_k$ and choose an $r>0$ such that $K\subseteq B_r$. We write $\sigma'_k$ for the probability measure of type $\mathbf{L}$ associated to the support $r^{-1}K$. Then we get
	\begin{align*}
		&\int_{H^0(X,L^{\otimes k})}\left|\log |\langle s,u\rangle|\right| d\sigma_k(s)=\frac{1}{\mathrm{Vol}(K)}\int_{K}\left|\log |\langle s,u\rangle|\right| d\lambda(s)\\
		&\le \frac{1}{\mathrm{Vol}(K)}\int_{K}\left|\log |\langle r^{-1}s,u\rangle|\right| d\lambda(s)+|\log r|\\
		&=\frac{1}{\mathrm{Vol}(r^{-1}K)}\int_{r^{-1}K}\log|\langle s,u\rangle| d\lambda(s)+|\log r|=\widetilde{F}_{\sigma'_k}(u)+|\log r|
	\end{align*}
	Since $\widetilde{F}_{\sigma'_k}$ is continuous, it attains its maximum on the compact set $S^{2d_k-1}$. Thus, the second assertion follows with $C_k=\max_{u\in S^{2d_k-1}}\widetilde{F}_{\sigma'_k}(u)+|\log r|$.
\end{proof}
We now prove that the condition in Lemma \ref{lem_equidistribution} (iii) stays true with probability $1$ if we reduce to the divisor of a randomly chosen section. We only prove it under the assumptions $\|s_p\|_{\sup}\le 1$, as this is enough for our applications and it simplifies the proof. But we also allow to take the integral over a subvariety.
\begin{Lem}\label{lem_restricted-integral}
	Let $k\ge 1$ be such that $L^{\otimes k}$ is very ample and let $\sigma_k$ be a probability measure on $H^0(X,L^{\otimes k})$ of type $\mathbf{L}$. Let $Y\subseteq X$ be a smooth projective subvariety of dimension $d_Y$, $(p_j)_{j\in\mathbb{Z}_{\ge 1}}$ be any increasing sequence of positive integers and $(s_{p_j})_{j\in\mathbb{Z}_{\ge1}}$ be a sequence of sections $s_{p_j}$ in $H^0(X,L^{\otimes p_j})$ with $\|s_{p_j}\|_{\sup}\le 1$ for all $j$ and $\sum_{j=1}^\infty\frac{1}{p_j}\int_Y\left|\log|s_{p_j}|\right|c_1(L,h)^{d_Y}<\infty$.
	Then $\sigma_k$-almost all $t\in H^0(X,L^{\otimes k})$ satisfy
	$$\sum_{j=1}^\infty\frac{1}{p_j}\int_{Y\cap\mathrm{div}(t)}\left|\log|s_{p_j}|\right|c_1(L,h)^{d_Y-1}<\infty$$
	and hence in particular,
	$\lim_{j\to\infty}\frac{1}{p_j}\int_{Y\cap \mathrm{div}(t)}\left|\log|s_{p_j}|\right|c_1(L,h)^{d_Y-1}=0.$
\end{Lem}
\begin{proof}
	By $\|s_{p_j}\|_{\sup}\le 1$ we always have $\left|\log|s_{p_j}|\right|=-\log|s_{p_j}|$, such that we do not have to care about taking absolute values.
	Let $t\in H^0(X,L^{\otimes k})$. 
	We may assume that $t|_Y\neq 0$, as the set of $t|_Y=0$ is a proper vector subspace of $H^0(X,L^{\otimes k})$. By the same reason we may assume $t|_{\mathrm{div}(s_{p_j})}\neq 0$.
	By Lemma \ref{lem_stokes} we have
	\begin{align}\label{equ_stokes_st}
		&\frac{k}{p_j}\int_Y\log |s_{p_j}|c_1(L,h)^n-\frac{1}{p_j}\int_{Y\cap\mathrm{div}(t)}\log|s_{p_j}| c_1(L,h)^{n-1}\\
		&=\int_Y\log|t| c_1(L,h)^n-\frac{1}{p_j}\int_{Y\cap\mathrm{div}(s_{p_j})}\log|t|c_1(L,h)^{n-1}.\nonumber
	\end{align}
	As we are only interested in a $\sigma_k$-almost sure assertion, we may integrate this equation with respect to $\sigma_k$. First, we consider the integrals of $\log|t|$.
		If $Z\subseteq X$ is any subvariety, then Equation (\ref{equ_norm-inner-product}) and Lemma \ref{lem_bergman} imply
	\begin{align*}
	&\int_{Z} \int_{H^0(X,L^{\otimes k})}\left|\log|t(x)|\right|d\sigma_k(t) c_1(L,h)^{\dim Z}\\
	&\le \int_Z \int_{H^0(X,L^{\otimes k})}\left(\left|\log|\langle t,U^k(x)\rangle|\right|+\tfrac{1}{2}|\log P_k(x)|\right)d\sigma_k(t) c_1(L,h)^{\dim Z}\\
	&\le \left(C_k+\tfrac{1}{2}\max_{x\in Z}|\log P_k(x)|\right)\cdot \int_Z c_1(L,h)^{\dim Z} <\infty
	\end{align*}
	Hence, we can apply Tonelli's theorem to obtain
	\begin{align*}
	\int_{H^0(X,L^{\otimes k})}\int_{Z} \log|t| c_1(L,h)^{\dim Z}d\sigma_k=\int_{Z} F_{\sigma_k}(x) c_1(L,h)^{\dim Z}.
	\end{align*}
	We will apply this to the case where $Z$ is either $Y$ or $Y\cap \mathrm{div}(s_{p_j})$. As $F_{\sigma_k}$ is continuous by Lemma \ref{lem_bergman}, we can apply Lemma \ref{lem_equidistribution} (iii) on $Y$ to obtain
	\begin{align*}
		&\left|\sum_{j=1}^\infty\int_{H^0(X,L^{\otimes k})}\left(\frac{1}{p_j}\int_{Y\cap\mathrm{div}(s_{p_j})} \log|t| c_1(L,h)^{d_Y-1}-\int_Y \log|t|c_1(L,h)^{d_Y}\right)\sigma_k\right|\\
		&=\left|\sum_{j=1}^\infty\left(\frac{1}{p_j}\int_{Y\cap\mathrm{div}(s_{p_j})} F_{\sigma_k}(x) c_1(L,h)^{d_Y-1}-\int_Y F_{\sigma_k}(x)c_1(L,h)^{d_Y}\right)\right|\\
		&\le\sum_{j=1}^\infty\left|\frac{1}{p_j}\int_{Y\cap\mathrm{div}(s_{p_j})} F_{\sigma_k}(x) c_1(L,h)^{d_Y-1}-\int_Y F_{\sigma_k}(x)c_1(L,h)^{d_Y}\right|<\infty.
	\end{align*}
	Thus, if we integrate Equation (\ref{equ_stokes_st}) with respect to $\sigma_k$ and if we sum over $j$, the right hand side stays finite. The first term of the left hand side in Equation (\ref{equ_stokes_st}) does not depend on $t$ and also stays finite after summing over $j$ by the assumptions. Hence, we obtain for the remaining term
	\begin{align}\label{equ_integral0}
	\left|\sum_{j=1}^\infty \int_{H^0(X,L^{\otimes k})}\left(\frac{1}{p_j}\int_{Y\cap\mathrm{div}(t)}\log|s_{p_j}|c_1(L,h)^{d_Y-1}\right)d\sigma_k\right|<\infty.
	\end{align}
	By the assumption $\|s_{p_j}\|_{\sup}\le 1$, this is equivalent to 
	$$\sum_{j=1}^\infty \int_{H^0(X,L^{\otimes k})}\left(\frac{1}{p_j}\int_{Y\cap\mathrm{div}(t)}\left|\log|s_{p_j}|\right|c_1(L,h)^{d_Y-1}\right)d\sigma_k<\infty.$$
	Thus, we can again apply Tonelli's theorem to obtain
	$$\int_{H^0(X,L^{\otimes k})}\sum_{j=1}^\infty\left(\frac{1}{p_j}\int_{Y\cap\mathrm{div}(t)}\left|\log|s_{p_j}|\right|c_1(L,h)^{d_Y-1}\right)d\sigma_k<\infty.$$
	But this implies, that for $\sigma_k$-almost all $t\in H^0(X,L^{\otimes k})$ we have
	$$\sum_{j=1}^\infty\frac{1}{p_j}\int_{Y\cap\mathrm{div}(t)}\left|\log|s_{p_j}|\right|c_1(L,h)^{d_Y-1}<\infty$$
	and hence in particular,
	$\lim_{j\to\infty}\frac{1}{p_j}\int_{Y\cap\mathrm{div}(t)}\left|\log|s_{p_j}|\right|c_1(L,h)^{d_Y-1}=0.$
\end{proof}

\subsection{Integrals on Complex Varieties}\label{sec_integral-higher-dim}
In this section we show that the condition in Lemma \ref{lem_equidistribution} (i) stays true after a small change of the sequence. The proof works by induction on the dimension of $X$. The main result of this section is the following proposition.
\begin{Pro}\label{pro_integral}
	Let $X$ be a smooth projective complex variety with a Kähler form $\omega$ and $(L,h)$ an ample line bundle on $X$ equipped with a positive hermitian metric.
	Further, let $Y\subseteq X$ be any smooth projective subvariety of dimension $d_Y\ge 0$.
	Let $(p_j)_{j\in\mathbb{Z}_{\ge1}}$ be any increasing sequence of positive integers and $(s_{p_j})_{j\in\mathbb{Z}_{\ge1}}$ and $(s'_{p_j})_{j\in\mathbb{Z}_{\ge1}}$ two sequences of sections in $H^0(X,L^{\otimes p_j})$.
	Assume that we have
	\begin{enumerate}[(i)]
		\item $\limsup_{j\to\infty} \|s_{p_j}\|^{1/{p_j}}\le 1$,
		\item $\limsup_{j\to \infty} \|s_{p_j}-s'_{p_j}\|^{1/p_j}<1$,
		\item $\lim_{j\to\infty}\frac{1}{p_j}\int_Y\left|\log |s_{p_j}|\right|\omega^{d_Y}=0.$
	\end{enumerate}
	Then it also holds $$\lim_{j\to\infty}\frac{1}{p_j}\int_Y\left|\log |s'_{p_j}|\right|\omega^{d_Y}=0.$$
\end{Pro}
We will deduce the proposition from a similar result in the case $Y=X$, which we will prove by induction on the dimension $n=\dim X$ of $X$. Let us first show a simpler analogue result for the $\sup$-norm instead of the integral of the logarithm. In particular, this will provide the base case $n=0$ for the induction.
\begin{Lem}\label{lem_supnorm}
	Let $X$ be a smooth projective complex variety of dimension $n\ge 0$ and $(L,h)$ a line bundle on $X$ equipped with a hermitian metric. Let $(p_j)_{j\in\mathbb{Z}_{\ge1}}$ be any increasing sequence of positive integers. If $(s_{p_j})_{j\in\mathbb{Z}_{\ge1}}$ and $(s'_{p_j})_{j\in\mathbb{Z}_{\ge1}}$ are two sequences of sections in $H^0(X,L^{\otimes p_j})$ satisfying
	\begin{enumerate}[(i)]
		\item $\lim_{j\to\infty}\|s_{p_j}\|_{\sup}^{1/p_j}=c$,
		\item $\limsup_{j\to \infty} \|s_{p_j}-s'_{p_j}\|_{\sup}^{1/p_j}<c$,
	\end{enumerate}
	for some $c\in [0,\infty]$, then it also holds $\lim_{j\to\infty}\|s'_{p_j}\|_{\sup}^{1/p_j}=c$.
\end{Lem}
\begin{proof}
	As $\limsup_{j\to \infty} \|s_{p_j}-s'_{p_j}\|_{\sup}^{1/p_j}<\lim_{j\to\infty} \|s_{p_j}\|_{\sup}^{1/p_j}$, there exists some real number $\sigma\in (0,1)$ and an integer $N\in\mathbb{Z}_{\ge 1}$ such that
	$$\|s_{p_j}-s'_{p_j}\|_{\sup}^{1/p_j}<\sigma\cdot \|s_{p_j}\|_{\sup}^{1/p_j}$$
	for all $j\ge N$. Thus, we get by the triangle inequality
	\begin{align*}\limsup_{j\to\infty}\|s'_{p_j}\|_{\sup}^{1/p_j}&\le\limsup_{j\to\infty} \left(\|s_{p_j}\|_{\sup}+\|s'_{p_j}-s_{p_j}\|_{\sup}\right)^{1/p_j}\\
	&\le\limsup_{j\to\infty} \left(2\cdot\|s_{p_j}\|_{\sup}\right)^{1/p_j}=c
	\end{align*}
	and also, 
	\begin{align*}
		\liminf_{j\to \infty}\|s'_{p_j}\|_{\sup}^{1/p_j}&\ge \liminf_{j\to\infty}\left(\|s_{p_j}\|_{\sup}-\|s'_{p_j}-s_{p_j}\|_{\sup}\right)^{1/p_j}\\
		&\ge\liminf_{j\to\infty} \left((1-\sigma^{p_j})\cdot\|s_{p_j}\|_{\sup}\right)^{1/p_j}=c.
	\end{align*}
	We conclude $\lim_{j\to \infty}\|s'_{p_j}\|_{\sup}^{1/p_j}=c$.
\end{proof}
To make the application of the induction hypothesis possible we need a suitable hypersurface on $X$. The existence of such a hypersurface is guaranteed by the next lemma.
\begin{Lem}\label{lem_divisor}
	Let $X$ be a smooth projective complex variety of dimension $n\ge 1$ and $(L,h)$ an ample line bundle on $X$ equipped with a positive hermitian metric.
	Let $(p_j)_{j\in\mathbb{Z}_{\ge 1}}$ be an increasing sequence  of integers $p_j\in\mathbb{Z}_{\ge 1}$ and $(s_{p_j})_{j\in\mathbb{Z}_{\ge1}}$, $(s'_{p_j})_{j\in\mathbb{Z}_{\ge1}}$ sequences of sections $s_{p_j},s'_{p_j}\in H^0(X,L^{\otimes p_j})$ such that $\|s_{p_j}\|_{\sup}\le 1$ for all $j$ and $\sum_{j=1}^\infty\frac{1}{p_j}\int_X\left|\log|s_{p_j}|\right|c_1(L,h)^n<\infty$. For every $\epsilon>0$ there exist a positive integer $k\in\mathbb{Z}$ and a section $t\in H^0(X,L^{\otimes k})$ satisfying:
	\begin{enumerate}[(i)]
		\item\label{firstcondition} $\mathrm{div}(t)$ is smooth,
		\item\label{firstandhalfcondition} $\dim(\mathrm{div}(t)\cap\mathrm{div}(s'_{p_j}))=n-2$ for all $j$,
		\item\label{secondcondition} $\lim_{j\to \infty}\frac{1}{p_j}\int_{\mathrm{div}(t)}\left|\log|s_{p_j}|\right|c_1(L,h)^{n-1}=0$,
		\item\label{thirdcondition} $\frac{1}{k}\int_X\log|t|c_1(L,h)^{n}-\frac{1}{kp_j}\int_{\mathrm{div}(s'_{p_j})}\log|t|c_1(L,h)^{n-1}>-\epsilon$ for all $j$.
	\end{enumerate}
\end{Lem}
\begin{proof}
	We consider the $2d_k-1$-dimensional unit sphere
	$$U'_k=\{t\in H^0(X,L^{\otimes k})~|~\|t\|=1\}\subseteq H^0(X,L^{\otimes k}),$$
	the surface measure $\mathcal{A}'_k$ on $U'_k$ and the probability measure $\sigma'_k=\frac{1}{\mathcal{A}'_k(U'_k)}\mathcal{A}'_k$ on $U'_k$. By \cite[Lemma 4.11]{BCM20} there exists a constant $M_1>0$ independent of $k$, such that $\sigma_k$ satisfies Condition (B) with $C'_k=M_1\cdot \log d_k$. We recall that by the Hilbert--Serre theorem there exists a constant $M_2$ such that $d_k\le M_2 k^n$. Thus, we may replace $C'_k$ by $C'_k=M_1\cdot \log (M_2 k^n)$. We modify the radius of the sphere defining
	$$U_k=\left\{t\in H^0(X,L^{\otimes k})~|~\|t\|=\tfrac{1}{C_1k^n}\right\}\subseteq H^0(X,L^{\otimes k}),$$
	where $C_1$ denotes the constant in (\ref{equ_supl2}). In particular we have $\|t\|_{\sup}\le 1$ for all $t\in U_k$ by the inequality in (\ref{equ_supl2}).
	We define the probability measure $\sigma_k$ on $U_k$ by setting $\sigma_k(x)=\sigma'_k(C_1k^n\cdot x)$.
	
	Now we check Condition (B) for $\sigma_k$. By the triangle inequality we have
	\begin{align*}
		\int_{H^0(X,L^{\otimes k})}\left|\log|\langle t,u\rangle|\right|d\sigma_k(t)&\le \int_{H^0(X,L^{\otimes k})}\left|\log|\langle C_1k^n\cdot t,u\rangle|\right|d\sigma_k(t)+\left|\log (C_1k^n)\right|\\
		&=\int_{H^0(X,L^{\otimes k})}\left|\log|\langle t,u\rangle|\right|d\sigma'_k(t)+\log (C_1k^n)\\
		&\le M_1\cdot \log (M_2k^n)+\log (C_1k^n).
	\end{align*}
	Thus, $\sigma_k$ satisfies Condition (B) for $C_k=M_1\cdot \log (M_2k^n)+\log (C_1k^n)$. In particular, we have $\lim_{k\to \infty}\frac{C_k}{k}=0$, such that 
	$$\lim_{k\to \infty}\int_{H^0(X,L^{\otimes k})}\left(\frac{1}{k}\int_X\log|t|c_1(L,h)^{n}\right)d\sigma_k(t)=0$$
	by Lemma \ref{lem_integral-zero}. Note, that we always have $|\log|t||=-\log|t|$ since $\|t\|_{\sup}\le 1$.
	We choose $k\in\mathbb{Z}$, such that $L^{\otimes k}$ is very ample and that
	$$\int_{H^0(X,L^{\otimes k})}\left(\frac{1}{k}\int_X\log|t|c_1(L,h)^{n}\right)d\sigma_k(t)> -\frac{\epsilon}{2}.$$
	Then there exists a subset $A\subseteq U_k$ with $\sigma_k(A)>0$ such that
	$$\frac{1}{k}\int_X\log|t|c_1(L,h)^{n}> -\epsilon $$
	for all $t\in A$. Since $\|t\|_{\sup}\le 1$, this also implies
	$$\frac{1}{k}\int_X\log|t|c_1(L,h)^{n}-\frac{1}{kp}\int_{\mathrm{div}(s)}\log|t|c_1(L,h)^{n-1}>-\epsilon$$
	for all $t\in A$ and all sections $s\in H^0(X,L^{\otimes p})$ for some $p\in\mathbb{Z}$ satisfying $$\dim(\mathrm{div}(t)\cap\mathrm{div}(s))=n-2.$$
	Thus, every $t\in A$ satisfies (\ref{thirdcondition}) if it satisfies (\ref{firstandhalfcondition}).
	
	The conditions (\ref{firstcondition}), (\ref{firstandhalfcondition}) and (\ref{secondcondition}) are satisfied for $\sigma_k$-almost all $t\in U_k$. For (\ref{firstcondition}) this follows since the set of $t\in U_k$, such that $\mathrm{div}(t)$ is smooth, is dense and open in $U_k$ by Bertini's theorem. Also the set of sections $t\in U_k$ such that $\dim(\mathrm{div}(t)\cap \mathrm{div}(s'_{p_j}))=n-2$ is dense and open for every $j$. Hence, (\ref{firstandhalfcondition}) holds for $\sigma_k$-almost all $t\in U_k$ for any fixed $j$. As the union of countable many sets of measure $0$ has again measure $0$, (\ref{firstandhalfcondition}) also holds for $\sigma_k$-almost all $t\in U_k$ simultaneously for all $j$.
	It follows from Lemma \ref{lem_restricted-integral} that also (\ref{secondcondition}) holds for $\sigma_k$-almost all $t\in U_k$. 
	Indeed, one can apply Lemma \ref{lem_restricted-integral} to the closed ball corresponding to the sphere $U_k$ and use that the assertion of (\ref{secondcondition}) is independent of scaling $t$ by any non-zero complex number.
	Since $A$ has positive measure with respect to $\sigma_k$, we conclude that there is always a $t\in A$ satisfying (\ref{firstcondition}), (\ref{firstandhalfcondition}), (\ref{secondcondition}), and (\ref{thirdcondition}).
\end{proof}
Instead of directly proving Proposition \ref{pro_integral} by induction, we instead prove the following lemma by induction on $n$ and we will deduce Proposition \ref{pro_integral} from this lemma.
\begin{Lem}\label{lem_dimgen}
	Let $X$ be a smooth projective complex variety of dimension $n\ge 0$, $(L,h)$ an ample line bundle on $X$ equipped with a positive hermitian metric and $\omega=c_1(L,h)$ its first Chern form.
	Let $(p_j)_{j\in\mathbb{Z}_{\ge1}}$ be any increasing sequence of positive integers.	
	If $(s_{p_j})_{j\in\mathbb{Z}_{\ge1}}$ and $(s'_{p_j})_{j\in\mathbb{Z}_{\ge1}}$ are two sequences of sections in $H^0(X,L^{\otimes p_j})$ satisfying
	\begin{enumerate}[(i)]
		\item $\limsup_{j\to\infty}\frac{1}{p_j}\log\|s_{p_j}\|_{\sup}\int_X\omega^n\le \liminf_{j\to\infty}\frac{1}{p_j}\int_X\log|s_{p_j}|\omega^n<\infty$,
		\item $\limsup_{j\to\infty}\frac{1}{p_j}\log\|s_{p_j}-s'_{p_j}\|_{\sup}\int_X\omega^n<\liminf_{j\to\infty}\frac{1}{p_j}\int_X\log|s_{p_j}|\omega^n$,
	\end{enumerate}
	then we have
	$$\lim_{j\to\infty}\frac{1}{p_j}\int_X\log |s'_{p_j}|\omega^n=\lim_{j\to \infty}\frac{1}{p_j}\log\|s'_{p_j}\|_{\sup}\int_{X}\omega^n=\lim_{j\to \infty}\frac{1}{p_j}\log\|s_{p_j}\|_{\sup}\int_{X}\omega^n.$$
\end{Lem}
\begin{proof}
	First, note that by the trivial bound 
	$$\int_X\log|s_{p_j}|\omega^n\le \log\|s_{p_j}\|_{\sup}\int_X \omega^n,$$
	condition (i) is equivalent to
	\begin{align}\label{equ_limit-sup}
	\lim_{j\to\infty}\frac{1}{p_j}\log\|s_{p_j}\|_{\sup}\int_X\omega^n= \lim_{j\to\infty}\frac{1}{p_j}\int_X\log|s_{p_j}|\omega^n<\infty.
	\end{align}
	In particular, both limits are well-defined. Thus, the second equality in the consequence of the lemma directly follows from Lemma \ref{lem_supnorm}.
	
	By condition (ii) the value of (\ref{equ_limit-sup}) is also strictly bigger than $-\infty$. Thus, we may replace the sections $s_{p_j}$ by $\frac{s_{p_j}}{\|s_{p_j}\|_{\sup}}$ and the sections $s'_{p_j}$ by $\frac{s'_{p_j}}{\|s_{p_j}\|_{\sup}}$ to reduce to the case where $\|s_{p_j}\|_{\sup}=1$ for all $j\ge 1$. After this reduction, we have to show that
	$$\lim_{j\to\infty}\frac{1}{p_j}\int_X\log|s'_{p_j}|\omega^n=0.$$
		
	We want to apply induction on the dimension $n$ of $X$. If $n=0$ the integral is just a sum over the points of $X$.
	Thus, Equation (\ref{equ_limit-sup}) has the form
	$$\lim_{j\to\infty}\sum_{x\in X}\frac{1}{p_j}\log|s_{p_j}(x)|=0,$$
	Note, that every summand is non-positive by the assumption $\|s_{p_j}\|_{\sup}=1$. Hence, we get that
	$$\lim_{j\to\infty}\frac{1}{p_j}\log|s_{p_j}(x)|=0$$
	for every point $x\in X$. By condition (ii) we also have
	$$\limsup_{j\to \infty} \frac{1}{p_j}\log|(s_{p_j}-s'_{p_j})(x)|\le\limsup_{j\to \infty} \frac{1}{p_j}\log\|s_{p_j}-s'_{p_j}\|_{\sup}<0$$
	for every point $x\in X$. Thus, we can apply Lemma \ref{lem_supnorm} to the complex variety $\{x\}$ to get 
	$$\lim_{j\to\infty}\frac{1}{p_j}\log|s'_{p_j}(x)|=0$$
	for every point $x\in X$. Taking the sum over all points $x\in X$ we get the assertion of the lemma in the case $n=0$.
	 
	Next, we consider the case $n\ge 1$ and we assume that the lemma is true for smooth projective complex varieties of dimension $n-1$. As the value of (\ref{equ_limit-sup}) is $0$ by the assumption $\|s_{p_j}\|_{\sup}=1$, we get from Lemma \ref{lem_supnorm} that
	$$\limsup_{j\to \infty}\frac{1}{p_j}\int_{X}\log|s'_{p_j}|\omega^n\le \limsup_{j\to\infty}\frac{1}{p_j}\log\|s'_{p_j}\|_{\sup}\int_{X}\omega^n =0.$$
	Thus, $\frac{1}{p_j}\int_X\log |s'_{p_j}|\omega^n$ has only non-positive limit points.
	Let $\rho\in[-\infty,0]$ be such a limit point and $(p'_j)_{j\in\mathbb{Z}_{\ge 1}}\subseteq (p_j)_{j\in\mathbb{Z}_{\ge 1}}$ be a subsequence such that
	$$\lim_{j\to\infty}\frac{1}{p'_j}\int_X\log |s'_{p'_j}|\omega^n=\rho\qquad \text{and} \qquad \sum_{j=1}^\infty\frac{1}{p'_j}\int_X\left|\log |s_{p'_j}|\right|\omega^n<\infty.$$
	We want to show $\rho=0$. We take any positive real number $\epsilon>0$. Let $k\in\mathbb{Z}$ and $t\in H^0(X,L^{\otimes k})$ be associated to $\epsilon$ as in Lemma \ref{lem_divisor}. By Lemma \ref{lem_stokes} we have
	\begin{align}\label{equ_applicaiton-stokes}
	&\frac{1}{p'_j}\int_X\log|s'_{p'_j}|c_1(L,h)^n-\frac{1}{kp'_j}\int_{\mathrm{div}(t)}\log|s'_{p'_j}|c_1(L,h)^{n-1}\\
	&=\frac{1}{k}\int_X\log|t|c_1(L,h)^n-\frac{1}{kp'_j}\int_{\mathrm{div}(s'_{p'_j})}\log|t|c_1(L,h)^{n-1}.\nonumber
	\end{align}
	
	By the induction hypothesis we have
	$$\lim_{j\to \infty}\frac{1}{kp'_j}\int_{\mathrm{div}(t)}\log|s'_{p'_j}|c_1(L,h)^{n-1}=\lim_{j\to \infty}\frac{1}{kp'_j}\int_{\mathrm{div}(t)}\log|s_{p'_j}|c_1(L,h)^{n-1}=0.$$
	Indeed, conditions (i) and (ii) are satisfied for $s_{p'_j}|_{\mathrm{div}(t)}$ and $s'_{p'_j}|_{\mathrm{div}(t)}$ on $\mathrm{div}(t)$ by condition (\ref{secondcondition}) in Lemma \ref{lem_divisor}.
	By condition (\ref{thirdcondition}) in Lemma \ref{lem_divisor} we have
	$$\frac{1}{k}\int_X\log|t|c_1(L,h)^n-\frac{1}{kp'_j}\int_{\mathrm{div}(s'_{p'_j})}\log|t|c_1(L,h)^{n-1}\ge -\epsilon$$
	for all $j$.
	If we apply these two observations to Equation (\ref{equ_applicaiton-stokes}) we get
	$$0\ge\rho=\lim_{j\to\infty}\frac{1}{p'_j}\int_X\log|s'_{p'_j}|c_1(L,h)^n\ge -\epsilon.$$
	Since $\epsilon>0$ was arbitrary, we obtain $\rho=0$.
	As $\rho$ was chosen as an arbitrary limit point of $\frac{1}{p_j}\int_X\log |s'_{p_j}|\omega^n$, we get that $0$ is the only limit point. This proves the lemma.
\end{proof}

Now we can give the proof of Proposition \ref{pro_integral}.
\begin{proof}[Proof of Proposition \ref{pro_integral}]
	First, we may assume $\omega=c_1(L,h)$ as $\omega$ and $c_1(L,h)$ are both positive, such that there are constants $a>0$ and $A>0$ with $$ac_1(L,h)\le \omega \le Ac_1(L,h).$$
	As we are interested in the asymptotic vanishing of an integral of an non-negative function, it does not effect the assertion if we replace $\omega$ by $c_1(L,h)$. Also the limits $\limsup_{j\to\infty}\|s_{p_j}\|^{1/p_j}$ and $\limsup_{j\to\infty}\|s_{p_j}-s'_{p_j}\|^{1/p_j}$ does not depend on the choice of $\omega$ for the definition of the $L^2$-norm as one can check by comparing them with the corresponding limits with the $\sup$-norm using inequality (\ref{equ_supl2}).
	
	We will apply Lemma \ref{lem_dimgen} to the subvariety $Y$.
	Let us check conditions (i) and (ii) of Lemma \ref{lem_dimgen}. Condition (i) is satisfied as by Equation (\ref{equ_supl2}) and assumptions (i) and (iii) in the proposition we have
	\begin{align*}
	&\limsup_{j\to\infty}\tfrac{1}{p_j}\log \|s_{p_j}|_Y\|_{\sup}\le \limsup_{j\to\infty}\tfrac{1}{p_j}\log \|s_{p_j}\|_{\sup}\\
	&\le \limsup_{j\to\infty}\tfrac{1}{p_j}\left(\log\|s_{p_j}\|+\log (C_1 p_j^n)\right)\\
	&=\limsup_{j\to\infty}\tfrac{1}{p_j}\log\|s_{p_j}\|\le 0=\lim_{j\to \infty}\tfrac{1}{p_j}\int_Y\log|s_{p_j}|\omega^{d_Y}
	\end{align*}
	Condition (ii) of Lemma \ref{lem_dimgen} follows by the same argument using assumptions (ii) and (iii) in the proposition.
	Thus, we can apply Lemma \ref{lem_dimgen} to obtain
	\begin{align}\label{equ_applicaiton-of-lemma}
	\lim_{j\to\infty}\frac{1}{p_j}\int_Y\log|s'_{p_j}|\omega^{d_Y}=\lim_{j\to\infty}\frac{1}{p_j}\log\|s'_{p_j}|_Y\|_{\sup}\int_Y\omega^{d_Y}=0.
	\end{align}

	To get the absolute value in the integral, we do the following calculation
	\begin{align*}
		&\lim_{j\to\infty}\frac{1}{p_j}\int_Y\left|\log|s'_{p_j}|\right|\omega^{d_Y}=\lim_{j\to\infty}\frac{1}{p_j}\int_Y\left|\log\left|\frac{s'_{p_j}}{\|s'_{p_j}|_Y\|_{\sup}}\right|+\log\|s'_{p_j}|_Y\|_{\sup}\right|\omega^{d_Y}\\
		&\le -\lim_{j\to\infty}\frac{1}{p_j}\int_Y\log\left|\frac{s'_{p_j}}{\|s_{p_j}\|_{\sup}}\right|\omega^{d_Y}+\lim_{j\to\infty}\frac{1}{p_j}\left|\log\|s'_{p_j}|_Y\|_{\sup}\right|\int_X\omega^{d_Y}\\
		&\le-\lim_{j\to\infty}\frac{1}{p_j}\int_Y\log|s'_{p_j}|\omega^{d_Y}+2\lim_{j\to\infty}\frac{1}{p_j}\left|\log\|s'_{p_j}|_Y\|_{\sup}\right|\int_Y\omega^{d_Y}=0,
	\end{align*}
	where the last equality follows by Equation (\ref{equ_applicaiton-of-lemma}). As the integral at the beginning of this computation is always non-negative, we conclude that
	$$\lim_{j\to\infty}\frac{1}{p_j}\int_Y\left|\log|s'_{p_j}|\right|\omega^{d_Y}=0$$
	as claimed in the proposition.
\end{proof}
\begin{Rem}
	Proposition \ref{pro_integral} shows that there are sequences $(s_p)_{p\in \mathbb{Z}_{\ge 1}}$ with $s_p\in H^0(X,L^{\otimes p})$, $\|s_p\|=1$ and irreducible divisors $\mathrm{div}(s_p)$ satisfying
	$$\lim_{p\to \infty}\frac{1}{p}\int_{X}\left|\log|s_p|\right|>0.$$
	Indeed, if we assume for the sake of simplicity that $H^0(X,L)\neq 0$ and choose a non-zero section $s\in H^0(X,L)$ normed to $\|s\|_{\sup}=1$, then the sequence $(s^{\otimes p})_{p\in\mathbb{Z}_{\ge 1}}$ satisfies
	$$\lim_{p\to \infty}\frac{1}{p}\int_{X}\left|\log|s^{\otimes p}|\right|=\int_{X}\left|\log|s|\right|>0.$$
	Of course, $\mathrm{div}(s^{\otimes p})=p\mathrm{div}(s)$ is far from being irreducible. But by Bertini's theorem the set of $t\in H^0(X,L^{\otimes p})$ with irreducible divisor $\mathrm{div}(t)$ is dense in $H^0(X,L^{\otimes p})$ for $p$ large enough. Hence, we can choose $(s_p)_{p\in \mathbb{Z}_{\ge 1}}$ with $s_p\in H^0(X,L^{\otimes p})$ and irreducible divisor $\mathrm{div}(s_p)$, such that $\limsup_{p\to \infty}\|s^{\otimes p}-s_p\|^{1/p}<1$. Now Proposition \ref{pro_integral} shows, that we have
	$$\lim_{p\to \infty}\frac{1}{p}\int_{X}\left|\log|s_p|\right|>0.$$
	If we rescale the sections $s_p$, such that $\|s_p\|=1$, the value of the limit will not change by (\ref{equ_supl2}). Indeed, we have $\|s^{\otimes p}\|_{\sup}=1$ for all $p\ge 1$ and hence by construction, $\lim_{p\to \infty} \|s_p\|_{\sup}=1$.
\end{Rem}

\subsection{Distribution on Real Vector Subspaces}\label{sec_real-space}
In this section, we show that the sequences of probability measures on $H^0(X,L^{\otimes p_j})$ of type $\mathbf{L}_{\mathbb{R}}$ satisfy Condition (B) with $\lim_{p\to\infty}\frac{C_p}{p}=0$ if their supports are bounded by balls satisfying certain conditions. In particular, we can apply Lemma \ref{lem_integral-zero} in this case. 
\begin{Pro}\label{pro_conditionB}
	Let $X$ be a smooth projective complex variety and $(L,h)$ an ample line bundle on $X$ equipped with a positive hermitian metric.
	Let $(p_j)_{j\in\mathbb{Z}_{\ge 1}}$ be an increasing sequence of positive integers and $(\sigma_{p_j})_{j\in\mathbb{Z}_{\ge 1}}$ a sequence of probability measures $\sigma_{p_j}$ on $H^0(X,L^{\otimes p_j})$ of type $\mathbf{L}_{\mathbb{R}}$. Denote $K_j$ for the support of $\sigma_{p_j}$ and $V_j$ for the real vector space associated to $\sigma_{p_j}$. Let $(r_j)_{j\in\mathbb{Z}_{\ge 1}}$ and $(r'_j)_{j\in\mathbb{Z}_{\ge 1}}$ be two sequences of positive real numbers such that $\overline{B_{r'_j}}\subseteq K_j\subseteq \overline{B_{r_j}}$ for all $j\ge 1$, where $B_{r'_j}$ and $B_{r_j}$ denote the balls in $V_j$ around the origin of radius $r'_j$ and $r_j$. If
	$$\lim_{j\to \infty}r_j^{1/p_j}=\lim_{j\to \infty}{r'}_j^{1/p_j}=1,$$
	then $(\sigma_{p_j})_{j\in\mathbb{Z}_{\ge 1}}$ satisfies Condition (B) with $\lim_{j\to \infty} \frac{C_{p_j}}{p_j}=0$.
\end{Pro}
\begin{proof}
	First, note that $\dim V_j=d_{p_j}$, where $d_{p_j}=\dim_{\mathbb{C}}H^0(X,L^{\otimes p_j})$. Indeed, if $x_1,\dots, x_d\in V_j$ denotes an orthonormal basis of $V_j$, these vectors span $H^0(X,L^{\otimes p_j})$ over $\mathbb{C}$ by the assumptions and they are also orthonormal, and hence linearly independent, in $H^0(X,L^{\otimes p_j})$.
	
	We have to show that for all $u\in H^0(X,L^{\otimes p_j})$ with $\|u\|=1$ it holds
	\begin{align}\label{equ_toshow}
		\frac{1}{\mathrm{Vol}(K_j)}\int_{K_j}\left|\log|\langle x,u\rangle|\right|d\lambda_j(x)\le C_{p_j}
	\end{align}
	for some constants $C_{p_j}$ satisfying $\lim_{j\to \infty}\frac{C_{p_j}}{p_j}=0$. By the triangle inequality we have
	\begin{align*}
		&\frac{1}{\mathrm{Vol}(K_j)}\int_{K_j}\left|\log|\langle x,u\rangle|\right|d\lambda_j(x)\\
		&\le \frac{1}{\mathrm{Vol}(K_j)}\int_{K_j}\left|\log|\langle r_j^{-1}x,u\rangle|\right|d\lambda_j(x)+\left|\log r_j\right|\\
		&=\frac{1}{\mathrm{Vol}(r_j^{-1}K_j)}\int_{r_j^{-1}K_j}\left|\log|\langle x,u\rangle|\right|d\lambda_j(x)+\left|\log r_j\right|.
	\end{align*}
	Since $\lim_{j\to\infty}r_j^{1/p_j}=1$, it holds $\lim_{j\to \infty}\frac{|\log r_j|}{p_j}=0$. Thus, it is enough to show (\ref{equ_toshow}) after replacing $K_j$ by $r_j^{-1}K_j$. In particular, we can assume that $\|x\|\le 1$ for all $x\in K_j$ for all $j\ge 1$. This implies $\left|\log|\langle x,u\rangle|\right|=-\log|\langle x,u\rangle|$ for all $x\in K_j$ and $u\in H^0(X,L^{\otimes p_j})$ with $\|u\|=1$ for all $j\ge 1$.
	
	For every $u\in H^0(X,L^{\otimes p_j})$ there exist $u_1,u_2\in V_j$ such that $u=u_1+iu_2$. If $\|u\|=1$, then $\frac{1}{\sqrt{2}}\le\max(\|u_1\|,\|u_2\|)\le 1$. We choose $k\in\{1,2\}$ such that $\frac{1}{\sqrt{2}}\le\|u_k\|\le 1$. Then
	\begin{align*}
		\frac{1}{\mathrm{Vol}(K_j)}\int_{K_j}\log|\langle x,u\rangle|d\lambda_j(x)&=\frac{1}{\mathrm{Vol}(K_j)}\int_{K_j}\log\sqrt{|\langle x,u_1\rangle|^2+|\langle x,u_2\rangle|^2}d\lambda_j(x)\\
		&\ge\frac{1}{\mathrm{Vol}(K_j)}\int_{K_j}\log|\langle x,u_k\rangle|d\lambda_j(x)\\
		&\ge \frac{1}{\mathrm{Vol}(K_j)}\int_{K_j}\log|\langle x,\|u_k\|^{-1}u_k\rangle|d\lambda_j(x)-\log \sqrt{2}.
	\end{align*}
	Since $\lim_{j\to\infty}\frac{\sqrt{2}}{p_j}=0$, it remains to prove that for all $u'\in V_j$ with $\|u'\|=1$ it holds
	$$\frac{1}{\mathrm{Vol}(K_j)}\int_{K_j}\log|\langle x,u'\rangle|d\lambda_j(x)\ge -C_{p_j}$$
	with $\lim_{j\to\infty}\frac{C_{p_j}}{p_j}=0$.
	To prove this, we choose an orthonormal basis $u'_1,\dots,u'_{d_{p_j}}\in V_j$ of $V_j$ with $u'_1=u'$ and for any $x\in V_j$ we write $x=\sum_{k=1}^{d_p} x_k u'_k$ with $x_k\in \mathbb{R}$. We write $H_{j,t}=\{x\in K_j~|~x_1=t\}$. For every $\epsilon\in (0,1]$ we have
	\begin{align}\label{equ_integral-epsilon}
		&\frac{1}{\mathrm{Vol}(K_j)}\int_{K_j}\log|\langle x,u'\rangle|d\lambda_j(x)\\
		&=\frac{1}{\mathrm{Vol}(K_j)}\int_{[-1,1]}\log|t|\left(\int_{K_j\cap H_{j,t}}dx_2\cdots dx_{d_{p_j}}\right)dt\nonumber\\
		&\ge\frac{2}{\mathrm{Vol}(K_j)}\int_{[0,\epsilon]}\log t\left(\int_{K_j\cap H_{j,t}}dx_2\cdots dx_{d_{p_j}}\right)dt+\log \epsilon\nonumber
	\end{align}
	The inner integral is equal to $\mathrm{Vol}_{d_{p_j}-1}(K_j\cap H_{j,t})$, where $\mathrm{Vol}_{d_{p_j}-1}$ denotes the volume on $H_{j,t}\cong H_{j,0}\cong\mathbb{R}^{n-1}$. We bound its value in the following claim.
	\begin{clm}
		It holds
		$$\frac{\mathrm{Vol}_{{d_{p_j}}-1}(K_j\cap H_{j,t})}{\mathrm{Vol}(K_j)}\le \frac{d_{p_j}}{2r'_j}.$$
	\end{clm}
	\begin{proof}[Proof of Claim]
		Since $\overline{B_{r'_j}}\subseteq K_j$ there exists a point $x\in K_j$ with $|x_1|=r'_j$.
		Since $K_j$ is symmetric and convex, it contains the convex body spanned by $K_j\cap H_{j,0}$ and the points $x$ and $-x$. Therefore, we can bound its volume by
		$$\mathrm{Vol}(K_j)\ge \frac{2r'_j}{d_{p_j}} \cdot \mathrm{Vol}_{d_{p_j}-1}(K_j\cap H_{j,0}).$$
		By Brunn's theorem the radius function $r(t)=\mathrm{Vol}_{d_{p_j}-1}(K_j\cap H_{j,t})^{1/(d_{p_j}-1)}$ is concave. By the symmetry of $K_j$ we have $r(t)=r(-t)$. Hence, $\mathrm{Vol}_{d_{p_j}-1}(K_j\cap H_{j,t})$ has to be maximal in $t=0$. Thus, we obtain
		$$\mathrm{Vol}(K_j)\ge \frac{2r'_j}{d_{p_j}} \cdot \mathrm{Vol}_{d_{p_j}-1}(K_j\cap H_{j,0})\ge \frac{2r'_j}{d_{p_j}} \cdot \mathrm{Vol}_{d_{p_j}-1}(K_j\cap H_{j,t})$$
		for every $t$. The claim follows after multiplying this inequality with $\frac{d_{p_j}}{2r'_j\mathrm{Vol}(K_j)}$.
	\end{proof}
	
	By the claim we can bound
	\begin{align*}
		\frac{2}{\mathrm{Vol}(K_j)}\int_{[0,\epsilon]}\log t\left(\int_{K_j\cap H_{j,t}}dx_2\cdots dx_{d_{p_j}}\right)dt\ge \frac{d_{p_j}}{r'_j}\int_{[0,\epsilon]}\log tdt=\frac{d_{p_j}\epsilon(\log\epsilon-1)}{r'_j}.
	\end{align*}
	If we set $\epsilon=\min\left\{1,\frac{r'_j}{d_{p_j}}\right\}$ and apply the above estimation to the calculation in (\ref{equ_integral-epsilon}), we get
	$$\frac{1}{\mathrm{Vol}(K_j)}\int_{K_j}\log|\langle x,u'\rangle|d\lambda_j(x)\ge \min\left\{-1,2\log r'_j-2\log d_{p_j}-1\right\}.$$
	We set $C_{p_j}=\max\left\{1,1+2\log d_{p_j}-2\log r'_j\right\}$ and we have to show $\lim_{j\to \infty}\frac{C_{p_j}}{p_j}=0$. By the Hilbert--Serre theorem there exists a constant $M\in\mathbb{Z}$ such that $d_p\le Mp^n$ for all $p$. As $L$ is ample, we have $d_p\ge 1$ if $p$ is sufficiently large. Hence, we get
	$$0\le \lim_{p\to\infty}\frac{\log d_p}{p}\le \lim_{p\to\infty}\left(\log M^{1/p}+n\log p^{1/p}\right)=0,$$
	which implies $\lim_{j\to\infty}\frac{\log d_{p_j}}{p_j}=0$. As we also have $\lim_{j\to\infty} \frac{\log r'_j}{p_j}=0$ by the assumption $\lim_{j\to\infty}{r'_j}^{1/p_j}=1$, we obtain $\lim_{j\to \infty}\frac{C_{p_j}}{p_j}=0$ as desired.
\end{proof}

\section{Arithmetic Intersection Theory}\label{sec_arithmetc-intersection-theory}
The idea of arithmetic intersection theory is to give an analogue of classical intersection theory for flat and projective schemes over $\mathrm{Spec}(\mathbb{Z})$ by adding some analytic information playing the role of a hypothetical fiber at infinity. It was invented in 1974 by Arakelov \cite{Ara74} for arithmetic surfaces and generalized to higher dimensions in 1990 by Gillet and Soul\'e \cite{GS90}. In this section we discuss some definitions and basic facts in arithmetic intersection theory. First, we clarify the notion of projective arithmetic varieties and hermitian line bundles on them in Section \ref{sec_arithmetic-varieties}. As an example, we consider the projective line $\mathbb{P}_{\mathbb{Z}}^1$ and the line bundle $\mathcal{O}(1)$ equipped with the Fubini--Study metric in Section \ref{sec_projective-line}. In Section \ref{sec_intersection} we discuss arithmetic cycles, the arithmetic Chow group and the arithmetic intersection product of arithmetic cycles with hermitian line bundles. We prove some easy lemmas about arithmetic intersection numbers on restrictions in Section \ref{sec_restriction}. This also includes a proof of Proposition \ref{pro_equidistribution-individual}. Finally, we discuss some properties of the lattice of global sections of a hermitian line bundle in Section \ref{sec_lattice}. For more details we refer to Moriwaki's book \cite{Mor14}.
\subsection{Projective Arithmetic Varieties and Hermitian Line Bundles}\label{sec_arithmetic-varieties}
We discuss hermitian line bundles on projective arithmetic varieties in this section. For more details we refer to \cite[Section 5.4]{Mor14}.
By a \emph{projective arithmetic variety} we mean an integral scheme, which is projective, separated, flat and of finite type over $\mathrm{Spec}~\mathbb{Z}$. We call a projective arithmetic variety $\mathcal{X}$ \emph{generically smooth} if its generic fiber $\mathcal{X}_{\mathbb{Q}}=\mathcal{X}\times_{\mathrm{Spec}(\mathbb{Z})}\mathrm{Spec}(\mathbb{Q})$ is smooth. We call a horizontal irreducible closed subscheme $\mathcal{Y}\subseteq \mathcal{X}$ an \emph{arithmetic subvariety}.

Let $\mathcal{X}\to\mathrm{Spec}~\mathbb{Z}$ be a projective arithmetic variety of dimension $d$. The complex points $\mathcal{X}(\mathbb{C})$ form a projective complex variety.
Let $\mathcal{L}$ be a line bundle on $\mathcal{X}$. By a \emph{hermitian metric on $\mathcal{L}$}, we mean a family $h=(h_{x})_{x\in\mathcal{X}(\mathbb{C})}$ of hermitian metrics $h_x$ on the fibers $\mathcal{L}(\mathbb{C})_{x}$. For any complex point $x\in\mathcal{X}(\mathbb{C})$ considered as a morphism $\mathrm{Spec}(\mathbb{C})\xrightarrow{x}\mathcal{X}$ we denote by $\overline{x}\in\mathcal{X}(\mathbb{C})$ the point obtained by the composition of the morphism $\mathrm{Spec}(\mathbb{C})\xrightarrow{\mathrm{conj}} \mathrm{Spec}(\mathbb{C})$ induced by the complex conjugation and the morphism $x$. We write 
$$F_\infty\colon \mathcal{X}(\mathbb{C})\to\mathcal{X}(\mathbb{C}),\qquad x\mapsto \overline{x}$$
for the complex conjugation map. We also denote by $F_\infty$ the induced anti-$\mathbb{C}$-linear map $F_\infty\colon \mathcal{L}_x\xrightarrow{\sim}\mathcal{L}_{\overline{x}}$ given by $s\otimes_x\alpha\mapsto s\otimes_{\overline{x}}\overline{\alpha}$.
Recall that a continuous function $f\colon T\to \mathbb{R}$ on an $n$-dimensional complex quasi-projective variety $T$ is called smooth if for any analytic map $\phi\colon \Delta^n\to T$ from the polydisc $\Delta^n=\{z\in\mathbb{C}^n~|~|z|<1\}$ to $T$ the composition $f\circ \phi$ is smooth on $\Delta^n$. 

We call $\overline{\mathcal{L}}=(\mathcal{L},h)$ a \emph{hermitian line bundle on $\mathcal{X}$} if $\mathcal{L}$ is a line bundle on $\mathcal{X}$, $h$ is a hermitian metric on $\mathcal{L}(\mathbb{C})$, $h_x(s(x),s(x))$ is a smooth function on $U$ for any section $s$ over an open subset $U\subseteq \mathcal{X}(\mathbb{C})$ and $h$ is of real type, that is
\begin{align}\label{equ_complexconjugation}
h_x(s, s')=\overline{h_{\overline{x}}(F_\infty(s),F_\infty(s'))}
\end{align}
for all $x\in \mathcal{X}(\mathbb{C})$ and all $s,s'\in\mathcal{L}_{x}$. We will use the notation
$$|s(x)|_{\overline{\mathcal{L}}}=\sqrt{h_x(s(x),s(x))}\qquad \text{and}\qquad \|s\|_{\sup}=\sup_{x\in\mathcal{X}(\mathbb{C})}|s(x)|_{\overline{\mathcal{L}}}$$
for any section $s\in H^0(\mathcal{X}(\mathbb{C}),\mathcal{L}(\mathbb{C}))$.
We write $\widehat{\mathrm{Pic}}(\mathcal{X})$ for the group of hermitian line bundles on $\mathcal{X}$.
We call $\overline{\mathcal{L}}$ \emph{semipositive} if for any analytic map $\phi\colon \Delta^n\to \mathcal{X}(\mathbb{C})$ the curvature form $c_1(\phi^{*}\mathcal{L}(\mathbb{C}),\phi^* h)=\frac{i\partial\overline{\partial}}{2\pi}\log \phi^*h_{y}(s(y),s(y))$ for any invertible section $s$ of $\phi^*\mathcal{L}(\mathbb{C})$ over $\Delta^n$ is semipositive. Moreover, we call $\overline{\mathcal{L}}$ \emph{positive} if for any smooth function $f\colon \mathcal{X}(\mathbb{C})\to\mathbb{R}$ with compact support there exists $\lambda_0\in\mathbb{R}_{>0}$ such that $(\mathcal{L}, \exp(-\lambda f)h)$ is semipositive for any $\lambda\in \mathbb{R}$ with $|\lambda|\le \lambda_0$.

If $\mathcal{X}(\mathbb{C})$ is smooth, then a hermitian line bundle $\overline{\mathcal{L}}$ on $\mathcal{X}$ induces an hermitian line bundle $\overline{\mathcal{L}}(\mathbb{C})$ on $\mathcal{X}(\mathbb{C})$ in the sense of Section \ref{sec_preliminaries}. The definitions of smoothness and positiveness are compatible with the definitions in Section \ref{sec_preliminaries}.

If $\overline{\mathcal{L}}$ is any hermitian line bundle on $\mathcal{X}$ and $f\colon \mathcal{Y}\to\mathcal{X}$ is a morphism of projective arithmetic varieties, then we obtain a hermitian line bundle $f^*\overline{\mathcal{L}}=(f^*\mathcal{L},f^*h)$ on $\mathcal{Y}$. If in particular $f$ is an embedding $\mathcal{Y}\subseteq \mathcal{X}$, then we write $\overline{\mathcal{L}}|_{\mathcal{Y}}=f^*\overline{\mathcal{L}}$. 

Next, we define ampleness for hermitian line bundles.
Let $\overline{\mathcal{L}}=(\mathcal{L},h)$ be a hermitian line bundle on $\mathcal{X}$. A global section $s\in H^0(\mathcal{X},\mathcal{L})$ is called \emph{small} if $\|s\|_{\sup}\le 1$ and \emph{strictly small} if $\|s\|_{\sup}<1$. We call $\mathcal{L}$ \emph{arithmetically ample} if the following conditions are satisfied:
\begin{enumerate}
	\item[(A1)] The line bundle $\mathcal{L}$ is relatively ample on $\mathcal{X}\to \mathrm{Spec}(\mathbb{Z})$.
	\item[(A2)] The hermitian line bundle $\overline{\mathcal{L}}$ is positive.
	\item[(A3)] The group $H^0(\mathcal{X},\mathcal{L}^{\otimes p})$ is generated by strictly small sections for infinitely many $p\in\mathbb{Z}_{\ge0}$.
\end{enumerate}
If $\overline{\mathcal{L}}$ is arithmetically ample, one directly checks, that the restriction $\overline{\mathcal{L}}|_{\mathcal{Z}}$ to any arithmetic subvariety $\mathcal{Z}\subseteq \mathcal{X}$ is also an arithmetically ample hermitian line bundle on $\mathcal{Z}$.

If $\overline{\mathcal{L}}=(\mathcal{L},h)$ is a hermitian line bundle, we can vary the metric by a constant factor to obtain another hermitian line bundle. For any real number $\tau\in\mathbb{R}$ we denote $\overline{\mathcal{L}}(\tau)$ for the hermitian line bundle $(\mathcal{L},e^{-2\tau}\cdot h)$. By the Poincaré--Lelong formula (\ref{equ_poincare-lelong}) the curvature form of $\overline{\mathcal{L}}$ is not effected by this change, that is
\begin{align}\label{equ_c1-metric-change}
c_1(\overline{\mathcal{L}})=c_1(\overline{\mathcal{L}}(\tau)).
\end{align}
If $\overline{\mathcal{L}}$ is arithmetically ample and $\tau\ge 0$, then also $\overline{\mathcal{L}}(\tau)$ is arithmetically ample. Indeed, the underlying line bundle and the curvature form are the same, such that we only have to check (A3). Let $s_1,\dots,s_n\in H^0(\mathcal{X},\mathcal{L}^{\otimes p})$ be a set of strictly small sections with respect to $|\cdot|_{\overline{\mathcal{L}}}$ generating $H^0(\mathcal{X},\mathcal{L}^{\otimes p})$ for some $p$. Then we get
$$\sup_{x\in\mathcal{X}(\mathbb{C})} |s_j(x)|_{\overline{\mathcal{L}}(\tau)}=\sup_{x\in\mathcal{X}(\mathbb{C})}e^{-\tau}|s_j(x)|_{\overline{\mathcal{L}}}<e^{-\tau}\cdot 1\le 1$$
for all $1\le j\le n$. That means $s_1,\dots, s_n$ are also strictly small with respect to $|\cdot|_{\overline{\mathcal{L}}(\tau)}$. Thus, $\overline{\mathcal{L}}(\tau)$ satisfies (A3) and hence, it is arithmetically ample.

Finally in this section, we define the following notation
$$\widehat{H}_{\le r}^0(\mathcal{X},\overline{\mathcal{L}})=\{s\in H^0(\mathcal{X},\mathcal{L})~|~\|s\|_{\sup}\le r\}$$
for any hermitian line bundle $\overline{\mathcal{L}}$ and any real number $r\in \mathbb{R}$. Moreover, we write $\widehat{H}^0(\mathcal{X},\overline{\mathcal{L}})=\widehat{H}_{\le 1}^0(\mathcal{X},\overline{\mathcal{L}})$.
\subsection{An Example: The Arithmetic Projective Line}\label{sec_projective-line}
In this section we study the line bundle $\mathcal{O}(1)$ on the projective line $\mathbb{P}_{\mathbb{Z}}^1$ equipped with its Fubini--Study metric over $\mathbb{P}_{\mathbb{C}}^1$. In particular, we show that $\overline{\mathcal{O}(1)}$ is not arithmetically ample, but $\overline{\mathcal{O}(1)}(\epsilon)$ is arithmetically ample for any $\epsilon>0$.

We write $Z_0$ and $Z_1$ for homogeneous coordinate functions on $\mathbb{P}^1_{\mathbb{Z}}$ and for every sections $s,t\in H^0(\mathbb{P}^1_\mathbb{Z},\mathcal{O}(1))_{\mathbb{C}}\cong H^0(\mathbb{P}^1_{\mathbb{C}},\mathcal{O}(1))$ we write $s=a_0Z_0+a_1Z_1$ and $t=b_0Z_0+b_1Z_1$ with $a_0,a_1,b_0,b_1\in\mathbb{C}$. Then the Fubini--Study metric $h_{\mathrm{FS}}$ on $\mathcal{O}(1)$ is defined by
$$h_{\mathrm{FS}}(s,t)=\frac{\overline{(a_0Z_0+a_1Z_1)}(b_0Z_0+b_1Z_1)}{|Z_0|^2+|Z_1|^2}.$$
It induces the Fubini--Study norm
$$|s|^2_{\mathrm{FS}}=h_{\mathrm{FS}}(s,s)=\frac{\left|a_0Z_0+a_1Z_1\right|^2}{|Z_0|^2+|Z_1|^2}$$
as a smooth function on $\mathbb{P}^1_\mathbb{C}$. Thus, we obtain on $\mathbb{P}^1_{\mathbb{Z}}$ the hermitian line bundle $\overline{\mathcal{O}(1)}=(\mathcal{O}(1),h_{\mathrm{FS}})$.  For the induced hermitian line bundle $\overline{\mathcal{O}(1)}^{\otimes n}$ for $n\in\mathbb{Z}_{\ge 1}$ the Fubini--Study norm of a section $s=\sum_{j=0}^n a_j Z_0^{j}Z_1^{n-j}\in H^0(\mathbb{P}^1_{\mathbb{Z}},\mathcal{O}(1)^{\otimes n})_{\mathbb{C}}$ is given by
$$|s|^2_{\mathrm{FS}}=\frac{\left|\sum_{j=0}^n a_j Z_0^{j}Z_1^{n-j}\right|^2}{\left(|Z_0|^2+|Z_1|^2\right)^n}.$$
We write $S^n_j=Z_0^jZ_1^{n-j}$. Then $S_0^n,\dots,S_n^n$ form a basis of $H^0(\mathbb{P}_\mathbb{Z}^1,\mathcal{O}(1)^{\otimes n})_{\mathbb{C}}$.

We denote the affine subset $U=\{[Z_0:Z_1]\in\mathbb{P}^1_{\mathbb{C}}~|~Z_1\neq 0\}\subseteq \mathbb{P}_\mathbb{C}^1$. This can be identified with $\mathbb{C}$ by setting $Z_1=1$. We write $z=\frac{Z_0}{Z_1}$ for the corresponding coordinate function on $\mathbb{C}$. The curvature form $c_1(\overline{\mathcal{O}(1)})$ is given on $U$ by
\begin{align}\label{equ_c1FS}
	c_1(\overline{\mathcal{O}(1)})=\frac{i}{2\pi}\frac{1}{(|z|^2+1)^2}dz\wedge d\overline{z}.
\end{align}
In particular, $c_1(\overline{\mathcal{O}(1)})$ is positive.

The hermitian line bundle $\overline{\mathcal{O}(1)}$ is not arithmetically ample on $\mathbb{P}_{\mathbb{Z}}^1$. Indeed, if $s_1,\dots,s_m$ are sections generating $H^0(\mathbb{P}_{\mathbb{Z}}^1,\mathcal{O}(1)^{\otimes n})$, then there has to be at least one section $s_k$ with $s_k([0:1])\neq 0$. Thus, we can write $s_k=\sum_{j=0}^n a_j S_j^n$ for some integers $a_j\in\mathbb{Z}$ with $a_0\neq 0$. As $S_0^n([0:1])=1$ and $S_j^n([0:1])=0$ for $j\ge 1$, we get
$$\|s_k\|_{\sup}\ge |s_k([0:1])|_{\mathrm{FS}}=\left|\sum_{j=0}^n a_j S_j^n([0:1])\right|_{\mathrm{FS}}=|a_0|\cdot |S_0^n([0:1])|_{\mathrm{FS}}=|a_0|\ge 1.$$
Hence, $s_k$ is not strictly small, such that $\overline{\mathcal{O}(1)}$ is not arithmetically ample. The following lemma shows, that however $\overline{\mathcal{O}(1)}(\epsilon)$ is arithmetically ample for every $\epsilon>0$.
\begin{Lem}
	The hermitian line bundle $\overline{\mathcal{O}(1)}(\epsilon)$ on $\mathbb{P}_{\mathbb{Z}}^1$ is arithmetically ample for every $\epsilon>0$.
\end{Lem}
\begin{proof}
	Clearly, the underlying line bundle $\mathcal{O}(1)$ is ample on $\mathbb{P}_{\mathbb{Z}}^1$. As $c_1(\overline{\mathcal{O}(1)}(\epsilon))=c_1(\overline{\mathcal{O}(1)})$ by Equation (\ref{equ_c1-metric-change}), the curvature form $c_1(\overline{\mathcal{O}(1)}(\epsilon))$ is also positive by (\ref{equ_c1FS}). Finally, for every $n\ge 1$ the basis $S_0^n,\dots,S_n^n$ of $H^0(\mathbb{P}_{\mathbb{Z}}^1,\mathcal{O}(1)^{\otimes n})$ consists of strictly small sections as
	$$\|S_j^n\|^2_{\sup}=e^{-2\epsilon n}\sup_{[z_0:z_1]\in \mathbb{P}_\mathbb{C}^1}\frac{|z_0^jz_1^{n-j}|^2}{(|z_0|^2+|z_1|^2)^n}\le e^{-2\epsilon n}<1.$$
\end{proof}

The next lemma computes some integrals with respect to the Fubini--Study metric. In particular, it shows that $S_0^n,\dots,S_n^n$ is an orthogonal basis with respect to the inner product associated to the Kähler form $c_1(\overline{\mathcal{O}(1)})$ as in (\ref{equ_inner-product}).
\begin{Lem}\label{lem_integral-P1}
	\hspace{1pt}
	\begin{enumerate}[(i)]
		\item We have
		$$\langle S_j^n,S_k^n\rangle=\int_{\mathbb{P}_{\mathbb{C}}^1}h_{\mathrm{FS}}(S_j^n,S_k^n)c_1(\overline{\mathcal{O}(1)})=\begin{cases} 0 & \text{if } j\neq k,\\ \frac{1}{\binom{n}{k}(n+1)} & \text{if } j=k.\end{cases}$$
		\item It holds
		$$\int_{\mathbb{P}^1_{\mathbb{C}}}\log|Z_0|_{\mathrm{FS}}c_1(\overline{\mathcal{O}(1)})=\int_{\mathbb{P}^1_{\mathbb{C}}}\log|Z_1|_{\mathrm{FS}}c_1(\overline{\mathcal{O}(1)})=-\frac{1}{2}.$$
	\end{enumerate}
\end{Lem}
\begin{proof}
	\begin{enumerate}[(i)]
		\item 
		We compute the integral over the affine subset $U\cong \mathbb{C}$
		\begin{align*} 
			\langle S_j^n, S_k^n\rangle=\int_{\mathbb{P}_{\mathbb{C}}^1} h_{\mathrm{FS}}(S_j^n, S_k^n)c_1(\overline{\mathcal{O}(1)})=\frac{i}{2\pi}\int_{\mathbb{C}}\frac{\overline{z}^jz^k}{(|z|^2+1)^{n+2}}dz\wedge d\overline{z}.
		\end{align*}
		Changing to polar coordinates $(r,\theta)$ with $z=re^{i\theta}$ and $dz\wedge d\overline{z}=-2ri dr\wedge d\theta$ we get 
		\begin{align*}
			\frac{i}{2\pi}\int_{\mathbb{C}}\frac{\overline{z}^jz^k}{(|z|^2+1)^{n+2}}dz\wedge d\overline{z}&=\frac{1}{\pi}\int_0^{2\pi}\int_0^\infty \frac{(r^je^{-i\theta j})(r^ke^{i\theta k})}{(r^2+1)^{n+2}} rdr\wedge d\theta\\
			&=\frac{1}{\pi}\int_0^\infty \frac{r^{j+k+1}}{(r^2+1)^{n+2}}dr\int_0^{2\pi}e^{i\theta(k-j)}d\theta.
		\end{align*}
		The case $j\neq k$ follows by
		\begin{align*}
			\int_0^{2\pi}e^{i\theta(k-j)}d\theta=\tfrac{1}{i(k-j)}(e^{2\pi i(k-j)}-e^0)=0.
		\end{align*}
		
		Let us now assume $j=k$.
		By a coordinate change $u=r^2+1$ we get
		\begin{align*}
			\langle S_k,S_k\rangle&=2\int_0^\infty\frac{r^{2k+1}}{(r^2+1)^{n+2}}dr=\int_1^\infty \frac{(u-1)^k}{u^{n+2}}du\\
			&=\sum_{j=0}^k \binom{k}{j}(-1)^{k-j}\int_1^\infty u^{j-n-2}du=\sum_{j=0}^k \binom{k}{j}\frac{(-1)^{k-j}}{n+1-j}=\frac{1}{\binom{n}{k}(n+1)}.
		\end{align*}
		Let us check the last equality by induction over $k$. It is clearly true for $k=0$. For $k>0$ the induction hypothesis allows us to compute
		\begin{align*}
			\sum_{j=0}^k \binom{k}{j}\frac{(-1)^{k-j}}{n+1-j}&=-\sum_{j=0}^{k-1}\left(\binom{k-1}{j}+\binom{k-1}{j-1}\right)\frac{(-1)^{(k-1)-j}}{n+1-j}+\frac{1}{n+1-k}\\
			&=-\frac{1}{\binom{n}{k-1}(n+1)}+\sum_{j=0}^{k-1}\binom{k-1}{j}\frac{(-1)^{(k-1)-j}}{(n-1)+1-j}\\
			&=-\frac{1}{\binom{n}{k-1}(n+1)}+\frac{1}{\binom{n-1}{k-1}n}\\
			&=\frac{(k-1)!\cdot(-(n-k+1)+(n+1))}{(n+1)\cdot\ldots\cdot (n-k+1)}=\frac{1}{\binom{n}{k}(n+1)}.
		\end{align*}
		Note that we always have $k\le n$. This proves (i).
		\item
		By symmetry it is enough to show $\int_{\mathbb{P}^1_{\mathbb{C}}}\log|Z_1|_{\mathrm{FS}}c_1(\overline{\mathcal{O}(1)})=-\frac{1}{2}$. We compute the integral on the affine set $U\cong \mathbb{C}$
		\begin{align*}
			\int_{\mathbb{P}^1_{\mathbb{C}}}\log|Z_1|_{\mathrm{FS}}c_1(\overline{\mathcal{O}(1)})=\frac{1}{2}\int_{\mathbb{C}}\left( \log \frac{1}{|z|^2+1}\right)\frac{i\cdot dz\wedge d\overline{z}}{2\pi(|z|^2+1)^2}
		\end{align*}
		We do again a coordinate change to polar coordinates $(r,\theta)$ by $z=re^{i\theta}$. This yields
		\begin{align*}
			\frac{1}{2}\int_{\mathbb{C}}\left( \log \frac{1}{|z|^2+1}\right)\frac{i\cdot dz\wedge d\overline{z}}{2\pi(|z|^2+1)^2}
			&=-\frac{1}{2\pi}\int_{0}^{2\pi}\int_0^\infty \frac{\log (r^2+1)}{(r^2+1)^2}rdr\wedge d\theta\\
			&=-\int_0^\infty \frac{\log (r^2+1)}{(r^2+1)^2}rdr
		\end{align*}
		We make another coordinate change $u=r^2+1$ to obtain
		\begin{align*}
			-\int_0^\infty \frac{\log (r^2+1)}{(r^2+1)^2}rdr=\frac{1}{2}\int_1^\infty \frac{-\log u}{u^2}du=\frac{1}{2}\left[\frac{\log u+1}{u}\right]_1^\infty=-\frac{1}{2}.
		\end{align*}
		This completes the proof.
	\end{enumerate}
\end{proof}
The lemma allows us to compute the $L^2$-norm with respect to the Fubini--Study metric of any section $s=\sum_{j=0}^n a_j S_j^n$ in $H^0(\mathbb{P}_{\mathbb{Z}}^1,\mathcal{O}(1)^{\otimes n})_{\mathbb{C}}$. Indeed, by the orthogonality of the basis $S_0^n,\dots,S_n^n$ of $H^0(\mathbb{P}_{\mathbb{Z}}^1,\mathcal{O}(1))$ we get
\begin{align}\label{equ_L2norm-FS}
	\|s\|^2=\left\|\sum_{j=0}^n a_jS_j^n\right\|^2=\sum_{j=0}^n |a_j|^2\cdot \|S_j^n\|^2=\frac{1}{n+1}\sum_{j=0}^n\frac{|a_j|^2}{\binom{n}{j}}.
\end{align}

\subsection{Arithmetic Intersection Numbers}\label{sec_intersection}
In this section we recall the definitions of arithmetic cycles, the arithmetic Chow group and the arithmetic intersection product of arithmetic cycles with hermitian line bundles. We also discuss the projection formula and the arithmetic degree function. For details we refer to \cite[Section 5.4]{Mor14}.

Let $\mathcal{X}$ be a generically smooth projective arithmetic variety of dimension $d$.
 An \emph{arithmetic cycle of codimension $p$ (or of dimension $d-p$)} is a pair $\overline{\mathcal{Z}}=(\mathcal{Z},T)$, where $\mathcal{Z}\subseteq \mathcal{X}$ is a cycle of pure codimension $p$ and $T\in D^{p-1,p-1}(\mathcal{X}(\mathbb{C}))$ is a real $(p-1,p-1)$-current on $\mathcal{X}(\mathbb{C})$ satisfying $F_{\infty}^*T=(-1)^{p-1}T$. We say that an arithmetic cycle $(\mathcal{Z},T)$ is of \emph{Green type}, if $T$ is a Green current of $\mathcal{Z}(\mathbb{C})$. We call an arithmetic cycle $(\mathcal{Z},T)$ \emph{effective} if $\mathcal{Z}$ is effective and $T\ge 0$. We write $\widehat{Z}_D^p(\mathcal{X})$ for the group of all arithmetic cycles of codimension $p$ on $\mathcal{X}$ and $\widehat{Z}^p(\mathcal{X})$ for the subgroup of $\widehat{Z}_D^p(\mathcal{X})$ of arithmetic cycles of Green type. Further, we denote $\widehat{Z}_{D,l}(\mathcal{X})=\widehat{Z}_D^{d-l}(\mathcal{X})$ and $\widehat{Z}_l(\mathcal{X})=\widehat{Z}^{d-l}(\mathcal{X})$.

We give two families of examples of arithmetic cycles of Green type, which will be defined to be rationally equivalent to $0$. Let $\mathcal{Y}$ be a positive dimensional integral subscheme of $\mathcal{X}$ and $\phi$ a non-zero rational section of $\mathcal{O}_{\mathcal{Y}}$ on $\mathcal{Y}$. Let $p-1$ denote the codimension of $\mathcal{Y}$. We obtain a current $[-\log|\phi|^2]_{\mathcal{Y}(\mathbb{C})}\in D^{p-1,p-1}(\mathcal{X}(\mathbb{C}))$ by sending any $\eta \in A^{d-p,d-p}(\mathcal{X}(\mathbb{C}))$ to
$$[-\log|\phi|^2]_{\mathcal{Y}(\mathbb{C})}(\eta)=\int_{\mathcal{Y}(\mathbb{C})}(-\log|\phi|^2)\eta.$$
Note that $\dim \mathcal{X}(\mathbb{C})=d-1$ and $\dim \mathcal{Y}(\mathbb{C})=d-p$ if it is non-empty.
Then
$$\widehat{(\phi)}:=\left(\mathrm{div}(\mathcal{O}_{\mathcal{Y}},\phi),[-\log|\phi|^2]_{\mathcal{Y}(\mathbb{C})}\right)$$
is an arithmetic cycle of codimension $p$, which is of Green type by the Poincaré--Lelong formula (\ref{equ_poincare-lelong}). If $u\in D^{p-2,p-1}(\mathcal{X}(\mathbb{C}))$ and $v\in D^{p-1,p-2}(\mathcal{X}(\mathbb{C})$ are two currents, then
$$(0,\partial u+\overline{\partial} v)\in \widehat{Z}^p(\mathcal{X})$$
is an arithmetic cycle of codimension $p$ and of Green type.
We write $\widehat{\mathrm{Rat}}^p(\mathcal{X})$ for the subgroup of $\widehat{Z}^p(\mathcal{X})$ generated by all arithmetic cycles of the form $\widehat{(\phi)}$ and $(0,\partial u+\overline{\partial} v)$ as above and we define the \emph{arithmetic Chow group (of Green type)} by
$$\widehat{\mathrm{CH}}_D^p(\mathcal{X})=\widehat{Z}_D^p(\mathcal{X})/\widehat{\mathrm{Rat}}^p(\mathcal{X}) \quad \text{and} \quad \widehat{\mathrm{CH}}^p(\mathcal{X})=\widehat{Z}^p(\mathcal{X})/\widehat{\mathrm{Rat}}^p(\mathcal{X}).$$
We set $\widehat{\mathrm{CH}}_{D,l}(\mathcal{X})=\widehat{\mathrm{CH}}_D^{d-l}(\mathcal{X})$ and $\widehat{\mathrm{CH}}_l(\mathcal{X})=\widehat{\mathrm{CH}}^{d-l}(\mathcal{X})$.

If $\overline{\mathcal{L}}$ is a hermitian line bundle on $\mathcal{X}$ and $s$ is a non-zero rational section of $\mathcal{L}$, then 
$$\widehat{\mathrm{div}}(\overline{\mathcal{L}},s)=(\mathrm{div}(s),-\log h(s,s))\in \widehat{Z}^1(\mathcal{X})$$
is an arithmetic cycle of codimension $1$, which is of Green type by the Poincaré--Lelong formula (\ref{equ_poincare-lelong}). Since the quotient of two rational sections of $\mathcal{L}$ is a rational section $\phi$ of $\mathcal{O}_\mathcal{X}$ and the quotient of their norms is the absolute value $|\phi|$, we obtain that the class of $\widehat{\mathrm{div}}(\overline{\mathcal{L}},s)$ in the arithmetic Chow group $\widehat{\mathrm{CH}}^{1}(\mathcal{X})$ does not depend on the choice of $s$. Thus, we have a well-defined class $\widehat{c_1}(\overline{\mathcal{L}})\in\widehat{\mathrm{CH}}^{1}(\mathcal{X})$.
The map $\widehat{c_1}\colon \widehat{\mathrm{Pic}}(\mathcal{X})\xrightarrow{\sim} \widehat{\mathrm{CH}}^{1}(\mathcal{X})$ is an isomorphism, see for example \cite[Proposition 1.3.4]{KMY02}. Thus, we also just write $\overline{\mathcal{L}}$ for $\widehat{c_1}(\overline{\mathcal{L}})$.

Gillet and Soulé \cite{GS90} defined an intersection pairing
$$\widehat{\mathrm{CH}}^q(\mathcal{X})\times \widehat{\mathrm{CH}}^p(\mathcal{X})\to\widehat{\mathrm{CH}}^{p+q}(\mathcal{X})\otimes_\mathbb{Z}\mathbb{Q}$$
for the arithmetic Chow groups under the assumption that $\mathcal{X}$ is regular.
We will recall its definition for $q=1$. In this case we do not have to assume that $\mathcal{X}$ is regular but only that $\mathcal{X}_{\mathbb{Q}}$ is smooth. Let $(\mathcal{Z},g_\mathcal{Z})\in \widehat{Z}_D^p(\mathcal{X})$ be an arithmetic cycle, $\overline{\mathcal{L}}$ a hermitian line bundle and $s$ a non-zero rational section of $\mathcal{L}$, such that $\dim (\mathrm{div}(s)\cap \mathcal{Z})=d-p-1$. The arithmetic intersection of $\widehat{\mathrm{div}}(\overline{\mathcal{L}},s)$ and $(\mathcal{Z},g_\mathcal{Z})$ is defined by
\begin{align}\label{equ_intersection-definition}
\widehat{\mathrm{div}}(\overline{\mathcal{L}},s)\cdot (\mathcal{Z},T)=(\mathrm{div}(s)\cdot \mathcal{Z},[-\log h(s,s)]_{\mathcal{Z}(\mathbb{C})}+c_1(\overline{\mathcal{L}})\wedge T)\in \widehat{Z}^{p+1}(\mathcal{X}),
\end{align}
where $\mathrm{div}(s)\cdot \mathcal{Z}$ denotes the classical intersection product of $\mathrm{div}(s)$ and $\mathcal{Z}$ in $\mathcal{X}$ and $[-\log h(s,s)]_{\mathcal{Z}(\mathbb{C})}$ denotes the $(p,p)$-current given by
$$[-\log h(s,s)]_{\mathcal{Z}(\mathbb{C})}\colon A^{d-1-p,d-1-p}\to\mathbb{C},\qquad \eta\mapsto \int_{\mathcal{Z}(\mathbb{C})}(-\log h(s,s))\eta.$$
The class of $\widehat{\mathrm{div}}(\overline{\mathcal{L}},s)\cdot (\mathcal{Z},T)$ in $\widehat{\mathrm{CH}}_D^{p+1}(\mathcal{X})$ does not depend on the choice of $s$ and only depends on the class of $(\mathcal{Z},T)$ in $\widehat{\mathrm{CH}}_D^p(\mathcal{X})$. Thus, we obtain a bilinear intersection product
\begin{align}\label{equ_intersection-pairing}
\widehat{\mathrm{Pic}}(\mathcal{X})\times \widehat{\mathrm{CH}}_D^p(\mathcal{X})\to \widehat{\mathrm{CH}}_D^{p+1}(\mathcal{X}).
\end{align}
One can check directly that this pairing respects arithmetic cycles of Green type. Hence, we get by restriction also a bilinear intersection product
$$\widehat{\mathrm{Pic}}(\mathcal{X})\times \widehat{\mathrm{CH}}^p(\mathcal{X})\to \widehat{\mathrm{CH}}^{p+1}(\mathcal{X}).$$

Next, we recall the projection formula for the pairing (\ref{equ_intersection-pairing}). Let $f\colon \mathcal{Y}\to\mathcal{X}$ be any projective morphism of generically smooth projective arithmetic varieties. 
Further, let $\overline{\mathcal{L}}$ be any hermitian line bundle on $\mathcal{X}$ and $\alpha$ any class in $\widehat{\mathrm{CH}}_D^p(\mathcal{Y})$. If $\alpha$ is represented by some arithmetic cycle $(\mathcal{Z},T)$, then its push-forward $f_*\alpha$ is given by the class of $(f_*\mathcal{Z},f_*T)$, where $f_*\mathcal{Z}$ is the push-forward of $\mathcal{Z}$ as known in classical intersection theory and $f_*T$ is the push-forward of the current $T$ as defined in Section \ref{sec_preliminaries}. The arithmetic projection formula \cite[Proposition 5.18]{Mor14} states that
$$f_*((f^*\overline{\mathcal{L}})\cdot \alpha)=\overline{\mathcal{L}}\cdot f_*\alpha.$$

We now give the definition of the arithmetic degree.
Let $(\mathcal{Z},T)\in \widehat{Z}_{D,0}(\mathcal{X})$ be any arithmetic cycle of dimension $0$ on $\mathcal{X}$.
We can write $\mathcal{Z}=\sum_{x\in \mathcal{X}_{(0)}} n_x x$ for some $n_x\in\mathbb{Z}$, where $\mathcal{X}_{(0)}$ denotes the scheme theoretic $0$-dimensional points of $\mathcal{X}$.
The arithmetic degree of $(\mathcal{Z},T)$ is defined by
\begin{align}\label{equ_degree-definition}
\widehat{\deg}(\mathcal{Z},T)=\sum_{x\in\mathcal{X}_{(0)}}n_x\log \# k(x)+\frac{1}{2}\int_{\mathcal{X}(\mathbb{C})}T,
\end{align}
where $k(x)$ denotes the residue field at $x$. One directly checks, that $\widehat{\deg}$ is zero on $\widehat{\mathrm{Rat}}^{d}(\mathcal{X})$. Hence, we obtain a linear map
$$\widehat{\deg}\colon \widehat{\mathrm{CH}}_{D,0}(\mathcal{X})\to\mathbb{R}.$$
If we apply the intersection product (\ref{equ_intersection-pairing}) inductively and compose it with the arithmetic degree map, we obtain a multi-linear intersection pairing
$$\widehat{\mathrm{Pic}}(\mathcal{X})^l\times \widehat{\mathrm{CH}}_{D,l}(\mathcal{X})\to \mathbb{R},\qquad (\overline{\mathcal{L}}_1,\dots,\overline{\mathcal{L}}_l,\alpha)\mapsto\widehat{\deg}(\overline{\mathcal{L}}_1\cdots\overline{\mathcal{L}}_l\cdot\alpha).$$
We also just write
$$(\overline{\mathcal{L}}_1\cdots\overline{\mathcal{L}}_l\cdot(\mathcal{Z},T))=\widehat{\deg}(\overline{\mathcal{L}}_1\cdots\overline{\mathcal{L}}_l\cdot\alpha),$$
if $\alpha$ is represented by the arithmetic cycle $(\mathcal{Z},T)\in\widehat{Z}_{D,l}(\mathcal{X})$. Moreover, if $(\mathcal{Z},T)=(\mathcal{X},0)$, we omit it in the notation.

If $f\colon \mathcal{Y}\to\mathcal{X}$ is a projective morphism of projective arithmetic varieties with smooth generic fiber, then it holds $\widehat{\deg}~f_*(\mathcal{Z},T)=\widehat{\deg}(\mathcal{Z},T)$ for any arithmetic cycle $(\mathcal{Z},T)\in\widehat{Z}_{D,0}(\mathcal{Y})$, see \cite[Proposition 5.23 (1)]{Mor14}. Thus, an inductive application of the arithmetic projection formula gives
\begin{align}\label{equ_arithmetic-projection-formula}
	(f^*\overline{\mathcal{L}}_1\cdots f^*\overline{\mathcal{L}}_l\cdot (\mathcal{Z},T))=(\overline{\mathcal{L}}_1\cdots\overline{\mathcal{L}}_l\cdot f_*(\mathcal{Z},T))
\end{align}
for any hermitian line bundles $\overline{\mathcal{L}}_1,\dots\overline{\mathcal{L}}_l\in\widehat{\mathrm{Pic}}(\mathcal{X})$ and any $(\mathcal{Z},T)\in\widehat{Z}_{D,l}(\mathcal{Y})$.

The arithmetic projection formula allows us to define arithmetic intersection numbers even if the generic fiber is no longer smooth. Assume that $\mathcal{Y}$ is a projective arithmetic variety of dimension $d$, which is not necessarily generically smooth. Let $\overline{\mathcal{L}}_1,\dots,\overline{\mathcal{L}}_d\in\widehat{\mathrm{Pic}}(\mathcal{Y})$ be hermitian line bundles on $\mathcal{Y}$ and $\mu\colon \mathcal{Y}'\to\mathcal{Y}$ a generic resolution of singularities of $\mathcal{Y}$. By the arithmetic projection formula the number
$$(\overline{\mathcal{L}}_1\cdots\overline{\mathcal{L}}_d):=(\mu^*\overline{\mathcal{L}}_1\cdots\mu^*\overline{\mathcal{L}}_d)$$
does not depend on the choice of $\mu$ as for any two generic resolutions of singularities there exists a third resolution factorizing over both.

\subsection{Restricted Arithmetic Intersection Numbers}\label{sec_restriction}
In this section we study relations between the arithmetic intersection number of hermitian line bundles with some arithmetic cycle $(\mathcal{Z},T)$ and the arithmetic intersection numbers of the same line bundles restricted to the components of $\mathcal{Z}$. We also define heights and prove Proposition \ref{pro_equidistribution-individual}.
Again, we assume $\mathcal{X}$ to be a generically smooth projective arithmetic variety of dimension $d$.

First, we treat the case $T=0$. Here, we get the following lemma.
\begin{Lem}\label{lem_intersection-Z0}
	Let $\mathcal{Z}$ be any cycle of $\mathcal{X}$ of pure dimension $l$ and let $$\mathcal{Z}=\sum_{j=1}^{n_0}a_{0,j}\mathcal{Z}_{0,j}+\sum_{p\in|\mathrm{Spec}(\mathbb{Z})|}\sum_{j=1}^{n_p}a_{p,j}\mathcal{Z}_{p,j}$$
	be its decomposition into irreducible cycles, where all coefficients are integers, we write $|\mathrm{Spec}(\mathbb{Z})|$ for the closed points of $\mathrm{Spec}(\mathbb{Z})$, and under the structure map $\pi\colon\mathcal{X}\to\mathrm{Spec}(\mathbb{Z})$ we have $\pi(\mathcal{Z}_{0,j})=\mathrm{Spec}(\mathbb{Z})$ and $\pi(\mathcal{Z}_{p,j})=\{p\}$ for all primes $p\in|\mathrm{Spec}(\mathbb{Z})|$. For any hermitian line bundles $\overline{\mathcal{L}}_1,\dots,\overline{\mathcal{L}}_l$ on $D$ it holds
	\begin{align*}
	&(\overline{\mathcal{L}}_1\cdots\overline{\mathcal{L}}_l\cdot(\mathcal{Z},0))\\
	&=\sum_{j=1}^{n_0}a_{0,j}(\overline{\mathcal{L}}_1|_{\mathcal{Z}_{0,j}}\cdots\overline{\mathcal{L}}_l|_{\mathcal{Z}_{0,j}})+\sum_{p\in|\mathrm{Spec}(\mathbb{Z})|}\sum_{j=1}^{n_p}a_{p,j}(\mathcal{L}_1|_{\mathcal{Z}_{p,j}}\cdots\mathcal{L}_l|_{\mathcal{Z}_{p,j}})\log p,
	\end{align*}
	where $(\mathcal{L}_1|_{\mathcal{Z}_{p,j}}\cdots\mathcal{L}_l|_{\mathcal{Z}_{p,j}})$ denotes the classical intersection number of the line bundles $\mathcal{L}_1|_{\mathcal{Z}_{p,j}}$, ..., $\mathcal{L}_l|_{\mathcal{Z}_{p,j}}$ on the irreducible projective $\mathbb{F}_p$-variety $\mathcal{Z}_{p,j}$.
\end{Lem}
\begin{proof}
	By linearity we can assume that $\mathcal{Z}$ is irreducible. If $\mathcal{Z}=\mathcal{Z}_{p,1}$ for some prime $p$, then the we have
	$$(\overline{\mathcal{L}}_1\cdots\overline{\mathcal{L}}_l\cdot(\mathcal{Z},0))=\widehat{\deg}(\mathcal{L}_1\cdots\mathcal{L}_l\cdot\mathcal{Z},0)$$
	by the definition of the arithmetic intersection product as $\mathcal{Z}(\mathbb{C})=\emptyset$. Thus, in this case the assertion follows from \cite[Proposition 5.23 (3)]{Mor14}.
	
	If $\mathcal{Z}=\mathcal{Z}_{0,1}$, then $\mathcal{Z}$ is a projective arithmetic variety. Let $\mu\colon \mathcal{Z}'\to\mathcal{Z}$ be a generic resolution of singularities of $\mathcal{Z}$. We denote $\mu'\colon \mathcal{Z}'\to \mathcal{X}$ for the composition $\mathcal{Z}'\xrightarrow{\mu}\mathcal{Z}\xrightarrow{\subseteq}\mathcal{X}$. As $\mu'_*\mathcal{Z}'=\mathcal{Z}$, the arithmetic projection formula (\ref{equ_arithmetic-projection-formula}) gives
	$$({\mu'}^*\overline{\mathcal{L}}_1\cdots{\mu'}^*\overline{\mathcal{L}}_l)=({\mu'}^*\overline{\mathcal{L}}_1\cdots{\mu'}^*\overline{\mathcal{L}}_l\cdot(\mathcal{Z}',0))=(\overline{\mathcal{L}}_1\cdots\overline{\mathcal{L}}_l\cdot (\mathcal{Z},0)).$$
	If we denote $\iota\colon \mathcal{Z}\to\mathcal{X}$ for the inclusion map, we get by the definition of arithmetic intersection numbers on not necessarily generically smooth arithmetic varieties
	$$(\overline{\mathcal{L}}_1|_{\mathcal{Z}}\cdots\overline{\mathcal{L}}_l|_{\mathcal{Z}})=(\iota^*\overline{\mathcal{L}}_1\cdots\iota^*\overline{\mathcal{L}}_l)=({\mu'}^*\overline{\mathcal{L}}_1\cdots{\mu'}^*\overline{\mathcal{L}}_l).$$
	The assertion of the lemma follows by combing both equations.
\end{proof}

To reduce the general case to the case $T=0$, we compute their difference in the next lemma.
\begin{Lem}\label{lem_intersection-restriction}
	For any arithmetic cycle $(\mathcal{Z},T)\in \widehat{Z}_{D,l}(\mathcal{X})$ and any hermitian line bundles $\overline{\mathcal{L}}_1,\dots \overline{\mathcal{L}}_l\in \widehat{\mathrm{Pic}}(\mathcal{X})$ it holds
	\begin{align*}
		\left(\overline{\mathcal{L}}_1\cdots \overline{\mathcal{L}}_l\cdot (\mathcal{Z},T)\right)=\left(\overline{\mathcal{L}}_1\cdots \overline{\mathcal{L}}_{l}\cdot(\mathcal{Z},0)\right)+\frac{1}{2}\int_{\mathcal{X}(\mathbb{C})}T\wedge c_1(\overline{\mathcal{L}}_1)\wedge\dots\wedge c_1(\overline{\mathcal{L}}_l).
	\end{align*}	
\end{Lem}
\begin{proof}
	By linearity and an inductive application of Equation (\ref{equ_intersection-definition}) we get
	\begin{align*}
			\left(\overline{\mathcal{L}}_1\cdots \overline{\mathcal{L}}_l\cdot (\mathcal{Z},T)\right)-\left(\overline{\mathcal{L}}_1\cdots \overline{\mathcal{L}}_{l}\cdot(\mathcal{Z},0)\right)&=\left(\overline{\mathcal{L}}_1\cdots \overline{\mathcal{L}}_l\cdot (0,T)\right)\\
			&=\widehat{\deg}(0,c_1(\overline{\mathcal{L}}_l)\wedge\dots\wedge c_1(\overline{\mathcal{L}}_1)\wedge T)\\
			&=\frac{1}{2}\int_{\mathcal{X}(\mathbb{C})}T\wedge c_1(\overline{\mathcal{L}}_1)\wedge\dots\wedge c_1(\overline{\mathcal{L}}_l).
	\end{align*}
\end{proof}

If the arithmetic cycle $(\mathcal{Z},T)$ represents an hermitian line bundle $\overline{\mathcal{L}}_d$, Lemma \ref{lem_intersection-restriction} can be read as
\begin{align}\label{equ_chambert-loir}
\left(\overline{\mathcal{L}}_1\cdots\overline{\mathcal{L}}_d\right)=\left(\overline{\mathcal{L}}_1\cdots \overline{\mathcal{L}}_{d-1}\cdot(\mathrm{div}(s),0)\right)-\int_{\mathcal{X}(\mathbb{C})}\log |s|_{\overline{\mathcal{L}}_d} c_1(\overline{\mathcal{L}}_1)\cdots c_1(\overline{\mathcal{L}}_{d-1}),
\end{align}
where $s$ is any non-zero rational section of $\mathcal{L}_d$. If $\overline{\mathcal{L}}$ is an arithmetically ample hermitian line bundle we can use Lemmas \ref{lem_intersection-Z0} and \ref{lem_intersection-restriction} to bound the integral in Lemma \ref{lem_intersection-restriction} by the arithmetic intersection number.
\begin{Lem}\label{lem_integral-bound}
	Let $(\mathcal{Z},T)\in \widehat{Z}_p(\mathcal{X})$ be any arithmetic cycle of dimension $p$ of $\mathcal{X}$, such that $\mathcal{Z}$ is effective and non-zero. Further, let $\overline{\mathcal{L}}$ be any arithmetically ample hermitian line bundle on $\mathcal{X}$. It holds
	$$\frac{1}{2}\int_{\mathcal{X}(\mathbb{C})}T\wedge c_1(\overline{\mathcal{L}})^l< \left(\overline{\mathcal{L}}^l\cdot(\mathcal{Z},T)\right).$$
\end{Lem}
\begin{proof}
	By Lemma \ref{lem_intersection-restriction} we only have to show
	$$\left(\overline{\mathcal{L}}^l\cdot(\mathcal{Z},0)\right)> 0.$$
	By Lemma \ref{lem_intersection-Z0} and by the effectiveness of $\mathcal{Z}$ it is enough to prove 
	$$\left(\overline{\mathcal{L}}|_{\mathcal{Z}_{0,j}}^l\right)> 0\quad \text{and}\quad \left(\mathcal{L}|_{\mathcal{Z}_{p,j}}^l\right)> 0$$
	for any $\mathcal{Z}_{0,j}$ and $\mathcal{Z}_{p,j}$ as in the decomposition in Lemma \ref{lem_intersection-Z0}. The first inequality follows from \cite[Proposition 5.39]{Mor14} and the second inequality follows from the ampleness of $\mathcal{L}$.
\end{proof}
We again consider the case where $(\mathcal{Z},T)$ represents a hermitian line bundle. In this case Lemma \ref{lem_integral-bound} can be used to bound the integral in Equation (\ref{equ_chambert-loir}) independently of the choice of the section $s$.
\begin{Lem}\label{lem_bound-of-integral-of-section}
	Let $\overline{\mathcal{L}}$ be an arithmetically ample hermitian line bundle on $\mathcal{X}$. For all $p\in\mathbb{Z}_{\ge 1}$ and all non-zero sections $s\in H^0(\mathcal{X},\mathcal{L}^{\otimes p})$ we have
	$$-\frac{1}{p}\int_{\mathcal{X}(\mathbb{C})}\log|s|_{\overline{\mathcal{L}}^{\otimes p}}c_1(\overline{\mathcal{L}})^{d-1}\le \left(\overline{\mathcal{L}}^{d}\right).$$
\end{Lem}
\begin{proof}
	We can apply Lemma \ref{lem_integral-bound} to $(\mathcal{Z},T)=(\mathrm{div}(s),-\log|s|^2_{\overline{\mathcal{L}}^{\otimes p}})$, which represents the hermitian line bundle $\overline{\mathcal{L}}^{\otimes p}$. This gives
	\begin{align*}
	&-\int_{\mathcal{X}(\mathbb{C})}\log |s|_{\overline{\mathcal{L}}^{\otimes p}}c_1(\overline{\mathcal{L}})^{d-1}=\frac{1}{2}\int_{\mathcal{X}(\mathbb{C})}\left(-\log |s|^2_{\overline{\mathcal{L}}^{\otimes p}}\right)c_1(\overline{\mathcal{L}})^{d-1}\\
	&\le \left(\overline{\mathcal{L}}^{d-1}\cdot(\mathcal{Z},T)\right)=\left(\overline{\mathcal{L}}^{d-1}\cdot\overline{\mathcal{L}}^{\otimes p}\right)=p\cdot \left(\overline{\mathcal{L}}^{d}\right).
	\end{align*}
	We get the inequality in the lemma after dividing by $p$ on both sides.
\end{proof}

Next, we discuss heights. We fix an arithmetically ample hermitian line bundle $\overline{\mathcal{L}}$. Let $\mathcal{Z}$ be any effective and equidimensional cycle on $\mathcal{X}$ with non-zero horizontal part. We define the normalized height of $\mathcal{Z}$ by the arithmetic self-intersection number
$$h_{\overline{\mathcal{L}}}(\mathcal{Z})=\frac{(\overline{\mathcal{L}}|_{\mathcal{Z}}^{\dim\mathcal{Z}})}{\mathcal{L}_{\mathbb{C}}|_{\mathcal{Z}_{\mathbb{C}}}^{\dim \mathcal{Z}_{\mathbb{C}}}}$$
of the hermitian line bundle $\overline{\mathcal{L}}$ restricted to $\mathcal{Z}$. Here, $\mathcal{L}_{\mathbb{C}}|_{\mathcal{Z}_{\mathbb{C}}}^{\dim \mathcal{Z}_{\mathbb{C}}}$ denotes the geometric self-intersection number of $\mathcal{L}_{\mathbb{C}}$ on $\mathcal{Z}_{\mathbb{C}}$. This is classically defined if $\mathcal{Z}_{\mathbb{C}}$ is irreducible, and in general we extend the definition by linearity. In particular, if $\dim \mathcal{Z}_{\mathbb{C}}=0$, then $\mathcal{L}_{\mathbb{C}}|_{\mathcal{Z}_{\mathbb{C}}}^{\dim \mathcal{Z}_{\mathbb{C}}}=\deg \mathcal{Z}_{\mathbb{C}}$ is the degree of $\mathcal{Z}_\mathbb{C}$ as a $0$-dimensional cycle of $\mathcal{X}_{\mathbb{C}}$. If $s$ is any non-zero global section of $\mathcal{L}^{\otimes p}$ and $d\ge 2$, we get by Equation (\ref{equ_chambert-loir})
\begin{align}\label{equ_height}
h_{\overline{\mathcal{L}}}(\mathcal{X})&=\frac{(\overline{\mathcal{L}}^d)}{\mathcal{L}_\mathbb{C}^{d-1}}=\frac{(\overline{\mathcal{L}}\cdots\overline{\mathcal{L}}\cdot\overline{\mathcal{L}}^{\otimes p})}{p\cdot\mathcal{L}_\mathbb{C}^{d-1}}
=\frac{(\overline{\mathcal{L}}|_{\mathrm{div}(s)}^{d-1})-\int_{\mathcal{X}(\mathbb{C})}\log|s|_{\overline{\mathcal{L}}^{\otimes p}}c_1(\overline{\mathcal{L}})^{d-1}}{p\cdot \mathcal{L}_{\mathbb{C}}^{d-1}}\\
&=\frac{(\overline{\mathcal{L}}|_{\mathrm{div}(s)}^{d-1})}{\mathcal{L}_{\mathbb{C}}|_{\mathrm{div}(s)_{\mathbb{C}}}^{d-2}}-\frac{1}{p\cdot \mathcal{L}_{\mathbb{C}}^{d-1}}{\int_{\mathcal{X}(\mathbb{C})}\log|s|_{\overline{\mathcal{L}}^{\otimes p}}c_1(\overline{\mathcal{L}})^{d-1}}\nonumber\\
&=h_{\overline{\mathcal{L}}}(\mathrm{div}(s))-\frac{1}{p\cdot \mathcal{L}_{\mathbb{C}}^{d-1}}{\int_{\mathcal{X}(\mathbb{C})}\log|s|_{\overline{\mathcal{L}}^{\otimes p}}c_1(\overline{\mathcal{L}})^{d-1}}.\nonumber
\end{align}
As an application of this formula, we can prove the following proposition, which provides a sufficient condition for sequences of sections, such that their divisors tend to equidistribute with respect to $c_1(\overline{\mathcal{L}})$. 
\begin{Pro}[= Proposition \ref{pro_equidistribution-individual}]\label{pro_equidistribution-individual-text}
	Let $\mathcal{X}$ be any generically smooth projective arithmetic variety of dimension $d\ge 2$. Let $\overline{\mathcal{L}}$ be any arithmetically ample hermitian line bundle on $\mathcal{X}$. For any sequence $(s_p)_{p\in\mathbb{Z}_{\ge 1}}$ of sections $s_p\in H^0(\mathcal{X},\overline{\mathcal{L}}^{\otimes p})$ satisfying 
	$$\lim_{p\to \infty}\left(h_{\overline{\mathcal{L}}}(\mathrm{div}(s_p))-\tfrac{1}{p}\log\|s_p\|_{\sup}\right)=h_{\overline{\mathcal{L}}}(\mathcal{X})$$
	and any $(d-2,d-2)$ $C^0$-form $\Phi$ on $\mathcal{X}(\mathbb{C})$ it holds
	$$\lim_{p\to \infty}\frac{1}{p}\int_{\mathrm{div}(s_p)(\mathbb{C})}\Phi=\int_{\mathcal{X}(\mathbb{C})}\Phi\wedge c_1(\overline{\mathcal{L}}).$$	
\end{Pro}
\begin{proof}
	By the assumptions we have
	$$\lim_{p\to \infty}\left(h_{\overline{\mathcal{L}}}(\mathrm{div}(s_p))-\tfrac{1}{p}\log\|s_p\|_{\sup}-h_{\overline{\mathcal{L}}}(\mathcal{X})\right)=0.$$
	By Equation (\ref{equ_height}) this is equivalent to
	$$\lim_{p\to \infty}\frac{1}{p}\left(\frac{1}{\mathcal{L}_{\mathbb{C}}^{d-1}}\int_{\mathcal{X}(\mathbb{C})}\log|s_p|_{\overline{\mathcal{L}}^{\otimes p}}c_1(\overline{\mathcal{L}})^{d-1}-\log\|s_p\|_{\sup}\right)=0.$$
	As $\mathcal{L}_{\mathbb{C}}^{d-1}=\int_{\mathcal{X}(\mathbb{C})}c_1(\overline{\mathcal{L}})^{d-1}$, the assertion of the proposition now follows from Lemma \ref{lem_equidistribution} (ii).
\end{proof}

\subsection{The Lattice of Global Sections}\label{sec_lattice}
We discuss some relations between the global sections $H^0(\mathcal{X},\mathcal{L}^{\otimes p})$, the real vector space $H^0(\mathcal{X},\mathcal{L}^{\otimes p})_\mathbb{R}$ and the complex vector space $H^0(\mathcal{X}(\mathbb{C}),\mathcal{L}^{\otimes p}(\mathbb{C}))$ for an arithmetically ample line bundle $\overline{\mathcal{L}}$. It turns out, that the first one is a lattice in the second one and the third one is equipped with an inner product induced by the metric of $\overline{\mathcal{L}}$, which takes real values on $H^0(\mathcal{X},\mathcal{L}^{\otimes p})_\mathbb{R}$. 

By $\mathcal{X}$ we always mean a generically smooth projective arithmetic variety of dimension $d$ and $\overline{\mathcal{L}}$ always denotes an arithmetically ample hermitian line bundle on $\mathcal{X}$. In the following we write
$$H^0(\mathcal{X},\mathcal{L}^{\otimes p})_\mathbb{R}=
H^0(\mathcal{X},\mathcal{L}^{\otimes p})\otimes_\mathbb{Z}\mathbb{R},\qquad H^0(\mathcal{X},\mathcal{L}^{\otimes p})_\mathbb{C}=
H^0(\mathcal{X},\mathcal{L}^{\otimes p})\otimes_\mathbb{Z}\mathbb{C}.$$
Since $H^0(\mathcal{X},\mathcal{L}^{\otimes p})_\mathbb{C}\cong H^0(\mathcal{X}(\mathbb{C}),\mathcal{L}^{\otimes p}(\mathbb{C}))$, the hermitian metric of $\overline{\mathcal{L}}^{\otimes p}$ defines an inner product $\langle\cdot,\cdot\rangle$ on $H^0(\mathcal{X},\mathcal{L}^{\otimes p})_\mathbb{C}$ by setting
\begin{align}\label{equ_innerproduct}
\langle s_1,s_2\rangle=\int_{\mathcal{X}(\mathbb{C})} h(s_1(x),s_2(x))\frac{c_1(\overline{\mathcal{L}})^{d-1}}{\mathcal{L}_{\mathbb{C}}^{d-1}}.
\end{align}
for every $s_1,s_2\in H^0(\mathcal{X},\mathcal{L}^{\otimes p})_\mathbb{C}$.
If we denote the map
$$F_\infty^*\colon H^0(\mathcal{X},\mathcal{L})\otimes_\mathbb{Z}\mathbb{C}\to H^0(\mathcal{X},\mathcal{L})\otimes_\mathbb{Z}\mathbb{C},\qquad s\otimes \alpha\mapsto s\otimes \overline{\alpha},$$
then it holds $(F_\infty^*(s))(\overline{x})=F_\infty(s(x))$ by \cite[Equation (5.1)]{Mor14}.

As the hermitian metric of $\overline{\mathcal{L}}$ is of real type, it holds 
\begin{align*}
|s(x)|^2_{\overline{\mathcal{L}}}=\overline{|F_\infty(s(x))|^2_{\overline{\mathcal{L}}}}=|F_\infty(s(x))|^2_{\overline{\mathcal{L}}}=|(F_\infty^*(s))(\overline{x})|^2_{\overline{\mathcal{L}}}=|s(\overline{x})|^2_{\overline{\mathcal{L}}}
\end{align*}
for every $x\in \mathcal{X}(\mathbb{C})$ and $s\in H^0(\mathcal{X}(\mathbb{C}),\mathcal{L}(\mathbb{C}))$.
Thus, we get
$$F_\infty^*\left(\partial\overline{\partial}\log |s(x)|\right)=\overline{\partial}\partial\log |s(\overline{x})|=-\partial\overline{\partial}\log |s(x)|$$
if $s$ is non-zero.
It follows from the Poincaré--Lelong formula (\ref{equ_poincare-lelong}) that $c_1(\overline{\mathcal{L}})^{d-1}$ is a $(d-1,d-1)$-form of real type, that is $F_\infty^*c_1(\overline{\mathcal{L}})^{d-1}=(-1)^{d-1}c_1(\overline{\mathcal{L}})^{d-1}$.
Hence, the inner product (\ref{equ_innerproduct}) satisfies 
\begin{align}\label{equ_innerproduct-real}
\langle s_1,s_2\rangle\in \mathbb{R} \text{ for all } s_1,s_2\in H^0(\mathcal{X},\mathcal{L}^{\otimes p})_{\mathbb{R}}
\end{align}
as shown in \cite[Section 5.1.2]{Mor14}. In particular, $\langle\cdot,\cdot\rangle$ defines an inner product on the real vector space $H^0(\mathcal{X},\mathcal{L}^{\otimes p})_\mathbb{R}$ and we denote this inner product space by $H^0(\mathcal{X},\overline{\mathcal{L}}^{\otimes p})_\mathbb{R}$. 

It follows from \cite[Ch. I, Proposition 7.4.5]{GD60} and \cite[Proposition III.9.7.]{Har77} that $H^0(\mathcal{X},\mathcal{L}^{\otimes p})$ is free. Thus, $H^0(\mathcal{X},\mathcal{L}^{\otimes p})$ injects into $H^0(\mathcal{X},\overline{\mathcal{L}}^{\otimes p})_\mathbb{R}$ and its image is a lattice which we denote by $H^0(\mathcal{X},\overline{\mathcal{L}}^{\otimes p})$.
As $\overline{\mathcal{L}}$ is arithmetically ample, the proof of \cite[Proposition 5.41]{Mor14} shows, that there are real numbers $v\in (0,1)$ and $A>0$, such that $\lambda_{\mathbb{Z}}(H^0(\mathcal{X},\overline{\mathcal{L}}^{\otimes p}))\le A d_p v^p$
for all $p\ge 1$, where $d_p=\dim H^0(\mathcal{X},\overline{\mathcal{L}}^{p})_\mathbb{R}$.
Since we have $d_p\le Mp^{d-1}$ for some constant $M>0$ by the Hilbert--Serre theorem, we get 
\begin{align}\label{equ_successive-bound}
\lambda_{\mathbb{Z}}(H^0(\mathcal{X},\overline{\mathcal{L}}^{\otimes p}))\le B p^{d-1} v^p
\end{align}
for some constant $B>0$.

For later use, we define a decomposition of $H^0(\mathcal{X},\overline{\mathcal{L}}^{\otimes p})_\mathbb{R}$ into cells induced by a minimal $\mathbb{Z}$-basis of $H^0(\mathcal{X},\overline{\mathcal{L}}^{\otimes p})$. That means, for every $p\ge 1$ we fix a $\mathbb{Z}$-basis $s_{p,1},\dots, s_{p,d_p}$ of $H^0(\mathcal{X},\overline{\mathcal{L}}^{\otimes p})$ such that 
$$\max_{1\le j\le d_p}\{\|s_{p,j}\|\}=\lambda_{\mathbb{Z}}(H^0(\mathcal{X},\overline{\mathcal{L}}^{\otimes p})).$$ 
To every lattice point $x\in H^0(\mathcal{X},\overline{\mathcal{L}}^{p})$ we associate the following cell in $H^0(\mathcal{X},\overline{\mathcal{L}}^{p})_\mathbb{R}$
$$Q_{p,x}=\left\{\left.x+\sum_{j=1}^{d_p}a_js_{p,j}~\right|~a_j\in [0,1)\right\}\subseteq H^0(\mathcal{X},\overline{\mathcal{L}}^{p})_\mathbb{R}.$$
This induces a disjoint decomposition
\begin{align}\label{equ_cell-decomposition}
H^0(\mathcal{X},\overline{\mathcal{L}}^{\otimes p})_\mathbb{R}=\coprod_{x\in H^0(\mathcal{X},\overline{\mathcal{L}}^{\otimes p})} Q_{p,x}.
\end{align}
Note, that $\mathrm{Vol}(Q_{p,x})$ does not depend on the choice of the point $x\in H^0(\mathcal{X},\overline{\mathcal{L}}^{\otimes p})$ and we will in the following use the notation $\mathrm{Vol}(Q_{p,x})$ without specifying the point $x\in H^0(\mathcal{X},\overline{\mathcal{L}}^{\otimes p})$.

\section{The Distribution of Small Sections}\label{sec_distribution}
We study the distribution of the divisors in a sequence of small sections of the tensor powers of an arithmetically ample hermitian line bundle $\overline{\mathcal{L}}$ on a generically smooth projective arithmetic variety $\mathcal{X}$ in this section. By Section \ref{sec_distribution-complex} we know, that this is closely related to the vanishing of the limit of the normalized integrals of the logarithm of the norm of these sections. Thus, we first study the integral in this setting in Section \ref{sec_integral-small-sections}. In particular, we give the proof of Theorem \ref{thm_equidistribution}. In Section \ref{sec_distribution-small-sections} we deduce results on the distribution of the divisors. Especially, we prove Corollary \ref{cor_height-converges}. In Section \ref{sec_bogomolov} we discuss a result on the set of small section, which can be seen as an analogue of the generalized Bogomolov conjecture. More precisely, we show that for any given horizontal closed subvariety $\mathcal{Y}\subseteq \mathcal{X}$ the ratio of small sections of $\overline{\mathcal{L}}^{\otimes p}$, whose height is near to the height of their restriction to $\mathcal{Y}$, is asymptotically zero for $p\to\infty$. This proves Corollary \ref{cor_bogomolov}. We also discuss the Northcott property in this setting. Finally in Section \ref{sec_algebraic-numbers}, we study as an application the distribution of zero sets of integer polynomials in sequences with increasing degree. In particular, we give the proof of Proposition \ref{pro_equidistribution-polynomials} and Corollary \ref{cor_distribution-algebraic-numbers}.
\subsection{Integrals of the Logarithm of Small Sections}\label{sec_integral-small-sections}
In this section we prove Theorem \ref{thm_equidistribution}. In fact, we will prove the following even more general version of Theorem \ref{thm_equidistribution}. By $\mathcal{X}$ we always denote a generically smooth projective arithmetic variety of dimension $d$ and by $\overline{\mathcal{L}}$ an arithmetically ample hermitian line bundle on $\mathcal{X}$.
\begin{Thm}\label{thm_equidistribution-general}
	Let $\mathcal{X}$ be any generically smooth projective arithmetic variety and $\mathcal{Y}\subseteq \mathcal{X}$ any generically smooth arithmetic subvariety of dimension $e\ge 1$. Let $\overline{\mathcal{L}}$ be any arithmetically ample hermitian line bundle on $\mathcal{X}$. Further, let $(p_j)_{j\in \mathbb{Z}_{\ge 1}}$ be any increasing sequence of positive integers and $(r_j)_{j\in \mathbb{Z}_{\ge 1}}$ and $(r'_j)_{j\in \mathbb{Z}_{\ge 1}}$ two sequences of positive real numbers, such that
	$$\lim_{j\to \infty}r_j^{1/p_j}=\lim_{j\to \infty}{r'}_j^{1/p_j}=\tau\in[1,\infty).$$
	Denote by $B_{r_j}$ and $B_{r'_j}$ the balls in $H^0(\mathcal{X},\overline{\mathcal{L}}^{\otimes p_j})_\mathbb{R}$ of radius $r_j$ and $r'_j$ around the origin. Let $(K_j)_{j\in\mathbb{Z}_{\ge 1}}$ be any sequence of compact, convex and symmetric sets such that $\overline{B_{r'_j}}\subseteq K_j\subseteq \overline{B_{r_j}}$ for all $j$. If we set
	$$S_j=\left\{\left.s\in K_j\cap H^0(\mathcal{X},\overline{\mathcal{L}}^{\otimes p_j})~\right|~s|_{\mathcal{Y}}\neq 0\right\},$$
	then it holds
	$$\lim_{j\to \infty} \frac{1}{\#S_j}\sum_{s\in S_j}\frac{1}{p_j}\int_{\mathcal{Y}(\mathbb{C})}\left|\log|s|_{\overline{\mathcal{L}}^{\otimes p_j}}-p_j\log\tau\right|c_1(\overline{\mathcal{L}})^{e-1}=0.$$
\end{Thm}
In the following we shortly write
$$R_j=K_j\cap H^0(\mathcal{X},\overline{\mathcal{L}}^{\otimes p_j}).$$
Thus, $S_j$ is the subset of $R_j$ consisting of sections not identically vanishing on $\mathcal{Y}$.
The idea of the proof of the theorem is to apply the distribution result from Section \ref{sec_distribution-complex} to the probability measures associated to the Haar measures on the $K_j$'s. These satisfy Condition (B) by an application of Proposition \ref{pro_conditionB}. As in Section \ref{sec_lattice}, we will decompose $K_j$ into cells associated to the points in $R_j$ and use Proposition \ref{pro_integral} to deduce, that also the points in $R_j$ satisfy such a distribution result. Finally, the theorem will follow as $S_j$ has asymptotical density $1$ in $R_j$.
Let us first prove a helpful lemma. It shows that if a cell $Q_{p_j,x}$ in the decomposition (\ref{equ_cell-decomposition}) intersects $K_j$, it is most times already completely contained in $K_j$.
\begin{Lem}\label{lem_ratio-of-latticepoints}
	With the same notation and assumptions as in Theorem \ref{thm_equidistribution-general} it holds
	$$\lim_{j\to \infty}\frac{\#\{x\in H^0(\mathcal{X},\overline{\mathcal{L}}^{\otimes p_j})~|~\emptyset\subsetneq Q_{p_j,x}\cap K_j\subsetneq Q_{p_j,x}\}}{\# R_j}=0.$$
	Moreover, it holds $\lim_{j\to \infty}\frac{\mathrm{Vol}(K_j)}{\#R_j\cdot \mathrm{Vol}(Q_{p_j,x})}=1$.
\end{Lem}
\begin{proof}
	We shortly write $\lambda_{\mathbb{Z},p}=\lambda_{\mathbb{Z}}(H^0(\mathcal{X},\overline{\mathcal{L}}^{\otimes p}))$. By (\ref{equ_successive-bound}) it holds
	$$\limsup_{p\to \infty}\left(3pd_p^2\lambda_{\mathbb{Z},p}\right)^{1/p}\le \limsup_{p\to \infty} \left(3pM^2p^{2d-2} B p^{d-1} v^p\right)^{1/p}=v<1\le\lim_{j\to\infty}{r'_j}^{1/p_j}.$$
	Thus, there exists an integer $m\ge 0$ such that $3p_jd_{p_j}^2\lambda_{\mathbb{Z},p_j}\le r'_j$ for all $j\ge m$. We set $\mu_j=1-\frac{d_{p_j}\lambda_{\mathbb{Z},p_j}}{r'_j}$. It holds $\mu_j\in\left[\frac{2}{3},1\right]$ for all $j\ge m$.
	\begin{clm}
		Let $j\ge m$ and $x\in H^0(\mathcal{X},\overline{\mathcal{L}}^{\otimes p_j})$. 
		\begin{enumerate}[(i)]
			\item If $Q_{p_j,x}\not\subseteq K_j$, then $x\notin \mu_j K_j$.
			\item If $Q_{p_j,x}\cap K_j\neq \emptyset$ then $x\in \mu_j^{-1} K_j$. 
		\end{enumerate}
		
	\end{clm}
	\begin{proof}[Proof of Claim]
		Assume that $x\in \mu_j K_j$. Then $\mu_j^{-1}x\in K_j$. Since $K_j$ is convex, this implies $x+(1-\mu_j)y\in K_j$ for all $y\in K_j$. Since $\overline{B_{r'_j}}\subseteq K_j$, it follows that $x+y\in K_j$ if $\|y\|\le (1-\mu_j)r'_j=d_{p_j}\lambda_{\mathbb{Z},p_j}$. Now let $z=x+\sum_{k=1}^{d_{p_j}}a_ks_{p_j,k}$ be an arbitrary element of $Q_{p_j,x}$ and set $y=z-x$. Then 
		$$\|y\|\le \sum_{k=1}^{d_{p_j}}a_k\|s_{p_j,k}\|\le d_{p_j} \lambda_{\mathbb{Z},p_j}.$$
		Hence, $z=x+y\in K_j$ and thus, $Q_{p_j,x}\subseteq K_j$. This proves the (i).
		
		To prove (ii) let $z\in Q_{p_j,x}\cap K_j$. We can again write $z=x+y$ with $\|y\|\le d_{p_j} \lambda_{\mathbb{Z},p_j}$, which implies that $y\in (1-\mu_j)K_j$. Thus, there exists $y'\in K_j$ with $y=(1-\mu_j)y'$. Since $\mu_j\le 1$ and $K_j$ is symmetric, we also have $-\mu_j y'\in K_j$. Hence, we obtain
		$$\mu_j x=\mu_j z-\mu_j y=\mu_j z+(1-\mu_j)(-\mu_j y')\in K_j$$
		by the convexity of $K_j$. This is equivalent to $x\in\mu_j^{-1} K_j$.
	\end{proof}
	By the claim it is enough to show that
	\begin{align}\label{equ_claim-limit}
	\lim_{j\to\infty}\frac{\#(\mu_j K_j\cap H^0(\mathcal{X},\overline{\mathcal{L}}^{\otimes p_j}))}{\#(K_j\cap H^0(\mathcal{X},\overline{\mathcal{L}}^{\otimes p_j}))}=\lim_{j\to\infty}\frac{\#( K_j\cap H^0(\mathcal{X},\overline{\mathcal{L}}^{\otimes p_j}))}{\#(\mu_j^{-1}K_j\cap H^0(\mathcal{X},\overline{\mathcal{L}}^{\otimes p_j}))}=1
	\end{align}
	to get the first assertion of the lemma.
	In the following we always assume $j\ge m$.
	Since $3d_{p_j}\lambda_{\mathbb{Z},p_j}\le r'_j$ and $\mu_j\ge \frac{2}{3}$, we can apply Lemma \ref{lem_geometry-of-numbers} to get
	\begin{align}\label{equ_quotient}
	1\le \frac{\#(K_j\cap H^0(\mathcal{X},\overline{\mathcal{L}}^{\otimes p_j}))}{\#(\mu_j K_j\cap H^0(\mathcal{X},\overline{\mathcal{L}}^{\otimes p_j}))}\le \mu_j^{-d_{p_j}}\left(1+2d_{p_j}(1+\mu_j^{-1}){r'_j}^{-1}\lambda_{\mathbb{Z},p_j}\right)^{d_{p_j}}.
	\end{align}
	Let us compute the limit of the right hand side separately for both factors.
	Let $\epsilon>0$ be an arbitrary positive real number. Note that by $p_jd_{p_j}^2\lambda_{\mathbb{Z},p_j}\le r'_j$ we have $\mu_j\ge 1-\frac{\epsilon}{d_{p_j}}$ for $j\ge \max\{m,\epsilon^{-1}\}$. Thus, we compute
	$$1\ge \lim_{j\to \infty}\mu_j^{d_{p_j}}\ge \lim_{j\to\infty}\left(1-\frac{\epsilon}{d_{p_j}}\right)^{d_{p_j}}=e^{-\epsilon},$$
	where the last equality follows since $\lim_{p\to \infty}d_p=\infty$ as $\mathcal{L}$ is ample. Hence, we have $\lim_{j\to \infty}\mu_j^{-d_{p_j}}=1$.
	For the second factor note, that we have $d_{p_j}{r'_j}^{-1}\lambda_{\mathbb{Z},p_j}=1-\mu_j$. For sufficiently large $j$ we may assume that $\epsilon< \frac{d_{p_j}}{2}$. Therefore, we can compute
	\begin{align*}
	1&\le \lim_{j\to\infty} \left(1+2d_{p_j}(1+\mu_j^{-1}){r'_j}^{-1}\lambda_{\mathbb{Z},p_j}\right)^{d_{p_j}}=\lim_{j\to\infty} \left(1+2(\mu_j^{-1}-\mu_j)\right)^{d_{p_j}}\\
	&\le\lim_{j\to\infty}\left(1+2\left(\frac{1}{1-\frac{\epsilon}{d_{p_j}}}-1+\frac{\epsilon}{d_{p_j}}\right)\right)^{d_{p_j}}=\lim_{j\to\infty}\left(1+2\epsilon\left(\frac{1}{d_{p_j}-\epsilon}+\frac{1}{d_{p_j}}\right)\right)^{d_{p_j}}\\
	&\le \lim_{j\to\infty}\left(1+\frac{6\epsilon}{d_{p_j}}\right)^{d_{p_j}}=e^{6\epsilon}.
	\end{align*}
	Thus we get $\lim_{j\to\infty} \left(1+2d_{p_j}(1+\mu_j^{-1}){r'_j}^{-1}\lambda_{\mathbb{Z},p_j}\right)^{d_{p_j}}=1$. If we apply this to inequality (\ref{equ_quotient}) we get that the first limit in Equation (\ref{equ_claim-limit}) is $1$. To show that the second limit is also $1$, one may do the same computation for $K_j$ replaced by $\mu_j^{-1}K_j$. Note, that $\mu_j^{-1}K_j$ still satisfies the assumptions in Theorem \ref{thm_equidistribution-general} as $\lim_{j\to\infty}\mu_j=1$.
	
	To prove the second assertion of the lemma, 
	let us define 
	\begin{align*}
	A_j&=\#\{x\in H^0(\mathcal{X},\overline{\mathcal{L}}^{\otimes p_j})~|~Q_{p_j,x}\subseteq K_j\},\\
	B_j&=\#\{x\in H^0(\mathcal{X},\overline{\mathcal{L}}^{\otimes p_j})~|~Q_{p_j,x}\cap K_j\neq \emptyset\}.
	\end{align*}
	By the first assertion we have 
	$$\lim_{j\to\infty}\frac{A_j}{\#R_j}=\lim_{j\to\infty}\frac{B_j}{\#R_j}=1.$$
	We can bound the volume of $K_j$ by
	$$A_j\mathrm{Vol}(Q_{p_j,x})\le \mathrm{Vol}(K_j)\le B_j \mathrm{Vol}(Q_{p_j,x})$$
	Thus, we get 
	$$\lim_{j\to \infty}\frac{\mathrm{Vol}(K_j)}{\#R_j\cdot \mathrm{Vol}(Q_{p_j,x})}=1$$
	as claimed.
\end{proof}
As a first step to the proof of Theorem \ref{thm_equidistribution-general}, we prove that in a certain sense most of the sequences of sections satisfy a vanishing property as in Theorem \ref{thm_equidistribution-general}.
\begin{Pro}\label{pro_positive-density}
	We use the same notation and assumptions as in Theorem \ref{thm_equidistribution-general}.
	If $(T_{j})_{j\in\mathbb{Z}_{\ge1}}$ is any sequence of subsets $T_{j}\subseteq R_j$ such that 
	$$\liminf_{j\to \infty}\frac{\#T_{j}}{\#R_j}>0,$$
	then there exist an $m\in\mathbb{Z}_{\ge 1}$ and a sequence $(s_{p_j})_{j\in\mathbb{Z}_{\ge m}}$
	of sections $s_{p_j}\in T_{j}$, such that
	 	$$\lim_{j\to \infty} \frac{1}{p_j}\int_{\mathcal{Y}(\mathbb{C})}\left|\log|s_{p_j}|_{\overline{\mathcal{L}}^{\otimes p_j}}-p_j\log\tau\right|c_1(\overline{\mathcal{L}})^{e-1}=0.$$
	 In particular, we have $\lim_{j\to\infty}\frac{\#S_{j}}{\#R_j}=1$.
\end{Pro}
Before we give the proof of the proposition, let us reduce the proofs of Theorem \ref{thm_equidistribution-general} and Proposition \ref{pro_positive-density} to the case $\tau=0$ by the following lemma.
\begin{Lem}
	If Theorem \ref{thm_equidistribution-general} holds for $\tau=1$, then it holds for any $\tau\in[1,\infty)$. The same holds for Proposition \ref{pro_positive-density}.
\end{Lem}
\begin{proof}
	Assume that Theorem \ref{thm_equidistribution-general} (respectively Proposition \ref{pro_positive-density}) holds for $\tau=1$ and consider an arbitrary $\tau\in [1,\infty)$. The idea is to apply the known case to the hermitian line bundle $\overline{\mathcal{L}}(\log\tau)$. By $\log \tau\ge 0$ the hermitian line bundle $\overline{\mathcal{L}}(\log\tau)$ is still arithmetically ample. Note that
	\begin{align}\label{equ_metric-change}
	\log|s|_{\overline{\mathcal{L}}(\log\tau)^{\otimes p}}=\log\sqrt{e^{-2 p\log \tau}h^{\otimes p}(s,s)}=\log |s|_{\overline{\mathcal{L}}^{\otimes p}}-p\log\tau,
	\end{align}
	where $h$ is the hermitian metric on $\mathcal{L}({\mathbb{C}})$ and $s\in H^0(\mathcal{X},\mathcal{L}^{\otimes p})_{\mathbb{R}}$.
	Thus, we get for the sets of sections of bounded $L^2$-norm
	\begin{align*}
		&\left\{s\in H^0(\mathcal{X},\overline{\mathcal{L}}^{\otimes p})_{\mathbb{R}}~\left|~\|s\|^2=\int_{\mathcal{X}(\mathbb{C})}|s(x)|_{\overline{\mathcal{L}}^{\otimes p}}^2c_1(\overline{\mathcal{L}})^{d-1}\le r^2\right.\right\}\\
		&=\left\{s\in H^0(\mathcal{X},\overline{\mathcal{L}}(\log\tau)^{\otimes p})_{\mathbb{R}}~\left|~\|s\|^2=\int_{\mathcal{X}(\mathbb{C})}|s(x)|_{\overline{\mathcal{L}}(\log\tau)^{\otimes p}}^2c_1(\overline{\mathcal{L}})^{d-1}\le \frac{r^2}{\tau^{2p}}\right.\right\}
	\end{align*}	
	as subsets of the linear space $H^0(\mathcal{X},\mathcal{L}^{\otimes p})_{\mathbb{R}}$.
	Hence, if we replace $\overline{\mathcal{L}}$ by $\overline{\mathcal{L}}(\log\tau)$ in Theorem \ref{thm_equidistribution-general} (respectively Proposition \ref{pro_positive-density}), we have to replace the balls $B_{r_j}$ and $B_{r'_j}$ by $B_{r_j/\tau^{p_j}}$ and $B_{r'_j/\tau^{p_j}}$ if we still want the same assertion. But
	$$\lim_{j\to \infty}\left(\frac{r_j}{\tau^{p_j}}\right)^{1/p_j}=\frac{\lim_{j\to \infty}r_j^{1/p_j}}{\tau}=\frac{\tau}{\tau}=1$$
	and similar $\lim_{j\to \infty}\left(\frac{r'_j}{\tau^{p_j}}\right)^{1/p_j}=1$. Thus, by the hypothesis we can apply Theorem \ref{thm_equidistribution-general} (respectively Proposition \ref{pro_positive-density}) to the line bundle $\overline{\mathcal{L}}(\log\tau)$. By Equation (\ref{equ_metric-change}) this gives exactly the assertion of Theorem \ref{thm_equidistribution-general} (respectively Proposition \ref{pro_positive-density}) for general $\tau\in[1,\infty)$.
\end{proof}
Thus, we may and will assume $\tau=1$ in the following. We continue with the proof of Proposition \ref{pro_positive-density}.
\begin{proof}[Proof of Proposition \ref{pro_positive-density}]
	We choose $m\in\mathbb{Z}_{\ge 0}$, such that $T_j\neq \emptyset$ for all $j\ge m$.
	We decompose every set $T_{j}$ into the two disjoint sets
	$$T'_{j}=\{x\in T_{j}~|~Q_{p_j,x}\subseteq K_{j}\},\qquad T''_{j}=\{x\in T_{j}~|~Q_{p_j,x}\not\subseteq K_{j}\}.$$
	By Lemma \ref{lem_ratio-of-latticepoints}, it holds $\lim_{j\to \infty} \frac{ \#T''_{j}}{\#R_{j}}=0$. Thus, by $\liminf_{j\to \infty}\frac{\# T_{j}}{\# R_{j}}>0$ we also have $\liminf_{j\to \infty}\frac{\# T'_{j}}{\# S_{j}}>0$. Replacing $T_{j}$ by $T'_{j}$ we may assume that $Q_{p_j,x}\subseteq K_{j}$ for all $x\in T_{j}$.
	
	By the property in (\ref{equ_innerproduct-real}) the real vector subspace $H^0(\mathcal{X},\overline{\mathcal{L}}^{\otimes p_j})_\mathbb{R}$ of the complex vector space $H^0(\mathcal{X}(\mathbb{C}),\mathcal{L}(\mathbb{C})^{\otimes p_j})$ satisfies the assumptions of Proposition \ref{pro_conditionB}. That means that the probability measures
	$$\sigma_{p_j}=\frac{1}{\mathrm{Vol}(K_{j})}\lambda_{p_j}|_{K_{j}}$$
	for the normed Haar measure $\lambda_{p_j}$ on $H^0(\mathcal{X},\overline{\mathcal{L}}^{\otimes p_j})_\mathbb{R}$ satisfy Condition (B) with $\lim_{j\to \infty} \frac{C_{p_j}}{p_j}=0$. By Lemma \ref{lem_integral-zero} this implies, that 
	$$\lim_{j\to\infty}\frac{1}{\mathrm{Vol}(K_{j})}\int_{s\in K_{j}}\frac{1}{p_j}\int_{\mathcal{Y}(\mathbb{C})}\left|\log|s|_{\overline{\mathcal{L}}^{\otimes p_{j}}}\right|c_1(\overline{\mathcal{L}})^{e-1}d\lambda_{p_j}(s)=0.$$
	Applying the decomposition in (\ref{equ_cell-decomposition}) we obtain
	$$\lim_{j\to \infty}\frac{1}{\mathrm{Vol}(K_{j})}\sum_{x\in H^0(\mathcal{X},\overline{\mathcal{L}}^{\otimes p_j})} \int_{Q_{p_j,x}\cap K_{j}}\frac{1}{p_j}\int_{\mathcal{Y}(\mathbb{C})}\left|\log |s|_{\overline{\mathcal{L}}^{\otimes p_{j}}}\right|c_1(\overline{\mathcal{L}})^{e-1}d\lambda_{p_j}(s)=0.$$
	As every term in the sum is non-negative, also any limit of a partial sums has to be zero. In particular, we obtain
	$$\lim_{j\to \infty}\frac{1}{\#R_j\mathrm{Vol}(Q_{p_j,x})}\sum_{x\in T_{j}} \int_{Q_{p_j,x}}\frac{1}{p_j}\int_{\mathcal{Y}(\mathbb{C})}\left|\log |s|_{\overline{\mathcal{L}}^{\otimes p_{j}}}\right|c_1(\overline{\mathcal{L}})^{e-1}d\lambda_{p_j}(s)=0,$$
	where we used Lemma \ref{lem_ratio-of-latticepoints} to replace the prefactor and the assumption $Q_{p_j,x}\subseteq K_{p_j}$ for $x\in T_{p_j}$ to replace the domain of integration. By $\liminf_{j\to \infty}\frac{\#T_{j}}{\#R_j}>0$, the above equation implies
	$$\lim_{j\to \infty}\frac{1}{\#T_{j}}\sum_{x\in T_{j}} \frac{1}{\mathrm{Vol}(Q_{p_j,x})}\int_{Q_{p_j,x}}\frac{1}{p_j}\int_{\mathcal{Y}(\mathbb{C})}\left|\log |s|_{\overline{\mathcal{L}}^{\otimes p_{j}}}\right|c_1(\overline{\mathcal{L}})^{e-1}d\lambda_{p_j}(s)=0.$$
	This guarantees, that we can choose a sequence $(x_{p_j})_{j\in\mathbb{Z}_{\ge m}}$ with $x_{p_j}\in T_{j}$ and a sequence of sections $(s_{p_j})_{j\in \mathbb{Z}_{\ge m}}$ with $s_{p_j}\in Q_{p_j,x_{p_j}}$, such that 
	\begin{align}\label{equ_{p_j}roperty4}
	\lim_{j\to \infty}\frac{1}{p_j}\int_{\mathcal{Y}(\mathbb{C})}\left|\log |s_{p_j}|_{\overline{\mathcal{L}}^{\otimes p_{j}}}\right|c_1(\overline{\mathcal{L}})^{e-1}=0.
	\end{align}
	We will apply Proposition \ref{pro_integral} to the sequences $(s_{p_j})_{p\in \mathbb{Z}_{\ge m}}$ and $(x_{p_j})_{p\in \mathbb{Z}_{\ge m}}$. Since $s_{p_j}$ is contained in $K_{j}\subseteq \overline{B_{r_{j}}}$ with $\lim_{j\to \infty}r_{j}^{1/p_j}=\tau=1$, it holds
	\begin{align}\label{equ_{p_j}roperty12}
	\limsup_{j\to \infty} \|s_{p_j}\|^{1/p_j}\le 1.
	\end{align}
	Since $s_{p_j}\in Q_{p_j,x_{p_j}}$, we can write it in the form $s_{p_j}=x_{p_j}+\sum_{k=1}^{d_{p_j}}a_ks_{p_j,k}$ for some $a_k\in [0,1)$.
	Thus
	$$\|s_{p_j}-x_{p_j}\|=\left\|\sum_{k=1}^{d_{p_j}}a_ks_{p_j,k}\right\|\le d_{p_j}\lambda_{\mathbb{Z}}(H^0(\mathcal{X},\overline{\mathcal{L}}^{\otimes p_j}))\le M\cdot B\cdot p_j^{2d-2}v^{p_j},$$
	where the last inequality follows by the Hilbert--Serre theorem for some constant $M\in\mathbb{Z}$ and by inequality (\ref{equ_successive-bound}). From this we deduce
	\begin{align}\label{equ_{p_j}roperty3}
	\limsup_{j\to \infty}\|s_{p_j}-x_{p_j}\|^{1/p_j}\le \limsup_{j\to\infty} \left(M\cdot B\cdot p_j^{2d-2}v^{p_j}\right)^{1/p_j}=v<1.
	\end{align}
	By (\ref{equ_{p_j}roperty12}), (\ref{equ_{p_j}roperty3}) and (\ref{equ_{p_j}roperty4}) the assumptions (i)-(iii) of Proposition \ref{pro_integral} are satisfied, such that we obtain
	$$\lim_{j\to \infty}\frac{1}{p_j}\int_{\mathcal{Y}(\mathbb{C})}\left|\log |x_{p_j}|_{\overline{\mathcal{L}}^{\otimes p_{j}}}\right|c_1(\overline{\mathcal{L}})^{e-1}=0.$$
	This completes the proof of the first assertion of the proposition. 
	
	To show the second assertion, we assume that $\lim_{j\to \infty}\frac{\#S_j}{\#R_j}\neq 1$. Thus, there exists some sequence $(j_k)_{k\in\mathbb{Z}_{\ge 1}}$ such that $\lim_{k\to \infty}\frac{\#(R_{j_k}\setminus S_{j_k})}{\#R_{j_k}}>0$. By the first assertion of the lemma, we can find a sequence $(s_k)_{k\in\mathbb{Z}_{\ge 1}}$ of sections $s_k\in R_{j_k}\setminus S_{j_k}$ with 
	$$\lim_{k\to \infty}\frac{1}{p_{j_k}}\int_{\mathcal{Y}(\mathbb{C})}\left|\log |s_k|_{\overline{\mathcal{L}}^{\otimes p_{j_k}}}\right| c_1(\overline{\mathcal{L}})^{e-1}=0.$$
	But as $s_k|_{\mathcal{Y}}=0$ for all $k$, this is absurd. Hence, we must have $\lim_{j\to \infty}\frac{\#S_j}{\#R_j}=1$.
\end{proof}
Before we proof Theorem \ref{thm_equidistribution-general}, we show in the following lemma that the integral of
$\frac{1}{p_j}\left|\log\|s\|_{\overline{\mathcal{L}}^{\otimes p_j}}\right|$
is uniformly bounded for $s\in S_j$.
\begin{Lem}\label{lem_integral-upper-bound}
	Let $S_j$ be as in Theorem \ref{thm_equidistribution-general}.
	There is a constant $C_3$ only depending on $\mathcal{X}$, $\mathcal{Y}$, $\overline{\mathcal{L}}$ and $(r_j)_{j\in\mathbb{Z}_{\ge 1}}$, but independent of $j$, such that
	$$\frac{1}{p_j}\int_{\mathcal{Y}(\mathbb{C})}\left|\log|s|_{\overline{\mathcal{L}}^{\otimes p_j}}\right| c_1(\overline{\mathcal{L}})^{e-1}\le C_3$$
	for all sections $s\in S_j$.
\end{Lem}
\begin{proof}
	If $\|s\|_{\sup}\le 1$, we have $\left|\log|s|_{\overline{\mathcal{L}}^{\otimes p_j}}\right|=-\log|s|_{\overline{\mathcal{L}}^{\otimes p_j}}$ such that the assertion follows from Lemma \ref{lem_bound-of-integral-of-section} with $C_3=(\overline{\mathcal{L}}|_{\mathcal{Y}}^{e})$. Thus, we may assume $\|s\|_{\sup}> 1$. Then we can bound by the triangle inequality
	\begin{align}\label{equ_integral-bound-arithmetic}
	&\frac{1}{p_j}\int_{\mathcal{Y}(\mathbb{C})}\left|\log|s|_{\overline{\mathcal{L}}^{\otimes p_j}}\right| c_1(\overline{\mathcal{L}})^{e-1}\\
	&\le -\frac{1}{p_j}\int_{\mathcal{Y}(\mathbb{C})}\log\left|\frac{s}{\|s\|_{\sup}}\right|_{\overline{\mathcal{L}}^{\otimes p_j}} c_1(\overline{\mathcal{L}})^{e-1}+\log\|s\|_{\sup}\frac{1}{p_j}\int_{\mathcal{Y}(\mathbb{C})}c_1(\overline{\mathcal{L}})^{e-1}\nonumber\\
	&= -\frac{1}{p_j}\int_{\mathcal{Y}(\mathbb{C})}\log|s|_{\overline{\mathcal{L}}^{\otimes p_j}} c_1(\overline{\mathcal{L}})^{e-1}+2\log\|s\|^{1/p_j}_{\sup}\int_{\mathcal{Y}(\mathbb{C})}c_1(\overline{\mathcal{L}})^{e-1}\nonumber\\
	&\le (\overline{\mathcal{L}}|_{\mathcal{Y}}^{e})+2\log\|s\|^{1/p_j}_{\sup}\int_{\mathcal{Y}(\mathbb{C})}c_1(\overline{\mathcal{L}})^{e-1},\nonumber
	\end{align}
	where the last inequality follows by Lemma \ref{lem_bound-of-integral-of-section}. By $s\in S_j\subseteq B_{r_j}$ we have $\|s\|\le r_j$ such that inequality (\ref{equ_supl2}) implies $\|s\|_{\sup}^{1/p_j}\le (r_jC_1p_j^{d-1})^{1/p_j}$. By $\lim_{j\to \infty} r_j^{1/p_j}=\tau$ we also have $\lim_{j\to\infty}(r_jC_1p_j^{d-1})^{1/p_j}=\tau$. Thus, there exists a constant $C'_3\ge 1$ with $(r_jC_1p_j^{d-1})^{1/p_j}\le C'_3$ for all $j\ge 1$. Applying this to (\ref{equ_integral-bound-arithmetic}), we get
	$$\frac{1}{p_j}\int_{\mathcal{Y}(\mathbb{C})}\left|\log|s|_{\overline{\mathcal{L}}^{\otimes p_j}}\right| c_1(\overline{\mathcal{L}})^{e-1}\le (\overline{\mathcal{L}}|_{\mathcal{Y}}^{e})+2(\log C'_3)\int_{\mathcal{Y}(\mathbb{C})}c_1(\overline{\mathcal{L}})^{e-1}.$$
	Setting $C_3=(\overline{\mathcal{L}}|_{\mathcal{Y}}^{e})+2(\log C'_3)\int_{\mathcal{Y}(\mathbb{C})}c_1(\overline{\mathcal{L}})^{e-1}$ we get the assertion of the lemma. Note that this constant is also large enough to cover the case $\|s\|_{\sup}\le 1$.
\end{proof}
Now we are able to give the proof of Theorem \ref{thm_equidistribution-general}.
\begin{proof}[Proof of Theorem \ref{thm_equidistribution-general}]
	The limit of a sequence of non-negative values is zero if and only if its limit superior is zero. Thus, we may replace $\lim_{j\to \infty}$ by $\limsup_{j\to \infty}$.
	To simplify notations we write
	$$F_j(s)=\frac{1}{p_j}\int_{\mathcal{Y}(\mathbb{C})}\left|\log|s|_{\overline{\mathcal{L}}^{\otimes p_j}}\right|c_1(\overline{\mathcal{L}})^{e-1}$$
	for every $s\in S_j$ in the following.
	Let us assume that 
	\begin{align}\label{equ_assumption}
	\limsup_{j\to \infty} \frac{1}{\#S_j}\sum_{s\in S_j}F_j(s)>0.
	\end{align}
	Let $\mu>0$ be a positive limit point of this sequence and $(j_k)_{k\in\mathbb{Z}_{\ge 1}}$ an increasing sequence such that
	$$\lim_{k\to \infty} \frac{1}{\#S_{j_k}}\sum_{s\in S_{j_k}}F_{j_k}(s)=\mu.$$
	Then there exists a number $N\in \mathbb{Z}$ such that
	$$\frac{1}{\#S_{j_k}}\sum_{s\in S_{j_k}}F_{j_k}(s)>\frac{\mu}{2}$$
	for all $k\ge N$. For any $k$ we decompose $S_{j_k}=S'_{j_k}\cup S''_{j_k}$ into the disjoint subsets
	$$S'_{j_k}=\left\{s\in S_{j_k}~|~F_{j_k}(s)\le \frac{\mu}{4}\right\},\qquad S''_{j_k}=\left\{s\in S_{j_k}~|~F_{j_k}(s)>\frac{\mu}{4}\right\}.$$
	To estimate the cardinalities of these sets, we compute by Lemma \ref{lem_integral-upper-bound}
	\begin{align*}
		\frac{\mu}{2}&<\frac{1}{\#S_{j_k}}\sum_{s\in S'_{j_k}}F_{j_k}(s)+\frac{1}{\#S_{j_k}}\sum_{s\in S''_{j_k}}F_{j_k}(s)\\
		&\le \frac{\#S'_{j_k}}{\#S_{j_k}}\cdot\frac{\mu}{4}+\frac{\#S''_{j_k}}{\#S_{j_k}}\cdot C_3\\
		&\le \frac{\mu}{4}+\frac{\#S''_{j_k}}{\#S_{j_k}}\cdot \left( C_3-\frac{\mu}{4}\right)
	\end{align*} 
	for all $k\ge N$. Thus, we get
	$$\frac{\#S''_{j_k}}{\#S_{j_k}}>\frac{\mu}{4\left(C_3-\frac{\mu}{4}\right)}>0.$$
	As $\frac{\mu}{2\left(C_3-\frac{\mu}{4}\right)}$ does not depend on $k$, we have $\liminf_{k\to \infty}\frac{\#S''_{j_k}}{\#S_{j_k}}>0$. But now Proposition \ref{pro_positive-density} implies that there are an $m\in\mathbb{Z}_{\ge 1}$ and a sequence $(s_{p_{j_k}})_{k\in\mathbb{Z}_{\ge m}}$ of sections $s_{p_{j_k}}\in S''_{j_k}$ such that 
	$$\lim_{k\to \infty} F_{j_k}(s_{p_{j_k}})=0.$$
	But this contradicts, that $F_{j_k}(s)>\frac{\mu}{4}$ for all $k\ge 1$ and all $s\in S''_{j_k}$. Hence, our assumption (\ref{equ_assumption}) has to be wrong and we must have
	$$\limsup_{j\to \infty} \frac{1}{\#S_j}\sum_{s\in S_j}\frac{1}{p_j}\int_{\mathcal{Y}(\mathbb{C})}\left|\log|s|_{\overline{\mathcal{L}}^{\otimes p_j}}\right|c_1(\overline{\mathcal{L}})^{e-1}=0$$
	as claimed.
\end{proof}
We can now present the following theorem as a special case of Theorem \ref{thm_equidistribution-general}.

\begin{Thm}[= Theorem \ref{thm_equidistribution}]\label{thm_equidistribution-text}
	Let $\mathcal{X}$ be any generically smooth projective arithmetic variety and $\mathcal{Y}\subseteq\mathcal{X}$ any generically smooth arithmetic subvariety of dimension $e\ge 1$. Let $\overline{\mathcal{L}}$ be any arithmetically ample hermitian line bundle on $\mathcal{X}$ and $(r_p)_{p\in \mathbb{Z}_{\ge 1}}$ any sequence of positive real numbers satisfying $\lim_{p\to \infty}r_p^{1/p}=\tau\in[1,\infty)$. If we write $$S_p=\left\{\left. s\in \widehat{H}^0_{\le r_{p}}\left(\mathcal{X},\overline{\mathcal{L}}^{\otimes p}\right)~\right|~s|_{\mathcal{Y}}\neq 0\right\},$$
	then it holds
	$$\lim_{p\to \infty} \frac{1}{\# S_p}\sum_{s\in S_p}\frac{1}{p}\int_{\mathcal{Y}(\mathbb{C})}\left|\log|s|_{\overline{\mathcal{L}}^{\otimes p}}-p\log\tau\right|c_1(\overline{\mathcal{L}})^{e-1}=0.$$
\end{Thm}
\begin{proof}
	To deduce the theorem from Theorem \ref{thm_equidistribution-general} we set
	$$K_p=\left.\left\{s\in H^0(\mathcal{X},\overline{\mathcal{L}}^{\otimes p})_{\mathbb{R}}~\right|~\|s\|_{\sup}\le r_p\right\}.$$
	With this choice, the definition of $S_p$ in this theorem coincides with definition of $S_p$ in Theorem \ref{thm_equidistribution-general}.	Let us show, that $K_p$ satisfies the assumptions of Theorem \ref{thm_equidistribution-general}. Since $\|\cdot\|_{\sup}$ is a norm on $H^0(\mathcal{X},\overline{\mathcal{L}}^{\otimes p})_{\mathbb{R}}$, the set $K_p$ is compact, convex and symmetric for all $p$. By inequality (\ref{equ_supl2}) we obtain that
	$$\overline{B_{\frac{r_p}{C_1p^{d-1}}}}\subseteq K_p\subseteq \overline{B_{r_p}}.$$
	As we have $\lim_{p\to\infty} r_p^{1/p}=\tau$ for some $\tau\in[1,\infty)$, we also get
	$$\lim_{p\to \infty}\left(\frac{r_p}{C_1p^{d-1}}\right)^{1/p}=\lim_{p\to\infty} r_p^{1/p}=\tau.$$
	Thus, $K_p$ satisfies all assumptions in Theorem \ref{thm_equidistribution-general}. The theorem follows.
\end{proof}
For future research let us ask the following question.
\begin{Que}\label{que_integral}
	Is there an analogue of Theorem \ref{thm_equidistribution-general} if one equips $\overline{\mathcal{L}}$ also with a metric $|\cdot|$ on the Berkovich space $\mathcal{X}_{\mathbb{Q}_p}^{\mathrm{an}}$ for some prime $p$ and one considers the integral
	$$\int_{\mathcal{Y}_{\mathbb{Q}_p}^{\mathrm{an}}}\left|\log|s|\right|c_1(\overline{\mathcal{L}})_p^{e-1},$$
	where $c_1(\overline{\mathcal{L}})_p$ denotes the measure on $\mathcal{X}_{\mathbb{Q}_p}^{\mathrm{an}}$ associated to the metric $|\cdot|$ on $\overline{\mathcal{L}}$ and defined by Chambert-Loir \cite{Cha06}. As we are interested in the average of this integral over finite subsets of $H^0(\mathcal{X},\mathcal{L})$, we will still need the metric on $\mathcal{X}(\mathbb{C})$. Finally, one may also ask for an adelic analogue, where $L$ is a line bundle on a smooth projective variety $X$ over $\mathbb{Q}$ equipped with a metric $|\cdot|_v$ on $X_{\mathbb{Q}_v}^{\mathrm{an}}$ for every place $v$ of $\mathbb{Q}$. In this case one can also define a finite set $\widehat{H}^0(X,\overline{L})$ of global sections which are small with respect to the metric $|\cdot|_v$ for every $v$. Do we have
	$$\lim_{p\to \infty}\frac{1}{\#\widehat{H}^0(X,\overline{L}^{\otimes p})}\sum_{s\in \widehat{H}^0(X,\overline{L}^{\otimes p})\setminus\{0\}}\frac{1}{p}\int_{X_{\mathbb{Q}_v}^{\mathrm{an}}}\left|\log|s|_v\right|c_1(\overline{L})_v^{\dim X}=0$$
	for every place $v$ if $\overline{L}$ is an ample adelic line bundle? We refer to \cite{YZ21} for more details on adelic line bundles.
\end{Que}

\subsection{Distribution of Small Sections}\label{sec_distribution-small-sections}
In this section we deduce a result on the distribution of the divisors of sections in $\widehat{H}_{\le r_p}^0(\mathcal{X},\overline{\mathcal{L}}^{\otimes p})$ for $p\to \infty$ from Theorem \ref{thm_equidistribution-general}. As a special case we will prove Corollary \ref{cor_height-converges}. But first, we give the following more general result.
\begin{Cor}\label{cor_equidistribution-general}
	Let $\mathcal{X}$ be any generically smooth projective arithmetic variety, $\mathcal{Y}\subseteq \mathcal{X}$ any generically smooth arithmetic subvariety of dimension $e\ge 2$ and $\overline{\mathcal{L}}$ any arithmetically ample hermitian line bundle on $\mathcal{X}$. Further, let $(p_j)_{j\in\mathbb{Z}_{\ge 1}}$ be any increasing sequence of positive integers and $(r_j)_{j\in \mathbb{Z}_{\ge 1}}$ and $(r'_j)_{j\in \mathbb{Z}_{\ge 1}}$ two sequences of positive real numbers, such that
	$$\lim_{j\to \infty}r_j^{1/p_j}=\lim_{j\to \infty}{r'}_j^{1/p_j}=\tau\in[1,\infty).$$
	Denote by $B_{r_j}$ and $B_{r'_j}$ the balls in $H^0(\mathcal{X},\overline{\mathcal{L}}^{\otimes p_j})_\mathbb{R}$ of radius $r_j$ and $r'_j$ around the origin. Let $(K_j)_{j\in\mathbb{Z}_{\ge 1}}$ be any sequence of compact, convex and symmetric sets such that $\overline{B_{r'_j}}\subseteq K_j\subseteq \overline{B_{r_j}}$ for all $j$. If we set 
	$$S_j=\left\{\left.s\in K_j\cap H^0(\mathcal{X},\overline{\mathcal{L}}^{\otimes p_j})~\right|~s|_{\mathcal{Y}}\neq 0\right\},$$ then for every $(e-2,e-2)$ $C^0$-form $\Phi$ on $\mathcal{Y}(\mathbb{C})$ it holds
	$$\lim_{j\to\infty}\frac{1}{\#S_j}\sum_{s\in S_j}\left|\frac{1}{p_j}\int_{\mathrm{div}(s|_{\mathcal{Y}})(\mathbb{C})}\Phi-\int_{\mathcal{Y}(\mathbb{C})}\Phi\wedge c_1(\overline{\mathcal{L}})\right|=0.$$
\end{Cor}
\begin{proof}
	Fix any $(e-2,e-2)$ $C^0$-form $\Phi$ on $\mathcal{Y}(\mathbb{C})$.
	As in the proof of Lemma \ref{lem_equidistribution} there is an $A\in\mathbb{R}$ such that $\left|\frac{\partial\overline{\partial}\Phi}{\pi i}\right|\le Ac_1(\overline{\mathcal{L}})^{e-1}$. Thus, using Equations (\ref{equ_equidistribution}) and (\ref{equ_metric-change}) and Theorem \ref{thm_equidistribution-general} we can bound
	\begin{align*}
	&\lim_{j\to\infty}\frac{1}{\#S_j}\sum_{s\in S_j}\left|\frac{1}{p_j}\int_{\mathrm{div}(s|_{\mathcal{Y}})(\mathbb{C})}\Phi-\int_{\mathcal{Y}(\mathbb{C})}\Phi\wedge c_1(\overline{L})\right|\\
	&\le \lim_{j\to\infty}\frac{1}{\#S_j}\sum_{s\in S_j} \frac{A}{p_j}\int_{\mathcal{Y}(\mathbb{C})}\left|\log|s|_{\overline{\mathcal{L}}(\log\tau)^{\otimes p_j}}\right|c_1(\overline{\mathcal{L}})^{e-1}\\
	&= A\cdot\lim_{j\to\infty}\frac{1}{\#S_j}\sum_{s\in S_j} \frac{1}{p_j}\int_{\mathcal{Y}(\mathbb{C})}\left|\log|s|_{\overline{\mathcal{L}}^{\otimes p_j}}-p_j\log \tau\right|c_1(\overline{\mathcal{L}})^{e-1}=0.
	\end{align*}
	This proves the corollary.
\end{proof}
In the special case where $K_j$ is defined by balls of the $\sup$-norm on $H^0(\mathcal{X},\mathcal{L}^{\otimes p_j})_{\mathbb{R}}$ we would like to get rid of the convergence of the sequences $r_j^{1/p_j}$ and ${r'}_j^{1/p_j}$. For this purpose, let us prove the following easy lemma about convergence.
\begin{Lem}\label{lem_convergence}
	Let $(x_n)_{n\in\mathbb{Z}_{\ge1}}$ be any sequence of real numbers with
	$$-\infty<\liminf_{n\to\infty} x_n \le \limsup_{n\to\infty} x_n<\infty$$
	and $F\colon \mathbb{R}\to \mathbb{R}$ any function. If we have $\lim_{j\to \infty} F(x_{n_j})=\alpha$ for some fixed value $\alpha\in \mathbb{R}\cup\{-\infty,\infty\}$ and for any subsequence $(n_j)_{j\in\mathbb{Z}_{\ge 1}}$, for which $(x_{n_j})_{j\in\mathbb{Z}_{\ge 1}}$ converges in $\mathbb{R}$, then we also have $\lim_{n\to \infty} F(x_n)=\alpha$.
\end{Lem}
\begin{proof}
	Let $\beta\in \mathbb{R}\cup\{-\infty,\infty\}$ be a limit point of the sequence $(F(x_n))_{n\in\mathbb{Z}_{\ge 1}}$. Then there exists a subsequence $(n_j)_{j\in\mathbb{Z}_{\ge 1}}$ of $\mathbb{Z}_{\ge 1}$, such that $\lim_{j\to \infty}F(x_{n_j})=\beta$. As $-\infty<\liminf_{n\to\infty} x_n \le \limsup_{n\to\infty} x_n<\infty$, we can again choose a subsequence $(n'_j)_{j\in\mathbb{Z}_{\ge 1}}$ of $(n_j)_{j\in\mathbb{Z}_{\ge 1}}$, such that $(x_{n'_j})_{j\in\mathbb{Z}_{\ge 1}}$ converges in $\mathbb{R}$. By our assumptions we get $\lim_{j\to \infty}F(x_{n'_j})=\alpha$. Thus, we must have $\beta=\alpha$, such that $\alpha$ is the only limit point of $(F(x_n))_{n\in\mathbb{Z}_{\ge 1}}$ in $\mathbb{R}\cup\{-\infty,\infty\}$. Hence, $\lim_{n\to \infty} F(x_n)=\alpha$.
\end{proof}

The following corollary shows that the condition in Proposition \ref{pro_equidistribution-individual-text} for sequences of sections to equidistribute is generically satisfied. Moreover, we get an equidistribution result for almost all sections in $\widehat{H}_{\le r_p}^0(\mathcal{X},\overline{\mathcal{L}}^{\otimes p})$ for $p\to \infty$. In particular, they are equidistributed on average. Here, we do not have to assume that $r_p^{1/p}$ converges. 
\begin{Cor}[= Corollary \ref{cor_height-converges}]\label{cor_height-converges-text}
	Let $\mathcal{X}$ be any generically smooth projective arithmetic variety and $\mathcal{Y}\subseteq \mathcal{X}$ any generically smooth arithmetic subvariety of dimension $e\ge 2$. Let $\overline{\mathcal{L}}$ be any arithmetically ample hermitian line bundle on $\mathcal{X}$ and $(r_p)_{p\in \mathbb{Z}_{\ge 1}}$ any sequence with $r_p\in \mathbb{R}_{>0}$. We write $$S_p=\left\{\left.s\in\widehat{H}^0_{\le r_{p}}\left(\mathcal{X},\overline{\mathcal{L}}^{\otimes p}\right)~\right|~s|_{\mathcal{Y}}\neq 0\right\}.$$
	\begin{enumerate}[(i)]
		\item If $r_p^{1/p}$ converges to a value $\tau\in [1,\infty)$, then it holds
		$$\lim_{p\to \infty} \frac{1}{\# S_p}\sum_{s\in S_p}\left| h_{\overline{\mathcal{L}}}(\mathcal{Y})+\log\tau-h_{\overline{\mathcal{L}}}(\mathrm{div}(s)\cdot\mathcal{Y})\right|=0.$$
		\item If $1\le\liminf_{p\to \infty}r_p^{1/p}\le \limsup_{p\to \infty}r_p^{1/p}<\infty$, then for every $(e-2,e-2)$ $C^0$-form $\Phi$ on $\mathcal{Y}(\mathbb{C})$ it holds
		$$\lim_{p\to \infty}\frac{1}{\# S_p}\sum_{s\in S_p}\left|\frac{1}{p}\int_{\mathrm{div}(s|_{\mathcal{Y}})(\mathbb{C})}\Phi-\int_{\mathcal{Y}(\mathbb{C})}\Phi\wedge c_1(\overline{\mathcal{L}})\right|=0.$$
		In particular, we have
		$$\lim_{p\to \infty}\frac{1}{\# S_p}\sum_{s\in S_p}\frac{1}{p}\int_{\mathrm{div}(s|_{\mathcal{Y}})(\mathbb{C})}\Phi=\int_{\mathcal{Y}(\mathbb{C})}\Phi\wedge c_1(\overline{\mathcal{L}}).$$
	\end{enumerate}
\end{Cor}
\begin{proof}
	We first prove (i). Applying the triangle inequality for integrals to Theorem \ref{thm_equidistribution-text} we get
	$$\lim_{p\to \infty}\frac{1}{\# S_p}\sum_{s\in S_p}\left|\frac{1}{p}\int_{\mathcal{Y}(\mathbb{C})}\left(\log|s|_{\overline{\mathcal{L}}^{\otimes p}}-p\log\tau \right)c_1(\overline{\mathcal{L}})^{e-1}\right|=0.$$
	After reordering terms and dividing by $\mathcal{L}_{\mathbb{C}}|_{\mathcal{Y}_{\mathbb{C}}}^{e-1}=\int_{\mathcal{Y}(\mathbb{C})}c_1(\overline{\mathcal{L}})^{e-1}$ we get
	$$\lim_{p\to\infty}\frac{1}{\#S_p}\sum_{s\in S_p}\left|\frac{1}{p\cdot\mathcal{L}_{\mathbb{C}}|_{\mathcal{Y}_{\mathbb{C}}}^{e-1}}\int_{\mathcal{Y}(\mathbb{C})}\log|s|_{\overline{\mathcal{L}}^{\otimes p}}c_1(\overline{\mathcal{L}})^{e-1}-\log\tau\right|=0.$$
	Now we get by an application of Equation (\ref{equ_height}) that
	$$\lim_{p\to\infty}\frac{1}{\#S_p}\sum_{s\in S_p}\left|h_{\overline{\mathcal{L}}|_{\mathcal{Y}}}(\mathrm{div}(s|_{\mathcal{Y}}))-h_{\overline{\mathcal{L}}|_{\mathcal{Y}}}(\mathcal{Y})-\log\tau\right|=0,$$
	which proves (i), as $h_{\overline{\mathcal{L}}|_{\mathcal{Y}}}(\mathrm{div}(s|_{\mathcal{Y}}))=h_{\overline{\mathcal{L}}}(\mathrm{div}(s)\cdot\mathcal{Y})$ and $h_{\overline{\mathcal{L}}|_{\mathcal{Y}}}(\mathcal{Y})=h_{\overline{\mathcal{L}}}(\mathcal{Y})$.
	
	To prove (ii), we set 
	$$K_p=\left.\left\{s\in H^0(\mathcal{X},\overline{\mathcal{L}}^{\otimes p})_{\mathbb{R}}~\right|~\|s\|_{\sup}\le r_p\right\}.$$
	With this choice of $K_p$ the definition of $S_p$ coincides with the definition of $S_p$ in Corollary \ref{cor_equidistribution-general}.
	By Lemma \ref{lem_convergence} it is enough to prove
	\begin{align}\label{equ_distribution-subsequence}
		\lim_{j\to \infty}\frac{1}{\# S_{p_j}}\sum_{s\in S_{p_j}}\left|\frac{1}{p_j}\int_{\mathrm{div}(s|_{\mathcal{Y}})(\mathbb{C})}\Phi-\int_{\mathcal{Y}(\mathbb{C})}\Phi\wedge c_1(\overline{\mathcal{L}})\right|=0
	\end{align}
	for any subsequence $(p_j)_{j\in\mathbb{Z}_{\ge 0}}$ of $\mathbb{Z}_{\ge 0}$, such that $r_{p_j}^{1/p_j}$ converges. If $(p_j)_{j\in\mathbb{Z}_{\ge 0}}$ is such a sequence and $\tau=\lim_{j\to \infty} r_{p_j}^{1/p_j}$, then it follows by the same argument as in the proof of Theorem \ref{thm_equidistribution-text} that we can apply Corollary \ref{cor_equidistribution-general} to obtain Equation (\ref{equ_distribution-subsequence}) as desired.
\end{proof}
Let us also ask some questions for future research related to the results in this section.
\begin{Que}
	\begin{enumerate}[(i)]
		\item What happens if we allow in Corollary \ref{cor_equidistribution-general} $\tau$ to be $\infty$? One may still assume $\lim \left(\frac{r_j}{r_j'}\right)^{1/p_j}=1$ to avoid that the sets $K_j$ are stretched too much to allow equidistribution. But the finiteness of $\tau$ seems not to be necessary for the result. Indeed, if one considers the line bundles $\overline{\mathcal{L}}(r_{p_j})^{\otimes p_j}$, the lattice points $H^0(\mathcal{X},\overline{\mathcal{L}}(r_{p_j})^{\otimes p_j})$ are getting dense in $H^0(\mathcal{X},\overline{\mathcal{L}}(r_{p_j})^{\otimes p_j})_{\mathbb{R}}$ much faster for $j\to\infty$. The only serious obstruction to treat the case $\tau=\infty$ seems to be Lemma \ref{lem_integral-upper-bound}. Everything else can be resolved by working with the arithmetically ample hermitian line bundles $\overline{\mathcal{L}}^{\otimes p_j}(\log^+r_j)$ instead of $\overline{\mathcal{L}}^{\otimes p_j}$ respectively $\overline{\mathcal{L}}(\log\tau)^{\otimes p_j}$.
		\item Similar to Question \ref{que_integral} one may ask for a non-archimedean and an adelic analogue of Corollary \ref{cor_equidistribution-general}.
	\end{enumerate}
\end{Que}
\subsection{A Bogomolov-type Result}\label{sec_bogomolov}
In this section we discuss an analogue of the generalized Bogomolov conjecture for the set of global sections of an arithmetically ample hermitian line bundle $\overline{\mathcal{L}}$ on a generically smooth projective arithmetic variety $\mathcal{X}$. In particular, we will prove Corollary \ref{cor_bogomolov}. Moreover, we will discuss the Northcott property in this setting.

Let us first recall the statement of the generalized Bogomolov conjecture. If $A$ is an abelian variety defined over a number field $K$, $L$ a symmetric ample line bundle on $A$ and $X\subseteq A$ a subvariety, which is not a translate of an abelian subvariety by a torsion point, then there is an $\epsilon>0$ such that the set
\begin{align}\label{equ_bogomolov}
\{P\in X(\overline{K})~|~h_L(P)<\epsilon\}
\end{align}
is not Zariski dense in $X$, where $h_L$ denotes the height associated to $L$. This has been proven by Zhang \cite{Zha98} based on an idea by Ullmo \cite{Ull98}. It has been already proven by Zhang \cite{Zha95} before, that the set (\ref{equ_bogomolov}) is not Zariski dense for some $\epsilon>0$ if and only if $h_L(X)>0=h_L(A)$, where $X$ denotes any subvariety of $A$.

Now we consider the setting of a generically smooth projective arithmetic variety $\mathcal{X}$ and an arithmetically ample hermitian line bundle $\overline{\mathcal{L}}$ on $\mathcal{X}$. Instead of considering $\overline{K}$-rational points, we want to have an analogue result for the sections in $H^0(\mathcal{X},\mathcal{L}^{\otimes p})$. Here, we may consider $p$ as an analogue of the degree of the field of definition of a point $P\in A(\overline{K})$. If we consider an arithmetic subvariety $\mathcal{Y}\subseteq \mathcal{X}$, then we get a restriction map
$$H^0(\mathcal{X},\mathcal{L}^{\otimes p})\to H^0(\mathcal{Y},\mathcal{L}^{\otimes p}), \qquad s\mapsto s|_{\mathcal{Y}}.$$
This is in some sense dual to the classical situation, where we have an embedding $X(\overline{K})\subseteq A(\overline{K})$. Thus, we need also a dual variant of the condition in the set (\ref{equ_bogomolov}). If we consider $h_{\overline{\mathcal{L}}}(\mathrm{div}(s))$ as an analogue of the height $h_L$ for $s\in H^0(\mathcal{X},\mathcal{L}^{\otimes p})$, a good candidate for a dual variant of (\ref{equ_bogomolov}) seems to be
$$\left\{s\in H^0(\mathcal{X},\mathcal{L}^{\otimes p})~|~p\in\mathbb{Z}_{\ge 1}, h_{\overline{\mathcal{L}}}(\mathrm{div}(s))-h_{\overline{\mathcal{L}}|_{\mathcal{Y}}}(\mathrm{div}(s|_{\mathcal{Y}}))\le\epsilon\right\}.$$
As we are now interested in a subset of a union $\bigcup_{p\ge 1}H^0(\mathcal{X},\mathcal{L}^{\otimes p})$ of free $\mathbb{Z}$-modules instead of a subset of geometric points of a subvariety, it does not make any sense to speak about Zariski density. Instead we can speak about the asymptotic density of this subset for $p\to \infty$. To compute it, we will restrict to the sections lying at the same time in some ball under the $\sup$-norm. If we replace the condition $h_L(X)>h_L(A)$ by its dual analogue $h_{\overline{\mathcal{L}}}(\mathcal{Y})<h_{\overline{\mathcal{L}}}(\mathcal{X})$ we can prove the following statement about the density of the above set.
\begin{Cor}[= Corollary \ref{cor_bogomolov}]
	Let $\mathcal{X}$ be any generically smooth projective arithmetic variety and $\overline{\mathcal{L}}$ any arithmetically ample hermitian line bundle on $\mathcal{X}$. Let $\mathcal{Y}\subseteq \mathcal{X}$ be any generically smooth arithmetic subvariety of dimension $e\ge 2$ with $h_{\overline{\mathcal{L}}}(\mathcal{Y})<h_{\overline{\mathcal{L}}}(\mathcal{X})$. For any $\epsilon\in (0,h_{\overline{\mathcal{L}}}(\mathcal{X})-h_{\overline{\mathcal{L}}}(\mathcal{Y}))$ and any sequence $(r_p)_{p\in\mathbb{Z}}$ of positive real numbers with $\lim_{p\to\infty}r_p^{1/p}=\tau\in [1,\infty)$ it holds
	$$\lim_{p\to\infty}\frac{\#\left\{s\in \widehat{H}^0_{\le r_p}(\mathcal{X},\overline{\mathcal{L}}^{\otimes p})~|~h_{\overline{\mathcal{L}}}(\mathrm{div}(s))-h_{\overline{\mathcal{L}}|_{\mathcal{Y}}}(\mathrm{div}(s|_{\mathcal{Y}}))\le\epsilon\right\}}{\#\widehat{H}^0_{\le r_p}(\mathcal{X},\overline{\mathcal{L}}^{\otimes p})}=0.$$
\end{Cor}
\begin{proof}
	We choose any section $s_p$ in the set in question
	\begin{align}\label{equ_section-in-set}
		s_p\in \left\{s\in \widehat{H}^0_{\le r_p}(\mathcal{X},\overline{\mathcal{L}}^{\otimes p})~|~h_{\overline{\mathcal{L}}}(\mathrm{div}(s))-h_{\overline{\mathcal{L}}|_\mathcal{Y}}(\mathrm{div}(s|_\mathcal{Y}))\le\epsilon\right\}.
	\end{align}
	As $h_{\overline{\mathcal{L}}|_{\mathcal{Y}}}(\mathcal{Y})=h_{\overline{\mathcal{L}}}(\mathcal{Y})$, we obtain by Equation (\ref{equ_height}) that
	\begin{align}\label{equ_height-restricted}
	h_{\overline{\mathcal{L}}|_{\mathcal{Y}}}(\mathrm{div}(s_p|_{\mathcal{Y}}))=h_{\overline{\mathcal{L}}}(\mathcal{Y})+\frac{1}{p\cdot\mathcal{L}_{\mathbb{C}}|_{\mathcal{Y}_{\mathbb{C}}}^{e-1}}\int_{\mathcal{Y}(\mathbb{C})}\log|s_p|_{\overline{\mathcal{L}}^{\otimes p}}c_1(\overline{\mathcal{L}})^{e-1}.
	\end{align}
	We will estimate all three terms in this equation.
	First, it follows from (\ref{equ_section-in-set}) that
	$$h_{\overline{\mathcal{L}}|_\mathcal{Y}}(\mathrm{div}(s_p|_\mathcal{Y}))\ge h_{\overline{\mathcal{L}}}(\mathrm{div}(s_p))-\epsilon.$$
	Let us write $\rho=\frac{h_{\overline{\mathcal{L}}}(\mathcal{X})-h_{\overline{\mathcal{L}}}(\mathcal{Y})-\epsilon}{2}$. By the assumptions we have $\rho>0$. We get
	$$h_{\overline{\mathcal{L}}}(\mathcal{Y})=h_{\overline{\mathcal{L}}}(\mathcal{X})-2\rho-\epsilon.$$
	By $\|s_p\|_{\sup}\le r_p$ we get for the integral
	$$\frac{1}{p\cdot\mathcal{L}|_{\mathcal{Y}_{\mathbb{C}}}^{e-1}}\int_{\mathcal{Y}(\mathbb{C})}\log|s_p|_{\overline{\mathcal{L}}^{\otimes p}}c_1(\overline{\mathcal{L}})^{e-1}\le \frac{1}{p}\log r_p.$$
	As $\lim_{p\to\infty}\frac{1}{p}\log r_p=\log\tau$, there is an $N\in\mathbb{Z}$ such that $\frac{1}{p}\log r_p\le \log \tau +\rho$ for all $p\ge N$.
	If we apply all this to Equation (\ref{equ_height-restricted}) and reorder terms, we get
	\begin{align}\label{equ_lower-bound-rho}
	\rho\le h_{\overline{\mathcal{L}}}(\mathcal{X})+\log\tau-h_{\overline{\mathcal{L}}}(\mathrm{div}(s_p))
	\end{align}
	for $p\ge N$, where $\rho$ is a positive number independent of $s_p$ and of $p$.
	If we now assume that 
	\begin{align}\label{equ_quotient-of-sets}
		\frac{\#\left\{s\in \widehat{H}^0_{\le r_p}(\mathcal{X},\overline{\mathcal{L}}^{\otimes p})~|~h_{\overline{\mathcal{L}}}(\mathrm{div}(s))-h_{\overline{\mathcal{L}}|_{\mathcal{Y}}}(\mathrm{div}(s|_{\mathcal{Y}}))\le\epsilon\right\}}{\#\widehat{H}^0_{\le r_p}(\mathcal{X},\overline{\mathcal{L}}^{\otimes p})}
	\end{align}
	does not converge to $0$, then there exists an increasing sequence $(p_j)_{j\in\mathbb{Z}_{\ge 1}}$ of integers and some real number $\mu>0$ such that
	\begin{align}\label{equ_limit-mu}
	\lim_{j\to\infty} \frac{\#\left\{s\in \widehat{H}^0_{\le r_{p_j}}(\mathcal{X},\overline{\mathcal{L}}^{\otimes {p_j}})~|~h_{\overline{\mathcal{L}}}(\mathrm{div}(s))-h_{\overline{\mathcal{L}}|_{\mathcal{Y}}}(\mathrm{div}(s|_{\mathcal{Y}}))\le\epsilon\right\}}{\#\widehat{H}^0_{\le r_{p_j}}(\mathcal{X},\overline{\mathcal{L}}^{\otimes {p_j}})}=\mu.
	\end{align}
	If we set $S_{p}=\widehat{H}^0_{\le r_{p}}(\mathcal{X},\overline{\mathcal{L}}^{\otimes {p}})\setminus \{0\}$, Equations (\ref{equ_lower-bound-rho}) and (\ref{equ_limit-mu}) imply
	$$\liminf_{j\to \infty} \frac{1}{\#S_{p_j}}\sum_{s\in S_{p_j}}\left|h_{\overline{\mathcal{L}}}(\mathcal{X})+\log\tau-h_{\overline{\mathcal{L}}}(\mathrm{div}(s))\right|\ge \mu\rho>0.$$
	But this contradicts Corollary \ref{cor_height-converges-text} (i). Thus, the value in (\ref{equ_quotient-of-sets}) must converge to $0$. This proves the corollary.
\end{proof}

For a further discussion on the analogy between $A(\overline{K})$ and $(H^0(\mathcal{X},\mathcal{L}^{\otimes p}))_{p\in\mathbb{Z}_{\ge 1}}$ let us consider the Northcott property. For the abelian variety $A$ the Northcott property states that for any $d\in\mathbb{Z}_{\ge 1}$ and any $M\in \mathbb{R}$ the set
$$\left\{ P\in A(\overline{K})~|~[K(P):K]=d, h_L(P)\le M\right\}$$
is finite. In general, we say that a pair $((S_p)_{p\in\mathbb{Z}_{\ge 1}},h)$ of a sequence of sets together with a function $h\colon \bigsqcup_{p\ge 1}S_p\to \mathbb{R}$ satisfies the \emph{Northcott property} if for any $p\in\mathbb{Z}_{\ge 1}$ and any $M\in \mathbb{R}$ the set
$$\{x\in S_p~|~h(x)\le M\}$$
is finite. It will turn out that $((H^0(\mathcal{X},\mathcal{L}^{\otimes p})\setminus\{0\})_{p\in\mathbb{Z}_{\ge 1}},h_{\overline{\mathcal{L}}}(\mathrm{div}(\cdot)))$ in general does not satisfies the Northcott property. The reason is, that the units $H^0(\mathcal{X},\mathcal{O}_{\mathcal{X}})^*$ of $H^0(\mathcal{X},\mathcal{O}_{\mathcal{X}})$ are acting on $H^0(\mathcal{X},\mathcal{L}^{\otimes p})$ without changing the divisor. If we divide out this action, we always get the Northcott property.
\begin{Pro}\label{pro_northcott}
	Let $\mathcal{X}$ be any generically smooth projective arithmetic variety and $\overline{\mathcal{L}}$ an arithmetically ample hermitian line bundle on $\mathcal{X}$.
	\begin{enumerate}[(i)]
		\item The pair $((H^0(\mathcal{X},\mathcal{L}^{\otimes p})\setminus\{0\})_{p\in\mathbb{Z}_{\ge 1}},h_{\overline{\mathcal{L}}}(\mathrm{div}(\cdot)))$ satisfies the Northcott property if and only if $\mathcal{X}_{\mathbb{R}}$ is connected.
		\item The pair $(((H^0(\mathcal{X},\mathcal{L}^{\otimes p})\setminus\{0\})/H^0(\mathcal{X},\mathcal{O}_{\mathcal{X}})^*)_{p\in\mathbb{Z}_{\ge 1}},h_{\overline{\mathcal{L}}}(\mathrm{div}(\cdot)))$ always satisfies the Northcott property.
	\end{enumerate}
\end{Pro}
We first state some basic facts on the ring of global regular functions $H^0(\mathcal{X},\mathcal{O}_{\mathcal{X}})$.
\begin{Lem}\label{lem_regular-functions}
	Let $\mathcal{X}$ be any projective arithmetic variety. Write $R=H^0(\mathcal{X},\mathcal{O}_{\mathcal{X}})$ and $K=R_{\mathbb{Q}}$.
	\begin{enumerate}[(i)]
		\item The ring $K$ is a number field and $R$ is an order in $K$.
		\item We have $\#R^*<\infty$ for the units $R^*$ of $R$ if and only if $\mathcal{X}_{\mathbb{R}}$ is connected.
	\end{enumerate}
\end{Lem}
\begin{proof}
	\begin{enumerate}[(i)]
		\item As $\mathcal{X}$ is integral, also $K=H^0(\mathcal{X},\mathcal{O}_{\mathcal{X}})_{\mathbb{Q}}$ has to be integral. As it is a finite dimensional $\mathbb{Q}$-vector space, it has to be a number field. Since $\mathcal{X}$ is projective and finitely generated over $\mathrm{Spec}(\mathbb{Z})$, $R=H^0(\mathcal{X},\mathcal{O}_{\mathcal{X}})$ is a finitely generated $\mathbb{Z}$-module. By the flatness of $\mathcal{X}$ over $\mathrm{Spec}(\mathbb{Z})$ the ring $R$ is also a free $\mathbb{Z}$-module and since $R\otimes_\mathbb{Z}\mathbb{Q}=K$, it generates $K$ over $\mathbb{Q}$. Thus, $R$ is an order of the number field $K$.
		\item 
		We have that $\mathcal{X}_{\mathbb{R}}$ is connected if and only if $H^0(\mathcal{X}_{\mathbb{R}},\mathcal{O}_{\mathcal{X}_{\mathbb{R}}})\cong K_{\mathbb{R}}$ is a field. But $K_\mathbb{R}$ is a field if and only if $K_{\mathbb{R}}\cong \mathbb{R}$ or $K_{\mathbb{R}}\cong \mathbb{C}$. If we write $r$ for the number of real embeddings $K\to\mathbb{R}$ and $s$ for the number of conjugate pairs of non-real embeddings $K\to\mathbb{C}$, the first case is equivalent to $r=1$ and $s=0$ and the second case occurs exactly if $r=0$ and $s=1$. Thus, $\mathcal{X}_{\mathbb{R}}$ is connected if and only if $r+s-1=0$ which is equivalent to $\#R^*<\infty$ by Dirichlet's unit theorem.
	\end{enumerate}
\end{proof}
We now give the proof of the proposition.
\begin{proof}[Proof of Proposition \ref{pro_northcott}]
	For any section $s\in H^0(\mathcal{X},\mathcal{L}^{\otimes p})\setminus\{0\}$ and any invertible regular function $\alpha\in H^0(\mathcal{X},\mathcal{O}_{\mathcal{X}})^*$ we get for the section 
	$\alpha\cdot s\in H^0(\mathcal{X},\mathcal{L}^{\otimes p})$ that 
	$$\mathrm{div}(\alpha\cdot s)=\mathrm{div}(s)$$
	as the divisor of $\alpha$ is trivial. On the other hand, if $s,s'\in H^0(\mathcal{X},\mathcal{L}^{\otimes p})\setminus\{0\}$ satisfy $\mathrm{div}(s)=\mathrm{div}(s')$, then $\frac{s}{s'}\in H^0(\mathcal{X},\mathcal{O}_{\mathcal{X}})^*$ is an invertible regular function as its divisor is trivial. In other words, we have the following diagram 
	$$\xymatrix{H^0(\mathcal{X},\mathcal{L}^{\otimes p})\setminus\{0\}\ar[dr] \ar[r] & S_p:=\left\{\mathcal{Z}\in Z^1(\mathcal{X})~|~\mathcal{L}_\mathbb{C}|_{\mathcal{Z}_{\mathbb{C}}}^{\dim \mathcal{X}_{\mathbb{C}}-1}=p\cdot \mathcal{L}_{\mathbb{C}}^{\dim \mathcal{X}_{\mathbb{C}}}\right\}\\ & (H^0(\mathcal{X},\mathcal{L}^{\otimes p})\setminus\{0\})/H^0(\mathcal{X},\mathcal{O}_X)^*\ar@{^{(}->}[u]},$$
	where the upper map sends $s$ to $\mathrm{div}(s)$. Note that
	$$\mathcal{L}_\mathbb{C}|_{\mathrm{div}(s)_{\mathbb{C}}}^{\dim\mathcal{X}_{\mathbb{C}}-1}=p\cdot \mathcal{L}_{\mathbb{C}}^{\dim \mathcal{X}_{\mathbb{C}}}$$
	as $\mathrm{div}(s)_{\mathbb{C}}$ represents the line bundle $\mathcal{L}_{\mathbb{C}}^{\otimes p}$ on $\mathcal{X}_{\mathbb{C}}$. Thus, to prove (ii) it is enough to prove the Northcott property for $((S_p)_{p\in\mathbb{Z}_{\ge 1}},h_{\overline{\mathcal{L}}})$. But this has been shown by Bost--Gillet--Soul\'e \cite[Proposition 3.2.5]{BGS94} even for cycles of higher codimension. Hence, (ii) follows.
	
	As $\mathcal{L}$ is ample, we can always choose $p\in\mathbb{Z}_{\ge 1}$ such that $H^0(\mathcal{X},\mathcal{L}^{\otimes p})\neq \{0\}$. But then there are at least $\# H^0(\mathcal{X},\mathcal{O}_{\mathcal{X}})^*$ many sections $s\in H^0(\mathcal{X},\mathcal{L}^{\otimes p})\neq \{0\}$ with the same height $h_{\overline{\mathcal{L}}}(\mathrm{div}(s))$. Thus, by Lemma \ref{lem_regular-functions} the pair $$((H^0(\mathcal{X},\mathcal{L}^{\otimes p})\setminus\{0\})_{p\in\mathbb{Z}_{\ge 1}},h_{\overline{\mathcal{L}}}(\mathrm{div}(\cdot)))$$ can not satisfy the Northcott property if $\mathcal{X}_{\mathbb{R}}$ is not connected. This shows the \emph{``only if''}-part of (i). Lemma \ref{lem_regular-functions} also shows that the \emph{``if''}-direction of (i) follows from (ii), which we already have proven.
\end{proof}
\subsection{Distribution of Algebraic Numbers}\label{sec_algebraic-numbers}
In this section we apply the results on the distribution of small sections in Section \ref{sec_distribution-small-sections} to the case $\mathcal{X}=\mathbb{P}^1_\mathbb{Z}$ and $\overline{\mathcal{L}}=\overline{\mathcal{O}(1)}(\epsilon)$, the line bundle $\mathcal{O}(1)$ equipped with the Fubini--Study metric multiplied by $e^{-2\epsilon}$ for some $\epsilon>0$. This can be interpreted as a result on the distribution of the zero sets of integer polynomials and hence, on the distribution of algebraic numbers. In particular, we will prove Proposition \ref{pro_equidistribution-polynomials} and Corollary \ref{cor_distribution-algebraic-numbers}. We continue the notation from Section \ref{sec_projective-line}. Especially, $Z_0, Z_1$ are homogeneous coordinate functions on $\mathbb{P}_\mathbb{Z}^1$, $U\subseteq\mathbb{P}_{\mathbb{C}}^1$ is defined by $Z_1\neq 0$ and $S_j^n=Z_0^jZ_1^{j-n}\in H^0(\mathbb{P}_{\mathbb{Z}}^1,\mathcal{O}(1)^{\otimes n})$.

First we discuss relations between polynomials and sections in $H^0(\mathbb{P}_{\mathbb{Z}}^1,\mathcal{O}(1))$. By choosing $[1:0]\in\mathbb{P}_{\mathbb{Z}}^1$ as the point at infinity, we get a group isomorphism
\begin{align}\label{equ_polynomial-map}
	H^0(\mathbb{P}_{\mathbb{Z}}^1,\mathcal{O}(1)^{\otimes n})&\to \{P\in \mathbb{Z}[X]~|~\deg P\le n\}\\
	s=\sum_{j=0}^n a_jZ_0^{j}Z_1^{n-j}&\mapsto P(X)=\sum_{j=1}^n a_j X^j.\nonumber
\end{align}
After tensoring with $\mathbb{C}$ we also get an isomorphism of linear spaces
$$H^0(\mathbb{P}_{\mathbb{Z}}^1,\mathcal{O}(1)^{\otimes n})_{\mathbb{C}}\to \{P\in \mathbb{C}[X]~|~\deg P\le n\}.$$
Let us recall the definitions of $h_B$ and $h_{\mathrm{FS}}$ from the introduction. If $P\in \mathbb{C}[X]$ is a polynomial of degree $\deg P=n$ with $P=a_n\prod_{j=1}^{n}(X-\alpha_j)=\sum_{j=0}^n a_j X^j$, then we define the heights
$$h_{\mathrm{FS}}(P)=\tfrac{1}{n}\log|a_n|+\tfrac{1}{2n}\sum_{j=1}^n\log(1+|\alpha_j|^2) \quad \text{and}\quad h_{B}(P)=\tfrac{1}{n}\log \max_{0\le j\le n}\frac{|a_j|}{\sqrt{\tbinom{n}{j}}}.$$
The following lemma connects $h_{\mathrm{FS}}$ to integrals of sections as occurring in Lemma \ref{lem_stokes}. 
\begin{Lem}\label{lem_hFS}
	Let $P\in \mathbb{C}[X]$ be any polynomial of degree $\deg P=n$ with
	$$P(X)=\sum_{k=0}^n a_k X^k=a_n\prod_{k=1}^n(X-\alpha_k)$$
	for complex numbers $a_0,\dots, a_n,\alpha_1,\dots,\alpha_n\in \mathbb{C}$. Let $s=\sum_{j=0}^n a_jS_j^n$ be the corresponding section in $H^0(\mathbb{P}_{\mathbb{Z}}^1,\mathcal{O}(1)^{\otimes n})_{\mathbb{C}}$. Then the following holds
	$$\tfrac{1}{n}\int_{\mathbb{P}_{\mathbb{C}}^1}\log |s|_{\mathrm{FS}}c_1(\overline{\mathcal{O}(1)})=h_{\mathrm{FS}}(P)-\tfrac{1}{2}.$$
\end{Lem}
\begin{proof}
	As $\deg P=n$, we have that $a_n\neq 0$. Thus, $s$ does not vanish at the point at infinity $[1:0]\in\mathbb{P}^1_{\mathbb{C}}$. Hence, the divisor $\mathrm{div}(s)$ of $s$ consists exactly of the zeros of $P$, more precisely $\mathrm{div}(s)=\sum_{j=1}^n [\alpha_j:1]$.
	Especially, $\mathrm{div}(s)$ is not supported in $[1:0]\in\mathbb{P}_{\mathbb{C}}^1$, such that $\mathrm{div}(s_n)\cap \mathrm{div}(Z_1)=\emptyset$. Thus, we can apply Lemmas \ref{lem_stokes} to obtain
	\begin{align}\label{equ_stokes-p1}
		&\tfrac{1}{n}\int_{\mathbb{P}_{\mathbb{C}}^1}\log|s|_{\mathrm{FS}}c_1(\overline{\mathcal{O}(1)})-\tfrac{1}{n}\int_{\mathrm{div}(Z_1)}\log|s|_{\mathrm{FS}}\\
		&=\int_{\mathbb{P}_{\mathbb{C}}^1}\log|Z_1|_{\mathrm{FS}}c_1(\overline{\mathcal{O}(1)})-\tfrac{1}{n}\int_{\mathrm{div}(s)}\log|Z_1|_{\mathrm{FS}}.\nonumber
	\end{align}
	Let us compute the integrals separately. First, we get 
	\begin{align*}
		\tfrac{1}{n}\int_{\mathrm{div}(Z_1)}\log|s|_{\mathrm{FS}}=\tfrac{1}{2n}\int_{[Z_0:Z_1]=[1:0]}\log\frac{\left|\sum_{j=0}^n a_j Z_0^j Z_1^{n-j}\right|^2}{(|Z_0|^2+|Z_1|^2)^n}=\tfrac{1}{n}\log |a_n|.
	\end{align*}
	By $\mathrm{div}(s)=\sum_{j=1}^n [\alpha_j:1]$ we have
	\begin{align*}
		-\tfrac{1}{n}\int_{\mathrm{div}(s)}\log|Z_1|_{\mathrm{FS}}=-\tfrac{1}{2n}\sum_{j=1}^n\int_{[\alpha_j:1]}\log \frac{|Z_1|^2}{|Z_0|^2+|Z_1|^2}=\tfrac{1}{2n}\sum_{j=1}^n\log(1+|\alpha_j|^2)
	\end{align*}
	Summing both formulas yields
	\begin{align}\label{equ_hFS}
	h_{\mathrm{FS}}(P)=\tfrac{1}{n}\int_{\mathrm{div}(Z_1)}\log|s|_{\mathrm{FS}}-\tfrac{1}{n}\int_{\mathrm{div}(s)}\log|Z_1|_{\mathrm{FS}}.
	\end{align}
	If we now apply Equation (\ref{equ_hFS}) and Lemma \ref{lem_integral-P1}(ii) to Equation (\ref{equ_stokes-p1}), we get
	$$\tfrac{1}{n}\int_{\mathbb{P}_{\mathbb{C}}^1}\log |s|_{\mathrm{FS}}c_1(\overline{\mathcal{O}(1)})=-\tfrac{1}{2}+h_{\mathrm{FS}}(P).$$
	as desired.
\end{proof}
The following Proposition gives a sufficient condition for a sequence of polynomials, such that their zero sets equidistribute on $\mathbb{P}_{\mathbb{C}}^1$ with respect to $c_1(\overline{\mathcal{O}(1)})$.
\begin{Pro}[= Proposition \ref{pro_equidistribution-polynomials}]
	For any sequence $(P_n)_{n\in\mathbb{Z}_{\ge 1}}$ of polynomials
	$P_n\in \mathbb{C}[X]$ of degree $\deg P_n=n$ satisfying 
	$\limsup_{n\to\infty} (h_{B}(P_n)+\frac{1}{2}-h_{\mathrm{FS}}(P_n))\le0$
	and any continuous function $f\colon \mathbb{C}\to\mathbb{C}$, such that $\lim_{|z|\to\infty}f(z)$ is well-defined and finite, it holds
	$$\lim_{n\to \infty}\frac{1}{n}\sum_{\gf{z\in\mathbb{C}}{P_n(z)=0}}f(z)=\frac{i}{2\pi}\int_\mathbb{C}f(z)\frac{dzd\overline{z}}{\left(1+|z|^2\right)^2}.$$
\end{Pro}
Note that we always take the sum over the zeros of $P_n$ counted with multiplicity.
\begin{proof}
	Let $(P_n)_{n\in\mathbb{Z}_{\ge 1}}$ be a sequence of polynomials $P_n\in\mathbb{C}[X]$ as in the proposition. Let us write $P_n=\sum_{j=0}^n a_{n,j}X^j=a_n\prod_{j=1}^n(X-\alpha_{n,j})$ and $s_n=\sum_{j=0}^n a_{n,j}S_j^n$ for the section $s_n\in H^0(\mathbb{P}^1_{\mathbb{C}},\mathcal{O}(1)^{\otimes n})\cong H^0(\mathbb{P}^1_{\mathbb{Z}},\mathcal{O}(1)^{\otimes n})_\mathbb{C}$ associated to $P_n$. By Lemma \ref{lem_hFS} the assumption in the proposition can be read as
	\begin{align}\label{equ_assumption-limit}
		\limsup_{n\to \infty}\left(h_B(P_n)-\tfrac{1}{n}\int_{\mathbb{P}_{\mathbb{C}}^1}\log|s_n|_{\mathrm{FS}} c_1(\overline{\mathcal{O}(1)})\right)\le 0.
	\end{align}
	Our goal is to apply Lemma \ref{lem_equidistribution} (ii). For this purpose, we have to compare $h_B(P_n)$ with $\|s_n\|_{\sup}$. First, let us compare $h_B(P_n)$ with the $L^2$-norm by using Equation (\ref{equ_L2norm-FS})
	\begin{align*}
		&h_B(P_n)-\tfrac{\log(n+1)}{2n}=\tfrac{1}{2n}\log\left(\tfrac{1}{n+1}\max_{0\le j\le n}\frac{|a_j|^2}{\binom{n}{j}}\right)\le\tfrac{1}{n}\log\|s_n\|\\
		&=\tfrac{1}{2n}\log\left(\frac{1}{n+1}\sum_{j=0}^n\frac{|a_j|^2}{\binom{n}{j}}\right)\le\tfrac{1}{2n}\log\max_{0\le j\le n}\frac{|a_j|^2}{\binom{n}{j}}=h_B(P_n).
	\end{align*}
	If we combine this with Equation (\ref{equ_supl2}) we get
	$$h_B(P_n)-\frac{\log(n+1)}{2n}\le\tfrac{1}{n}\log\|s_n\|_{\mathrm{sup}}\le h_B(P_n)+\frac{\log (C_1n)}{n}.$$
	This shows that $\lim_{n\to \infty} h_B(P_n)=\lim_{n\to\infty}\tfrac{1}{n}\log\|s_n\|_{\sup}$. Thus, the assumption (\ref{equ_assumption-limit}) is equivalent to
	$$\limsup_{n\to\infty}\frac{1}{n}\left(\log\|s_n\|_{\sup}-\int_{\mathbb{P}^1_{\mathbb{C}}}\log|s_n|_{\mathrm{FS}}c_1(\overline{\mathcal{O}(1)})\right)\le 0.$$
	As the value in the brackets is always non-negative, this is again equivalent to
	$$\lim_{n\to\infty}\frac{1}{n}\left(\log\|s_n\|_{\sup}-\int_{\mathbb{P}^1_{\mathbb{C}}}\log|s_n|_{\mathrm{FS}}c_1(\overline{\mathcal{O}(1)})\right)= 0.$$
	Now by Lemma \ref{lem_equidistribution} (ii) we get
	$$\lim_{n\to \infty}\frac{1}{n}\sum_{j=1}^n f(\alpha_{n,j})=\lim_{n\to\infty}\frac{1}{n}\int_{\mathrm{div}(s_n)}\overline{f}=\int_{\mathbb{P}_{\mathbb{C}}^1} \overline{f}c_1(\overline{\mathcal{O}(1)})=\int_{\mathbb{C}}f(z)\frac{dzd\overline{z}}{(1+|z|^2)^2},$$
	where $\overline{f}\colon \mathbb{P}_{\mathbb{C}}^1\to \mathbb{C}$ denotes the unique continuous extension of $f\colon \mathbb{C}\cong U\to \mathbb{C}$, which exists as $\lim_{|z|\to\infty}f(z)$ is well-defined and finite. Note that we used the identification $\mathbb{C}\cong U$ by setting $Z_1=1$ in the first and the last equality and that $$\mathrm{div}(s_n)=\{[\alpha_{n,j}:1]\in \mathbb{P}_{\mathbb{C}}^1~|~1\le j\le n\}.$$
\end{proof}
Our next result shows that in sets of integer polynomials of bounded height $h_B$ the condition in the last proposition is almost always satisfied if the degree goes to infinity. This corresponds to an equidistribution result on average, which we will deduce from the general equidistribution result in Corollary \ref{cor_equidistribution-general}.
\begin{Cor}[= Corollary \ref{cor_distribution-algebraic-numbers}]\label{cor_distribution-algebraic-numbers-text}
	For any $n\in\mathbb{Z}_{\ge 0}$ and any $r\in\mathbb{R}_{>0}$ we define
	$$\mathcal{P}_{n,r}=\left\{\left.P\in \mathbb{Z}[X]~\right|~\deg P=n,~ h_B(P)\le r\right\}.$$
	Let $(r_n)_{n\in\mathbb{Z}_{\ge 1}}$ be any  sequence of real numbers.
	\begin{enumerate}[(i)]
		\item If $r_n$ converges to a value $\tau\in (0,\infty)$, then it holds
		$$\lim_{n\to \infty} \frac{1}{\#\mathcal{P}_{n,r_n}}\sum_{P\in \mathcal{P}_{n,r_n}} \left|\tau+\tfrac{1}{2}-h_{\mathrm{FS}}(P)\right|=0.$$
		\item If $0<\liminf_{n\to \infty} r_n\le \limsup_{n\to \infty} r_n<\infty$, then for any continuous function $f\colon \mathbb{C}\to\mathbb{C}$, such that $\lim_{|z|\to\infty}f(z)$ is well-defined and finite, it holds
		$$\lim_{n\to \infty} \frac{1}{\#\mathcal{P}_{n,r_n}}\sum_{P\in \mathcal{P}_{n,r_n}} \left|\frac{1}{n}\sum_{\gf{z\in\mathbb{C}}{P(z)=0}}f(z)-\frac{i}{2\pi}\int_\mathbb{C}f(z)\frac{dzd\overline{z}}{\left(1+|z|^2\right)^2}\right|=0.$$
		In particular, we have
		$$\lim_{n\to \infty} \frac{1}{\#\mathcal{P}_{n,r_n}}\sum_{P\in \mathcal{P}_{n,r_n}} \frac{1}{n}\sum_{\gf{z\in\mathbb{C}}{P(z)=0}}f(z)=\frac{i}{2\pi}\int_\mathbb{C}f(z)\frac{dzd\overline{z}}{\left(1+|z|^2\right)^2}.$$		
	\end{enumerate}
\end{Cor}
The idea of the proof is to apply Theorem \ref{thm_equidistribution-general} and Corollary \ref{cor_equidistribution-general} to $\mathcal{X}=\mathbb{P}_{\mathbb{Z}}^1$ and the arithmetically ample hermitian line bundle $\overline{\mathcal{O}(1)}(\epsilon)$ for any $\epsilon>0$. To apply Theorem \ref{thm_equidistribution-general} and Corollary \ref{cor_equidistribution-general} we have to choose appropriate convex subsets $K_n\subseteq H^0(\mathbb{P}_{\mathbb{Z}}^1,\mathcal{O}(1)^{\otimes n})_{\mathbb{R}}$. For this purpose, we consider the Bombieri $\infty$-norm $\|\cdot\|_B$ on $H^0(\mathbb{P}_{\mathbb{Z}}^1,\mathcal{O}(1)^{\otimes n})_{\mathbb{C}}$, which is the analogue of $h_B$ on the set of polynomials. We define $\|\cdot\|_B$ by
$$\|\cdot\|_B\colon H^0(\mathbb{P}^1_\mathbb{Z},\mathcal{O}(1)^{\otimes n})_\mathbb{C}\to \mathbb{R}_{\ge 0},\qquad \sum_{k=0}^n a_kS_k\mapsto \max_{0\le k\le n}\frac{|a_k|}{\sqrt{\binom{n}{k}}}$$
and hence, it is the $\max$-norm associated to the basis $\sqrt{\binom{n}{0}}\cdot S^n_0,\dots,\sqrt{\binom{n}{n}}\cdot S^n_n$. Now we set
$$K_n=\overline{B_{e^{n r_n}}(\|\cdot\|_B)}\subseteq H^0(\mathbb{P}_\mathbb{Z}^1,\mathcal{O}(1)^{\otimes n})_\mathbb{R}$$
to be the closed and centered ball in the real space $H^0(\mathbb{P}_\mathbb{Z}^1,\mathcal{O}(1))_\mathbb{R}$ of radius $e^{n r_n}$ under the Bombieri $\infty$-norm $\|\cdot\|_B$, where $r_n$ is given as in Corollary \ref{cor_distribution-algebraic-numbers-text}. 

Next, let us bound $K_n$ from above and from below by balls in the inner product space $H^0(\mathbb{P}_{\mathbb{Z}}^1,\overline{\mathcal{O}(1)}(\epsilon)^{\otimes n})_{\mathbb{R}}$.
Let $\|\cdot\|_{\epsilon}$ denote the $L^2$-norm on $H^0(\mathbb{P}_{\mathbb{Z}}^1,\mathcal{O}(1)^{\otimes n})_{\mathbb{C}}$ induced by the metric of $\overline{\mathcal{O}(1)}(\epsilon)^{\otimes n}$ and the Kähler form $c_1(\overline{\mathcal{O}(1)})$. Note that $\|s\|_{\epsilon}=e^{-n\epsilon}\|s\|$ for any section $s\in H^0(\mathbb{P}_\mathbb{Z}^1,\mathcal{O}(1)^{\otimes n})_{\mathbb{C}}$. For any $s=\sum_{k=0}^n a_k S_k^n$ we get by Equation (\ref{equ_L2norm-FS})
\begin{align}\label{equ_norm-bombieri-l2}
	e^{2n\epsilon}\|s\|_{\epsilon}^2=\frac{1}{n+1}\sum_{k=0}^n \frac{|a_k|^2}{\binom{n}{k}}\le\|s\|_B^2=\max_{0\le k\le n}\frac{|a_k|^2}{\binom{n}{k}}\le \sum_{k=0}^n \frac{|a_k|^2}{\binom{n}{k}}=(n+1)e^{2n\epsilon}\|s\|_{\epsilon}^2.
\end{align}
Thus, in the inner product space $H^0(\mathbb{P}_{\mathbb{Z}}^1,\overline{\mathcal{O}(1)}(\epsilon)^{\otimes n})_{\mathbb{R}}$ we have that
$$\overline{B_{t'_n}}\subseteq K_n\subseteq \overline{B_{t_n}}$$
for all $n\in\mathbb{Z}_{\ge 1}$, where $t'_n=\frac{e^{n(r_n-\epsilon)}}{\sqrt{n+1}}$ and $t_n=e^{n(r_n-\epsilon)}$.

Now we give the proof of Corollary \ref{cor_distribution-algebraic-numbers-text}.
\begin{proof}[Proof of Corollary \ref{cor_distribution-algebraic-numbers-text}]
	To prove (i) let $(r_n)_{n\in\mathbb{Z}_{\ge 1}}$ be a sequence of positive reals with $\lim_{n\to\infty} r_n=\tau$ for some $\tau\in (0,\infty)$. As $\epsilon>0$ was arbitrary, we may choose $0<\epsilon<\tau$. Then we get
	$$\lim_{n\to\infty}(t_n')^{1/n}=\lim_{n\to\infty}(t_n)^{1/n}=\lim_{n\to\infty}e^{r_n-\epsilon}=e^{\tau-\epsilon}>1.$$
	Thus, we can apply Theorem \ref{thm_equidistribution-general}. If we write $T_n=K_n\cap H^0(\mathbb{P}_{\mathbb{Z}}^1,\mathcal{O}(1)^{\otimes n})\setminus\{0\}$, this yields
	$$\lim_{n\to\infty}\frac{1}{\# T_n}\sum_{s\in T_n}\frac{1}{n}\int_{\mathbb{P}_{\mathbb{C}}^1}\left|\log|s|_{\mathrm{FS},\epsilon}-n\log e^{\tau-\epsilon}\right|c_1(\overline{\mathcal{O}(1)})=0,$$
	where $|s|_{\mathrm{FS},\epsilon}=e^{-\epsilon n}|s|_{\mathrm{FS}}$ is the norm induced by the metric on $\overline{\mathcal{O}(1)}(\epsilon)^{\otimes n}$.
	As the factors $e^{-\epsilon n}$ cancel out, we get by the triangle inequality for integrals
	\begin{align}\label{equ_limit-of-Tn}
	\lim_{n\to \infty}\frac{1}{\#T_n}\sum_{s\in T_n}\left| \frac{1}{n}\int_{\mathbb{P}_{\mathbb{C}}^1}\log|s|_{\mathrm{FS}}c_1(\overline{\mathcal{O}(1)})-\tau\right|=0.
	\end{align}
	
	Let us define the subset
	$$T'_n=\left\{\left.\sum_{k=0}^n a_j S_k^n\in T_n~\right|~a_n\neq 0\right\}\subseteq T_n.$$
	As bounding the Bombieri $\infty$-norm means bounding every coefficient independently, we get
	$$1\ge\frac{\#T'_n}{\#T_n}\ge\frac{\#\{a_n\in\mathbb{Z}~|~a_n\neq 0\text{ and } |a_n|\le e^{nr_n}\}}{\#\{a_n\in\mathbb{Z}~|~|a_n|\le e^{nr_n}\}}\ge\frac{2e^{nr_n}-2}{2e^{nr_n}+1}.$$
	Note, that we only get an inequality instead of an equality in the second step as we removed $0$ from $T_n$.
	As $\lim_{n\to \infty}r_n=\tau>0$ we get $\lim_{n\to\infty} \frac{\# T_n'}{\#T_n}=1$. 
	Thus, we can deduce from Equation (\ref{equ_limit-of-Tn}) that
	$$\lim_{n\to \infty}\frac{1}{\#T'_n}\sum_{s\in T'_n}\left| \frac{1}{n}\int_{\mathbb{P}_{\mathbb{C}}^1}\log|s|_{\mathrm{FS}}c_1(\overline{\mathcal{O}(1)})-\tau\right|=0.$$
	By construction the map in (\ref{equ_polynomial-map}) induces a bijection
	$$T'_n\xrightarrow{\cong}\mathcal{P}_{n,r_n}.$$
	Using this bijection and Lemma \ref{lem_hFS} we can replace $T'_n$ by $\mathcal{P}_{n,r_n}$ and the integral by the height $h_{\mathrm{FS}}(P)-\frac{1}{2}$ to get
	$$\lim_{n\to \infty}\frac{1}{\#\mathcal{P}_{n,r_n}}\sum_{P\in \mathcal{P}_{n,r_n}}\left| h_{\mathrm{FS}}(P)-\tfrac{1}{2}-\tau\right|=0.$$
	as claimed in (i).
	
	To prove (ii) we now only assume that the sequence $(r_n)_{n\in\mathbb{Z}_{\ge 1}}$ of positive reals satisfies
	$$0<\liminf_{n\to\infty}r_n\le \limsup_{n\to \infty} r_n<\infty.$$
	By Lemma \ref{lem_convergence} it is enough to prove that 
	\begin{align}\label{equ_limit-polynomial-subsequence}
	\lim_{j\to\infty}\frac{1}{\#\mathcal{P}_{n_j,r_{n_j}}}\sum_{P\in\mathcal{P}_{n_j,r_{n_j}}}\left|\frac{1}{n_j}\sum_{\gf{z\in\mathbb{C}}{P(z)=0}}f(z)-\frac{i}{2\pi}\int_{\mathbb{C}}f(z)\frac{dzd\overline{z}}{(1+|z|^2)^2}\right|=0
	\end{align}
	for every subsequence $(n_j)_{j\in\mathbb{Z}_{\ge 1}}$ of $\mathbb{Z}_{\ge 1}$ such that $(r_{n_j})_{j\in\mathbb{Z}_{\ge 0}}$ converges in $(0,\infty)$. Let $(n_j)_{j\in\mathbb{Z}_{\ge 1}}$ be any such subsequence and $\tau=\lim_{j\to\infty}r_{n_j}$. We again choose $\epsilon>0$ such that $0<\epsilon<\tau$. Then
	$$\lim_{j\to\infty}(t'_j)^{1/n_j}=\lim_{j\to\infty}(t_j)^{1/n_j}=\lim_{j\to\infty}e^{r_{n_j}-\epsilon}>1.$$
	Thus, we can apply Corollary \ref{cor_equidistribution-general} and with the notation as in the proof of (i) we get
	$$\lim_{j\to \infty}\frac{1}{\#T_{n_j}}\sum_{s\in T_{n_j}}\left|\frac{1}{n_j}\int_{\mathrm{div}(s)(\mathbb{C})}\overline{f}-\int_{\mathbb{P}_{\mathbb{C}}^1}\overline{f}c_1(\overline{\mathcal{O}(1)})\right|=0,$$
	where $\overline{f}\colon \mathbb{P}_{\mathbb{C}}^1\to\mathbb{C}$ denotes the unique extension of $f$ from $\mathbb{C}\cong U$ to $\mathbb{P}_{\mathbb{C}}^1$, which exists as $\lim_{|z|\to\infty}f(z)$ is well-defined and finite. By $\lim_{j\to \infty} \frac{\# T'_j}{\#T_j}=1$, the bijection $T'_j\cong \mathcal{P}_{n_j,r_{n_j}}$, the identification $U\cong \mathbb{C}$, and the fact that $\mathrm{div}(s)(\mathbb{C})=\sum_{k=1}^n [\alpha_k:1]$ for the zeros $\alpha_k$ of the polynomial corresponding to $s$, we get Equation (\ref{equ_limit-polynomial-subsequence}) as desired.
\end{proof}
One may ask about a similar result for the algebraic points in $\mathbb{P}_{\mathbb{C}}^n$.
\begin{Que}
	Can we generalize the equidistribution result in Corollary \ref{cor_distribution-algebraic-numbers-text} (ii) to algebraic points in $\mathbb{P}_{\mathbb{C}}^n$ by applying the equidistribution result in Corollary \ref{cor_equidistribution-general} to the line bundle $\overline{\mathcal{O}(1)}(\epsilon)$ on the projective space $\mathbb{P}_{\mathbb{Z}}^n$ equipped with the Fubini--Study metric multiplied with $e^{-2\epsilon}$ for some $\epsilon>0$? As this gives only an equidistribution result on the $(n-1)$-dimensional divisors of the polynomials in $\mathbb{Z}[X_1,\dots,X_n]$, one may use Corollary \ref{cor_equidistribution-general} again on the divisors and use induction to obtain an equidistribution result on the algebraic points in $\mathbb{P}_{\mathbb{C}}^n$.
\end{Que}
\section{Irreducibility of Arithmetic Cycles}\label{sec_irreducibility}
In this section we define a new notion of irreducibility for horizontal arithmetic cycles of Green type. In Section \ref{sec_air} we motivate this notion by pointing out that the naive definition of irreducibility would make no sense for these cycles. Instead, we will define a horizontal arithmetic cycle of Green type to be $\epsilon$-irreducible if the degree of its analytic part is small compared to the degree of its irreducible classical part. We also discuss some trivial examples and some criteria for checking this definition of irreducibility. As an application, we discuss the computation of arithmetic intersection numbers proving Proposition \ref{pro_intersection-finite}. In section \ref{sec_bertini} we prove our main result Theorem \ref{thm_air}, which can be considered as an analogue of Bertini's theorem for arithmetically ample hermitian line bundles and our new notion of irreducibility. Finally proving Theorem \ref{thm_intersection-line-bundles-geometric}, we will show that the arithmetic intersection number of any arithmetically ample hermitian line bundles can be computed as a limit of classical geometric intersection numbers over the finite places.
\subsection{Asymptotically Irreducibly Representable Cycles}\label{sec_air}
In this section we motivate and define the notion of $(\epsilon,\overline{\mathcal{M}})$-irreducibility and asymptotical irreducible representability for arithmetic cycles. Further, we discuss some examples, basic properties and applications. Especially, we will prove Proposition \ref{pro_intersection-finite}.

Let $\mathcal{X}$ be a generically smooth projective arithmetic variety.
Naively, one would call an effective arithmetic cycle $(\mathcal{Z},T)\in \widehat{Z}_{D}^p(\mathcal{X})$ irreducible if for any decomposition $(\mathcal{Z},T)=(\mathcal{Z}_1,T_1)+(\mathcal{Z}_2,T_2)$ into effective arithmetic cycles $(\mathcal{Z}_1,T_1)$ and $(\mathcal{Z}_2,T_2)$ either $(\mathcal{Z}_1,T_1)=0$ or $(\mathcal{Z}_2,T_2)=0$. By this definition the irreducible cycles are exactly the arithmetic cycles of the form $(\mathcal{Z},0)$, where $\mathcal{Z}$ is irreducible. But if $\mathcal{Z}$ is horizontal, the cycle $(\mathcal{Z},0)$ is not really closed in the sense of arithmetic intersection theory, as it does not intersect the fiber at infinity. To avoid this issue one may restrict to the group $\widehat{Z}^p(\mathcal{X})$ of arithmetic cycles of Green type. In many situations this is advantageous. For example if $p=1$, its group of rational equivalent classes $\widehat{CH}^1(\mathcal{X})$ is isomorphic to the group $\widehat{\mathrm{Pic}}(\mathcal{X})$ of isomorphism classes of hermitian line bundles. 

We may again define an effective arithmetic cycle $(\mathcal{Z},g_{\mathcal{Z}})\in\widehat{Z}^p(\mathcal{X})$ of Green type to be irreducible if for any decomposition $(\mathcal{Z},g_{\mathcal{Z}})=(\mathcal{Z}_1,g_{1})+(\mathcal{Z}_2,g_{2})$ into effective arithmetic cycles $(\mathcal{Z}_1,g_{1})$ and $(\mathcal{Z}_2,g_{2})$ of Green type either $(\mathcal{Z}_1,g_{1})=0$ or $(\mathcal{Z}_2,g_{2})=0$. Then all irreducible elements in $\widehat{Z}^1(\mathcal{X})$ are of the form $(\mathcal{Z},0)$ for some vertical irreducible divisor $\mathcal{Z}$. Indeed, if $(\mathcal{Z},g_{\mathcal{Z}})\in \widehat{Z}^1(\mathcal{X})$ has a non-zero Green current $g_{\mathcal{Z}}$, we can subtract a non-negative and non-zero $C^{\infty}$-function $f\colon \mathcal{X}(\mathbb{C})\to\mathbb{R}$ such that $g_\mathcal{Z}-f$ is non-negative and another Green current for $\mathcal{Z}$. Then $(\mathcal{Z},g_{\mathcal{Z}})=(\mathcal{Z},g_{\mathcal{Z}}-f)+(0,f)$ is a decomposition into two non-trivial effective arithmetic cycles, such that $(\mathcal{Z},g_{\mathcal{Z}})$ is not irreducible. 
Thus, this definition does not allow an analogue of Bertini's Theorem, as for an arithmetically ample hermitian line bundle $\overline{\mathcal{L}}$ every representing arithmetic divisor $\widehat{\mathrm{div}}(\overline{\mathcal{L}},s)$ for some small section $s\in \widehat{H}^0(\mathcal{X},\overline{\mathcal{L}})$ is not irreducible in this sense.

To get a notion of irreducibility which is more compatible with classical results as Bertini's Theorem, we instead measure how far an arithmetic cycle of Green type is away from being irreducible. Thus, an arithmetic cycle of Green type $(\mathcal{Z},g_\mathcal{Z})\in \widehat{Z}_p(\mathcal{X})$ should be nearly irreducible if $\mathcal{Z}$ is irreducible and the degree of the part at infinity $(\overline{\mathcal{M}}^p\cdot (0,g_\mathcal{Z}))$ is very small compared to the degree of the classical part $(\overline{\mathcal{M}}^p\cdot (\mathcal{Z},0))$. Here, $\overline{\mathcal{M}}$ denotes any arithmetically ample hermitian line bundle. Let us put this into a formal definition.
\begin{Def}
Let $\mathcal{X}$ be any generically smooth projective arithmetic variety.
\begin{enumerate}[(i)]
	\item
	Let $(\mathcal{Z},g_{\mathcal{Z}})\in \widehat{Z}_{p}(\mathcal{X})$ be any effective arithmetic cycle on $\mathcal{X}$. For any positive real number $\epsilon>0$ and any arithmetically ample hermitian line bundle $\overline{\mathcal{M}}$ we say that $(\mathcal{Z},g_{\mathcal{Z}})$ is \emph{$(\epsilon,\overline{\mathcal{M}})$-irreducible} if $\mathcal{Z}$ is irreducible, we have
	\begin{align}\label{equ_air2}
		\int_{\mathcal{X}(\mathbb{C})}g_{\mathcal{Z}}\wedge c_1(\overline{\mathcal{M}})^{p}< \epsilon\cdot (\overline{\mathcal{M}}^p\cdot(\mathcal{Z},0))
	\end{align}
	and $\frac{i}{2\pi}\partial\overline{\partial}g_{\mathcal{Z}}+\delta_{\mathcal{Z}(\mathbb{C})}$ is represented by a semi-positive form.
	\item We say that a class $\alpha\in\widehat{\mathrm{CH}}_p(\mathcal{X})$ in the arithmetic Chow group of $\mathcal{X}$ is \emph{air (asymptotically irreducibly representable)} if for any positive real number $\epsilon>0$ and any arithmetically ample hermitian line bundle $\overline{\mathcal{M}}$ there exists an $n\in\mathbb{Z}_{\ge 1}$ such that $n\cdot \alpha$ can be represented by an $(\epsilon,\overline{\mathcal{M}})$-irreducible arithmetic cycle $(\mathcal{Z},g_{\mathcal{Z}})$. 
	We call $\alpha$ \emph{generically smoothly air} if the arithmetic cycle $(\mathcal{Z},g_{\mathcal{Z}})$ can always be chosen such that $\mathcal{Z}$ is horizontal and $\mathcal{Z}_\mathbb{Q}$ is smooth.
\end{enumerate}
\end{Def}
Note that it is sufficient to check (ii) only for small $\epsilon$. If $\mathcal{Z}$ is a irreducible horizontal cycle, $(\overline{\mathcal{M}}^p\cdot(\mathcal{Z},0))=(\overline{\mathcal{M}}|_{\mathcal{Z}}^p)$ is just the $p$-th arithmetic self-intersection of $\overline{\mathcal{M}}$ restricted to $\mathcal{Z}$ by Lemma \ref{lem_intersection-Z0}. The assumption of the semi-positivity of $\frac{i}{2\pi}\partial\overline{\partial}g_{\mathcal{Z}}+\delta_{\mathcal{Z}(\mathbb{C})}$ in (i) ensures that there are no ``local maxima'' of $g_{\mathcal{Z}}$ outside $\mathcal{Z}(\mathbb{C})$. For example, if $\mathcal{Z}$ is of codimension $1$, $g_{\mathcal{Z}}$ is just a real valued function on $\mathcal{X}(\mathbb{C})\setminus\mathcal{Z}(\mathbb{C})$ with pole at $\mathcal{Z}(\mathbb{C})$. By the semi-positivity of $\frac{i}{2\pi}\partial\overline{\partial}g_{\mathcal{Z}}+\delta_{\mathcal{Z}_{\mathbb{C}}}$ we get that $g_{\mathcal{Z}}$ is pluri-subharmonic on $\mathcal{X}(\mathbb{C})\setminus\mathcal{Z}(\mathbb{C})$ and hence, it has no local maxima on $\mathcal{X}(\mathbb{C})\setminus\mathcal{Z}(\mathbb{C})$.
 Let us consider some examples.
\begin{Exam}
	\begin{enumerate}[(i)]
		\item A trivial example of an $(\epsilon,\overline{\mathcal{M}})$-irreducible cycle for any $\epsilon>0$ and any arithmetically ample hermitian line bundle $\overline{\mathcal{M}}$ is $(\mathcal{X},0)\in\widehat{Z}^0(\mathcal{X})$ as by the arithmetic ampleness of $\overline{\mathcal{M}}$ we have $(\overline{\mathcal{M}}^d)>0$ with $d=\dim \mathcal{X}$, see \cite[Proposition 5.39]{Mor14}, and 
		$$0=\int_{\mathcal{X}(\mathbb{C})}0\cdot c_1(\overline{\mathcal{M}})^d<\epsilon\cdot (\overline{\mathcal{M}}^d\cdot(\mathcal{X},0))=\epsilon\cdot (\overline{\mathcal{M}}^d).$$
		In particular, the class of $(\mathcal{X},0)$ in $\widehat{\mathrm{CH}}^0(\mathcal{X})$ is generically smoothly air.
		\item If $\mathcal{Z}\subseteq \mathcal{X}$ is an irreducible vertical cycle, then $(\mathcal{Z},0)\in\widehat{Z}_p(\mathcal{X})$ is $(\epsilon,\overline{\mathcal{M}})$-irreducible for any $\epsilon>0$ and any arithmetically ample hermitian line bundle $\overline{\mathcal{M}}$, as the left hand side of (\ref{equ_air2}) is $0$ and the right hand side is positive by Lemma \ref{lem_integral-bound}. Thus, the class of $(\mathcal{Z},0)$ in $\widehat{\mathrm{CH}}_p(\mathcal{X})$ is air but not generically smoothly air as $(\mathcal{Z},0)$ is never rational equivalent to a horizontal arithmetic cycle.
		\item We will see in Theorem \ref{thm_air-text} that any arithmetically ample hermitian line bundle $\overline{\mathcal{L}}\in\widehat{\mathrm{Pic}}(\mathcal{X})\cong \widehat{\mathrm{CH}}^1(\mathcal{X})$ is generically smoothly air.
	\end{enumerate}
\end{Exam}

The following proposition is useful to check the property of $(\epsilon,\overline{\mathcal{M}})$-irreducibility and of airness.
\begin{Pro}\label{pro_airness-check}
	Let $\mathcal{X}$ be a generically smooth projective arithmetic variety.
	\begin{enumerate}[(a)]
		\item Let $(\mathcal{Z},g_{\mathcal{Z}})\in \widehat{Z}_p(\mathcal{X})$ be an effective arithmetic cycle, such that $\mathcal{Z}$ is irreducible and $\frac{i}{2\pi}\partial\overline{\partial}g_{\mathcal{Z}}+\delta_{\mathcal{Z}(\mathbb{C})}$ is represented by a semi-positive form. Further, let $\epsilon>0$ and $\overline{\mathcal{M}}$ be any arithmetically ample hermitian line bundle. The following are equivalent:
		\begin{enumerate}[(i)]
			\item The arithmetic cycle $(\mathcal{Z},g_{\mathcal{Z}})$ is $(\epsilon,\overline{\mathcal{M}})$-irreducible.
			\item It holds
			$$\int_{\mathcal{X}(\mathbb{C})} g_{\mathcal{Z}}\wedge c_1(\overline{\mathcal{M}})^p<\frac{\epsilon}{1+\epsilon}(\overline{\mathcal{M}}^p\cdot (\mathcal{Z},g_{\mathcal{Z}})).$$			
		\end{enumerate}
		\item Let $\alpha\in \widehat{\mathrm{CH}}_p(\mathcal{X})$ and $\overline{\mathcal{N}}$ be any fixed arithmetically ample hermitian line bundle on $\mathcal{X}$. The following are equivalent
		\begin{enumerate}[(i)]
			\item The class $\alpha$ is air.
			\item For every $\epsilon>0$ there exists an $n\in\mathbb{Z}_{\ge 1}$ such that $n\cdot \alpha$ can be represented by an effective arithmetic cycle $(\mathcal{Z},g_{\mathcal{Z}})$ with $\mathcal{Z}$ irreducible, satisfying
			$$\int_{\mathcal{X}(\mathbb{C})} g_{\mathcal{Z}}\wedge c_1(\overline{\mathcal{N}})^p<n\epsilon$$
			and $\frac{i}{2\pi}\partial\overline{\partial}g_{\mathcal{Z}}+\delta_{\mathcal{Z}(\mathbb{C})}\ge 0$ as a form.
		\end{enumerate}
		This equivalence stays true if we assume $\alpha$ in (i) to be generically smoothly air and $\mathcal{Z}_{\mathbb{Q}}$ in (ii) to be smooth.
	\end{enumerate}
\end{Pro}
\begin{proof}
\begin{enumerate}[(a)]
	\item This follows directly from Lemma \ref{lem_intersection-restriction}. 
	\item We first show (i)$\Rightarrow$(ii):
	As $\alpha$ is air, we can apply (a) for any $\epsilon>0$ to some multiple $n\cdot\alpha$ of $\alpha$ to obtain that $(\overline{\mathcal{N}}^p\cdot(n\cdot\alpha))>0$ and hence, also $(\overline{\mathcal{N}}^p\cdot\alpha)>0$. It is enough to prove (ii) for sufficiently small $\epsilon$. Thus, we may assume $\epsilon<(\overline{\mathcal{N}}^p\cdot\alpha)$. We set $\epsilon'=\frac{\epsilon}{(\overline{\mathcal{N}}^p\cdot \alpha)-\epsilon}>0$. By the airness of $\alpha$ there is an $n\in\mathbb{Z}_{\ge 1}$ such that $n\cdot \alpha$ is representable by an $(\epsilon',\overline{\mathcal{N}})$-irreducible arithmetic cycle $(\mathcal{Z},g_{\mathcal{Z}})\in \widehat{Z}_p(\mathcal{X})$.
	This means that $\mathcal{Z}$ is irreducible, it holds $g_{\mathcal{Z}}\ge 0$ and $\frac{i}{2\pi}\partial\overline{\partial}g_{\mathcal{Z}}+\delta_{\mathcal{Z}(\mathbb{C})}\ge 0$, and by (a)
	$$\int_{\mathcal{X}(\mathbb{C})} g_{\mathcal{Z}}c_1(\overline{\mathcal{N}})^p<\frac{\epsilon'}{1+\epsilon'} (\overline{\mathcal{N}}^p\cdot(\mathcal{Z},g_{\mathcal{Z}}))=\frac{\epsilon}{(\overline{\mathcal{N}}^p\cdot\alpha)}(\overline{\mathcal{N}}^p\cdot(n\cdot\alpha))=n\epsilon.$$
	Thus we get (ii).
	
	Now we show (ii)$\Rightarrow$(i). Let $\epsilon>0$ be arbitrary and $\overline{\mathcal{M}}$ any arithmetically ample hermitian line bundle on $\mathcal{X}$. We have to show that there is an $n\in\mathbb{Z}_{\ge 1}$ such that $n\cdot \alpha$ is representable by an $(\epsilon,\overline{\mathcal{M}})$-irreducible arithmetic cycle. Since $\mathcal{X}(\mathbb{C})$ is compact and $c_1(\overline{\mathcal{N}})$ is positive, there exists an $A>0$ such that
	$$c_1(\overline{\mathcal{M}})\le A\cdot c_1(\overline{\mathcal{N}}).$$
	By (ii) there exists an $n\in\mathbb{Z}_{\ge 1}$ such that $n\cdot \alpha$ is representable by an effective and non-trivial arithmetic cycle $(\mathcal{Z},g_{\mathcal{Z}})$. Thus, by Lemma \ref{lem_integral-bound} we get
	$$(\overline{\mathcal{M}}^p\cdot \alpha)=\tfrac{1}{n}(\overline{\mathcal{M}}^p\cdot (\mathcal{Z},g_{\mathcal{Z}}))>\tfrac{1}{n}\int_{\mathcal{X}(\mathbb{C})}g_{\mathcal{Z}}c_1(\overline{\mathcal{M}})^p\ge 0,$$
	where the last inequality follows by $g_{\mathcal{Z}}\ge 0$ and $c_1(\overline{\mathcal{M}})>0$. Now we can apply (ii) to $$\epsilon'=A^{-p}(\overline{\mathcal{M}}^p\cdot \alpha)\frac{\epsilon}{1+\epsilon}>0.$$
	This implies that there is an $n\in\mathbb{Z}_{\ge 1}$ such that $n\cdot \alpha$ is representable by an effective arithmetic cycle $(\mathcal{Z},g_{\mathcal{Z}})$ with $\mathcal{Z}$ irreducible, $\frac{i}{2\pi}\partial\overline{\partial}g_{\mathcal{Z}}+\delta_{\mathcal{Z}(\mathbb{C})}\ge 0$ and 
	$$\int_{\mathcal{X}(\mathbb{C})}g_{\mathcal{Z}}c_1(\overline{\mathcal{N}})^p< n\epsilon'.$$
	Finally, we can bound
	$$\int_{\mathcal{X}(\mathbb{C})}g_{\mathcal{Z}}c_1(\overline{\mathcal{M}})^p\le A^p\int_{\mathcal{X}(\mathbb{C})}g_{\mathcal{Z}}c_1(\overline{\mathcal{N}})^p<A^pn\epsilon'=\frac{\epsilon}{1+\epsilon}(\overline{\mathcal{M}}^p\cdot(\mathcal{Z},g_{\mathcal{Z}})).$$
	By (a) this means that $(\mathcal{Z},g_{\mathcal{Z}})$ is $(\epsilon,\overline{\mathcal{M}})$-irreducible as desired.
	
	If we assume in (b) that $\alpha$ is generically smoothly air and that $\mathcal{Z}_{\mathbb{Q}}$ is smooth, the proof above can be easily adapted.
\end{enumerate}	
\end{proof}
The following proposition shows how the airness of a class $\alpha\in \widehat{\mathrm{CH}}_p(\mathcal{X})$ can be used to compute arithmetic intersection numbers by considering only the classical part $\mathcal{Z}$ of an effective arithmetic cycle $(\mathcal{Z},g_{\mathcal{Z}})$ of Green type representing $\alpha$. For any $\alpha\in \widehat{\mathrm{CH}}_p(\mathcal{X})$ we write $(\alpha)_{\mathrm{Eff}}\subseteq \widehat{Z}_p(\mathcal{X})$ for the subset of arithmetic cycles of Green type which are effective and represent $\alpha$. One may consider $(\alpha)_{\mathrm{Eff}}$ as an analogue of a complete linear system in classical algebraic geometry.
\begin{Pro}[= Proposition \ref{pro_intersection-finite}]\label{pro_intersection-finite-text}
	Let $\mathcal{X}$ be any generically smooth projective arithmetic variety. 
	For any arithmetic cycle $\alpha\in\widehat{\mathrm{CH}}_p(\mathcal{X})$, which is air, and any arithmetically ample hermitian line bundles $\mathcal{L}_1,\dots,\mathcal{L}_{p}$ we can compute their arithmetic intersection number in the following way
	$$(\overline{\mathcal{L}}_1\cdots\overline{\mathcal{L}}_{p}\cdot\alpha)=\limsup_{n\to\infty}\sup_{(\mathcal{Z},g_{\mathcal{Z}}) \in (n\alpha)_{\mathrm{Eff}}} \tfrac{1}{n}(\overline{\mathcal{L}}_1\cdots\overline{\mathcal{L}}_{p}\cdot (\mathcal{Z},0)).$$
\end{Pro}
\begin{proof}
	For every $(\mathcal{Z},g_{\mathcal{Z}})\in \widehat{Z}_p(\mathcal{X})$ representing $n\alpha$ for some $n\in\mathbb{Z}_{\ge 1}$ we have by Lemma \ref{lem_intersection-restriction}
	\begin{align}\label{equ_intersection-split}
		n(\overline{\mathcal{L}}_1\cdots\overline{\mathcal{L}}_{p}\cdot\alpha)=(\overline{\mathcal{L}}_1\cdots\overline{\mathcal{L}}_{p}\cdot(\mathcal{Z},0))+\int_{\mathcal{X}(\mathbb{C})}g_{\mathcal{Z}}\wedge c_1(\overline{\mathcal{L}}_1)\wedge\cdots\wedge c_1(\overline{\mathcal{L}}_p).
	\end{align}
	If $(\mathcal{Z},g_{\mathcal{Z}})\in (n\alpha)_{\mathrm{Eff}}$, then $g_{\mathcal{Z}}$ is semi-positive and hence,
	$$\int_{\mathcal{X}(\mathbb{C})}g_{\mathcal{Z}}\wedge c_1(\overline{\mathcal{L}}_1)\wedge\cdots\wedge c_1(\overline{\mathcal{L}}_p)\ge 0$$
	by the positivity of $c_1(\overline{\mathcal{L}}_1),\dots, c_1(\overline{\mathcal{L}}_p)$. If we apply this to Equation (\ref{equ_intersection-split}), we can deduce
	\begin{align}\label{equ_inequality-intersection-limsup}
		(\overline{\mathcal{L}}_1\cdots\overline{\mathcal{L}}_{p}\cdot\alpha)\ge\limsup_{n\to\infty}\sup_{(\mathcal{Z},g_{\mathcal{Z}}) \in (n\alpha)_{\mathrm{Eff}}} \tfrac{1}{n}(\overline{\mathcal{L}}_1\cdots\overline{\mathcal{L}}_{p}\cdot (\mathcal{Z},0)).
	\end{align}
	To prove the inequality in the other direction let $\epsilon>0$ be arbitrary. By Proposition \ref{pro_airness-check} (b) there is an $n_0\in\mathbb{Z}_{\ge 1}$ and an arithmetic cycle $(\mathcal{Z},g_{\mathcal{Z}})\in (n_0\alpha)_{\mathrm{Eff}}$ with
	$$\int_{\mathcal{X}(\mathbb{C})}g_\mathcal{Z}\wedge c_1(\overline{\mathcal{L}}_1)^p<n_0\epsilon.$$
	By the positivity of $c_1(\overline{\mathcal{L}}_1)$ and the compactness of $\mathcal{X}(\mathbb{C})$ there is a constant $A>0$ such that
	$$c_1(\overline{\mathcal{L}}_1)\wedge\dots\wedge c_1(\overline{\mathcal{L}}_p)\le A\cdot c_1(\overline{\mathcal{L}}_1)^p.$$
	Thus, we get
	$$\int_{\mathcal{X}(\mathbb{C})}g_{\mathcal{Z}}\wedge c_1(\overline{\mathcal{L}}_1)\wedge\dots\wedge c_1(\overline{\mathcal{L}}_p)\le A n_0\epsilon.$$
	Applying this to Equation (\ref{equ_intersection-split}) we get
	$$(\overline{\mathcal{L}}_1\cdots\overline{\mathcal{L}}_{p}\cdot\alpha)\le \sup_{(\mathcal{Z},g_{\mathcal{Z}})\in (n_0\alpha)_{\mathrm{Eff}}} \tfrac{1}{n_0}(\overline{\mathcal{L}}_1\cdots\overline{\mathcal{L}}_p\cdot (\mathcal{Z},0))+A\epsilon.$$
	We can always multiply an arithmetic cycle $(\mathcal{Z},g_{\mathcal{Z}})\in (n_0\alpha)_{\mathrm{Eff}}$ with any $m\in\mathbb{Z}_{\ge 1}$ to get the arithmetic cycle $(m\mathcal{Z},mg_{\mathcal{Z}})\in (mn_0\alpha)_{\mathrm{Eff}}$. Thus,
	$$\sup_{(\mathcal{Z},g_{\mathcal{Z}})\in (n_0\alpha)_{\mathrm{Eff}}} \tfrac{1}{n_0}(\overline{\mathcal{L}}_1\cdots\overline{\mathcal{L}}_p\cdot (\mathcal{Z},0))\le \sup_{(\mathcal{Z},g_{\mathcal{Z}})\in (mn_0\alpha)_{\mathrm{Eff}}} \tfrac{1}{mn_0}(\overline{\mathcal{L}}_1\cdots\overline{\mathcal{L}}_p\cdot (\mathcal{Z},0))$$
	for every $m\in \mathbb{Z}_{\ge 0}$. Applying this to the inequality before we get
	$$(\overline{\mathcal{L}}_1\cdots\overline{\mathcal{L}}_{p}\cdot\alpha)\le \limsup_{n\to\infty}\sup_{(\mathcal{Z},g_{\mathcal{Z}})\in (n\alpha)_{\mathrm{Eff}}} \tfrac{1}{n}(\overline{\mathcal{L}}_1\cdots\overline{\mathcal{L}}_p\cdot (\mathcal{Z},0))+A\epsilon.$$
	If now $\epsilon$ tends to $0$, we get the other direction of inequality (\ref{equ_inequality-intersection-limsup}) and thus, we obtain the equality claimed in the proposition.
\end{proof}
Let us discuss some questions related to air classes which may be addressed in future research.
\begin{Que}
	Let $\mathcal{X}$ be a generically smooth projective arithmetic variety.
	\begin{enumerate}[(i)]
		\item As we will see in Theorem \ref{thm_air-text}, every arithmetically ample hermitian line bundle is generically smoothly air. As every hermitian line bundle $\overline{\mathcal{L}}$ can be written as a difference $\overline{\mathcal{L}}=\overline{\mathcal{L}}_1-\overline{\mathcal{L}}_2$ of two arithmetically ample hermitian line bundles $\overline{\mathcal{L}}_1$, $\overline{\mathcal{L}}_2$, one may ask whether for any $p$ every class $\alpha\in \widehat{\mathrm{CH}}_p(\mathcal{X})$ can be written as difference $\alpha=\alpha_1-\alpha_2$ of two (generically smoothly) air classes $\alpha_1,\alpha_2\in\widehat{\mathrm{CH}}_p(\mathcal{X})$. For a weaker statement, one may also just ask whether $\widehat{\mathrm{CH}}_p(\mathcal{X})$ is generated by (generically smoothly) air classes.
		\item Is the sum $\alpha_1+\alpha_2$ of two (generically smoothly) air classes $\alpha_1,\alpha_2\in\widehat{\mathrm{CH}}_p(\mathcal{X})$ again (generically smoothly) air?
		\item If $\mathcal{X}$ is regular, every two classes $\alpha_1\in\widehat{\mathrm{CH}}^p(\mathcal{X})$ and $\alpha_2\in\widehat{\mathrm{CH}}^q(\mathcal{X})$ have an arithmetic intersection product $\alpha_1\cdot\alpha_2\in\widehat{\mathrm{CH}}^{p+q}(\mathcal{X})_{\mathbb{Q}}$ defined by Gillet and Soul\'e \cite{GS90}. Is the product $\alpha_1\cdot \alpha_2$ (generically smoothly) air if $\alpha_1$ and $\alpha_2$ are (generically smoothly) air? Note that the definition of ``(generically smoothly) air'' directly extends to $\mathbb{Q}$-classes in $\widehat{\mathrm{CH}}_p(\mathcal{X})_{\mathbb{Q}}$.
	\end{enumerate}
\end{Que}
\subsection{A Bertini-type Result}\label{sec_bertini}
One formulation of Bertini's Theorem states that on a projective variety of dimension at least $2$ the irreducible divisors representing a very ample line bundle form a dense and open subset of all divisors representing this line bundle. As a weaker but in practice still very useful statement one gets that every ample line bundle has a tensor power which can be represented by an irreducible divisor. In this section we will prove the analogue in our setting: Every arithmetically ample hermitian line bundle is air. This forms the main theorem in our theory of air classes. As an application, we will discuss the computation of arithmetic intersection numbers of arithmetically ample hermitian line bundles.
In particular, we prove Theorems \ref{thm_air} and \ref{thm_intersection-line-bundles-geometric} in this section.

Let us start with our main result.
\begin{Thm}[= Theorem \ref{thm_air}]\label{thm_air-text}
	Let $\mathcal{X}$ be any generically smooth projective arithmetic variety of dimension $d\ge 2$. Every arithmetically ample hermitian line bundle $\overline{\mathcal{L}}$ on $\mathcal{X}$ is generically smoothly air. Moreover, if $\alpha\in\widehat{\mathrm{CH}}_p(\mathcal{X})$ is any generically smoothly air class for some $p\ge 2$, then also the class $\overline{\mathcal{L}}\cdot \alpha\in\widehat{\mathrm{CH}}_{p-1}(\mathcal{X})$ is generically smoothly air.
\end{Thm}
\begin{proof}
	The first assertion follows from the second assertion by setting $\alpha=(\mathcal{X},0)$. Hence, we only have to prove the second assertion. We fix an arbitrary $\epsilon>0$.
	By Proposition \ref{pro_airness-check} (b) it is enough to show that there is some $n_0\in\mathbb{Z}_{\ge 1}$ such that $n_0\cdot \overline{\mathcal{L}}\cdot \alpha$ is representable by an effective arithmetic cycle $(\mathcal{Z},g_\mathcal{Z})\in \widehat{Z}_{p-1}(\mathcal{X})$ with $\mathcal{Z}$ irreducible and horizontal, $\mathcal{Z}_{\mathbb{Q}}$ smooth, $\frac{i}{2\pi}\partial\overline{\partial}g_{\mathcal{Z}}+\delta_{\mathcal{Z}(\mathbb{C})}\ge 0$ and
	$$\int_{\mathcal{X}(\mathbb{C})} g_{\mathcal{Z}}\wedge c_1(\overline{\mathcal{L}})^{p-1}<n_0\epsilon.$$
	Since $\alpha$ is generically smoothly air, their exists an $n_1\in\mathbb{Z}_{\ge 1}$ such that $n_1\cdot \alpha$ is representable by an effective arithmetic cycle $(\mathcal{Z}_1,g_{\mathcal{Z}_1})\in \widehat{Z}_p(\mathcal{X})$ with $\mathcal{Z}_1$ irreducible and horizontal, $\mathcal{Z}_{1,\mathbb{Q}}$ smooth, $\frac{i}{2\pi}\partial\overline{\partial}g_{\mathcal{Z}_1}+\delta_{\mathcal{Z}_1(\mathbb{C})}\ge 0$ and 
	$$\int_{\mathcal{X}(\mathbb{C})} g_{\mathcal{Z}_1}\wedge c_1(\overline{\mathcal{L}})^{p}<\frac{n_1\epsilon}{2}.$$
	We write $\overline{\mathcal{L}}_1=\overline{\mathcal{L}}|_{\mathcal{Z}_1}$ for the arithmetically ample hermitian line bundle on $\mathcal{Z}_1$ obtained by restricting $\overline{\mathcal{L}}$ to $\mathcal{Z}_1$.
	
	Let $T_n\subseteq \widehat{H}^0(\mathcal{Z}_1,\overline{\mathcal{L}}_1^{\otimes n})$ denote the subset of small sections $s\in \widehat{H}^0(\mathcal{Z}_1,\overline{\mathcal{L}}_1^{\otimes n})$ of $\overline{\mathcal{L}}_1^{\otimes n}$ such that $\mathrm{div}(s)$ is irreducible and $\mathrm{div}(s)_\mathbb{Q}$ is smooth. By the arithmetic Bertini theorem by Charles \cite[Theorems 1.6 and 1.7]{Cha21} it holds
	$$\lim_{n\to \infty}\frac{\#T_n}{\#\widehat{H}^0(\mathcal{Z}_1,\overline{\mathcal{L}}_1^{\otimes n})}=1.$$
	As explained in the proof of Theorem \ref{thm_equidistribution-text}, the sets
	$$K_n=\left.\left\{s\in H^0(\mathcal{Z}_1,\overline{\mathcal{L}}_1^{\otimes n})_{\mathbb{R}}~\right|~\|s\|_{\sup}\le 1\right\}$$
	satisfy the assumptions in Theorem \ref{thm_equidistribution-general} for $\tau=1$. Hence, we can apply Proposition \ref{pro_positive-density} to the sets
	$$R_n=K_n\cap H^0\left(\mathcal{Z}_1,\overline{\mathcal{L}}_1^{\otimes n}\right)=\widehat{H}^0\left(\mathcal{Z}_1,\overline{\mathcal{L}}_1^{\otimes n}\right).$$
	Thus, there exists a sequence $(s_n)_{n\in\mathbb{Z}_{\ge 1}}$ of sections $s_n\in T_n$ such that
	$$\lim_{n\to\infty}\frac{1}{n}\int_{\mathcal{Z}_1(\mathbb{C})}\left(-\log|s_n|_{\overline{\mathcal{L}}_1^{\otimes n}}\right)c_1(\overline{\mathcal{L}}_1)^{p-1}=0.$$
	Note that $\left|\log|s_n|_{\overline{\mathcal{L}}_1^{\otimes n}}\right|=-\log|s_n|_{\overline{\mathcal{L}}_1^{\otimes n}}$ for $s_n\in T_n$.
	Thus, there is an $n_2\in \mathbb{Z}_{\ge 1}$ such that 
	\begin{align}\label{equ_bound-integral}
	\frac{1}{n}\int_{\mathcal{Z}_1(\mathbb{C})}\left(-\log|s_n|_{\overline{\mathcal{L}}_1^{\otimes n}}\right)c_1(\overline{\mathcal{L}}_1)^{p-1}<\frac{\epsilon}{2}
	\end{align}
	for all $n\ge n_2$.
	
	Since $\mathcal{L}$ is ample on $\mathcal{X}$, there exists an $n_3\in\mathbb{Z}_{\ge 1}$ such that the restriction map
	$$H^0(\mathcal{X},\mathcal{L}^{\otimes n})\to H^0(\mathcal{Z}_1,\mathcal{L}_1^{\otimes n})$$
	is surjective for all $n\ge n_3$.
	Now let $n=\max\{n_2,n_3\}$ and let $\widetilde{s}_n\in H^0(\mathcal{X},\mathcal{L}^{\otimes n})$ be a lift of $s_n$ under the restriction map. We set 
	$$\mathcal{Z}=\mathrm{div}(\widetilde{s}_n)\cdot \mathcal{Z}_1,\qquad g_{\mathcal{Z}}=\left[-\log|\widetilde{s}_n|_{\overline{\mathcal{L}}^{\otimes n}}\right]_{\mathcal{Z}_1(\mathbb{C})}+c_1(\overline{\mathcal{L}}^{\otimes n})\wedge g_{\mathcal{Z}_1}.$$
	Then $(\mathcal{Z},g_\mathcal{Z})=\widehat{\mathrm{div}}\left(\overline{\mathcal{L}}^{\otimes n},\widetilde{s}_n\right)\cdot (\mathcal{Z}_1,g_{\mathcal{Z}_1})$ represents $(n\cdot \overline{\mathcal{L}})\cdot (n_1\cdot\alpha)=n_0\cdot\overline{\mathcal{L}}\cdot \alpha$ with $n_0=nn_1$. Moreover, $\mathcal{Z}=\mathrm{div}(s_n)$ is irreducible, $\mathcal{Z}_{\mathbb{Q}}=\mathrm{div}(s_n)_{\mathbb{Q}}$ is smooth, and it holds $g_{\mathcal{Z}}\ge 0$, such that $(\mathcal{Z},g_\mathcal{Z})\in \widehat{Z}_{p-1}(\mathcal{X})$ is an effective arithmetic cycle.
	Further, we have
	\begin{align*}
		&\frac{i}{2\pi}\partial\overline{\partial}g_{\mathcal{Z}}+\delta_{\mathcal{Z}(\mathbb{C})}=\left(\frac{i\partial\overline{\partial}}{2\pi}(-\log|\widetilde{s}_n|_{\overline{\mathcal{L}}^{\otimes n}})\right)\wedge \delta_{\mathcal{Z}_1(\mathbb{C})}+c_1(\overline{\mathcal{L}}^{\otimes n})\wedge \frac{i\partial\overline{\partial}}{2\pi}g_{\mathcal{Z}_1}+\delta_{\mathcal{Z}(\mathbb{C})}\\
		&=c_1(\overline{\mathcal{L}}^{\otimes n})\wedge\delta_{\mathcal{Z}_1(\mathbb{C})}-\delta_{\mathrm{div}(\widetilde{s}_n)(\mathbb{C})}\wedge\delta_{\mathcal{Z}_1(\mathbb{C})}+c_1(\overline{\mathcal{L}}^{\otimes n})\wedge \frac{i\partial\overline{\partial}}{2\pi}g_{\mathcal{Z}_1}+\delta_{\mathcal{Z}(\mathbb{C})}\\
		&=c_1(\overline{\mathcal{L}}^{\otimes n})\wedge\left(\frac{i\partial\overline{\partial}}{2\pi}g_{\mathcal{Z}_1}+\delta_{\mathcal{Z}_1(\mathbb{C})}\right)\ge 0.
	\end{align*}
	Finally, we also get
		\begin{align*}
		&\int_{\mathcal{X}(\mathbb{C})}g_{\mathcal{Z}}\wedge c_1(\overline{\mathcal{L}})^{p-1}\\
		&=\int_{\mathcal{Z}_1(\mathbb{C})}\left(-\log|\widetilde{s}_n|_{\overline{\mathcal{L}}^{\otimes n}}\right)c_1(\overline{\mathcal{L}})^{p-1}+\int_{\mathcal{X}(\mathbb{C})}c_1(\overline{\mathcal{L}}^{\otimes n})\wedge g_{\mathcal{Z}_1}\wedge c_1(\overline{\mathcal{L}})^{p-1}\\
		&=\int_{\mathcal{Z}_1(\mathbb{C})}\left(-\log|s_n|_{\overline{\mathcal{L}}_1^{\otimes n}}\right)c_1(\overline{\mathcal{L}}_1)^{p-1}+n\int_{\mathcal{X}(\mathbb{C})}g_{\mathcal{Z}_1}\wedge c_1(\overline{\mathcal{L}})^p\\
		&<n\cdot\tfrac{\epsilon}{2}+n n_1\cdot\tfrac{\epsilon}{2}\le n_0\epsilon.
	\end{align*}
	This completes the proof of the theorem.
\end{proof}
\begin{Rem}\label{rem_dimension1}
	\hspace{1pt}
	\begin{enumerate}[(i)]
		\item If $\mathcal{X}$ is a generically smooth projective arithmetic variety of dimension $d\ge 1$, $\alpha\in \widehat{\mathrm{CH}}_1(\mathcal{X})$ an air class and $\overline{\mathcal{L}}$ an arithmetically ample hermitian line bundle on $\mathcal{X}$, the proof of the Theorem still gives that for any $\epsilon>0$ there exists an $n\in\mathbb{Z}_{\ge 1}$ such that $n\cdot\overline{\mathcal{L}}\cdot\alpha$ can be represented by an arithmetic divisor $(\mathcal{Z},g_{\mathcal{Z}})$ such that 
		$$\int_{\mathcal{X}(\mathbb{C})}g_{\mathcal{Z}}< n\epsilon.$$
		\item For arithmetically ample hermitian line bundles $\overline{\mathcal{L}}$ on a generically smooth projective variety $\mathcal{X}$ of dimension $d\ge 2$ we can also deduce from Theorem \ref{thm_equidistribution-text} and the results by Charles \cite{Cha21} a more Bertini-like result: For every $\epsilon>0$ we have
		$$\lim_{n\to\infty}\frac{\#\left\{s\in \widehat{H}^0(\mathcal{X},\overline{\mathcal{L}}^{\otimes n})~|~\widehat{\mathrm{div}}(\overline{\mathcal{L}}^{\otimes n},s)~(\epsilon,\overline{\mathcal{M}})\text{-irreducible}, ~\mathrm{div}(s)_{\mathbb{Q}} \text{ smooth.}\right\}}{\#\widehat{H}^0(\mathcal{X},\overline{\mathcal{L}}^{\otimes n})}=1$$
		for any fixed arithmetically ample hermitian line bundle $\overline{\mathcal{M}}$.
	\end{enumerate}
\end{Rem}
Theorem \ref{thm_air-text} shows, that we can use Proposition \ref{pro_intersection-finite-text} to compute the arithmetic intersection numbers of arithmetically ample hermitian line bundles. Using induction one can express the intersection numbers of arithmetically ample hermitian line bundles completely by classical geometric intersection numbers. Note that by Remark \ref{rem_dimension1} (i) we can also treat the $1$-dimensional case.  For any arithmetically ample hermitian line bundles $\overline{\mathcal{L}}_1,\dots,\overline{\mathcal{L}}_d$ on $\mathcal{X}$ we get
$$(\overline{\mathcal{L}}_1\cdots\overline{\mathcal{L}}_d)=\limsup_{n\to \infty}\sup_{(\mathcal{Z},g_{\mathcal{Z}})\in (n\cdot\overline{\mathcal{L}}_1\cdots\overline{\mathcal{L}}_d)_{\mathrm{Eff}}}\tfrac{1}{n}\sum_{p\in|\mathrm{Spec}(\mathbb{Z})|}i_p(\mathcal{Z})\log p,$$
where $|\mathrm{Spec}(\mathbb{Z})|$ denotes the closed points of $\mathrm{Spec}(\mathbb{Z})$ and $i_p(\mathcal{Z})$ denotes the degree of the $0$-cycle $\mathcal{Z}$ on the fiber $\mathcal{X}_p$ of $\mathcal{X}$ over $p$. With a bit more effort but without using Proposition \ref{pro_intersection-finite-text} or Theorem \ref{thm_air-text} we can even prove the following more explicit formula.
\begin{Thm}[= Theorem \ref{thm_intersection-line-bundles-geometric}]
	Let $\mathcal{X}$ be any generically smooth projective arithmetic variety. Further, let $\mathcal{Y}\subseteq \mathcal{X}$ be any generically smooth arithmetic subvariety  of dimension $e\ge 1$.	
	For any arithmetically ample hermitian line bundles $\overline{\mathcal{L}}_1,\dots,\overline{\mathcal{L}}_e$ and any $n\in\mathbb{Z}_{\ge 1}$ we define
	$$H_n=\left\{\left.(s_1,\dots,s_e)\in\prod_{i=1}^e \widehat{H}^0\left(\mathcal{X},\overline{\mathcal{L}}_i^{\otimes n}\right)~\right|~\dim\left(\mathcal{Y}\cap\bigcap_{i=1}^e \mathrm{Supp}(\mathrm{div}(s_i))\right)=0\right\}.$$
	Then the arithmetic intersection number $(\overline{\mathcal{L}}_1|_{\mathcal{Y}}\cdots\overline{\mathcal{L}}_e|_{\mathcal{Y}})$ can be computed by
	\begin{align*}
		(\overline{\mathcal{L}}_1|_{\mathcal{Y}}\cdots\overline{\mathcal{L}}_e|_{\mathcal{Y}})=\lim_{n\to\infty}\frac{1}{n^e}\max_{(s_1,\dots,s_e)\in H_n}\sum_{p\in|\mathrm{Spec}(\mathbb{Z})|}i_p(\mathcal{Y}\cdot\mathrm{div}(s_1)\cdots\mathrm{div}(s_e))\log p
	\end{align*}
	where $i_p(\mathcal{Y}\cdot\mathrm{div}(s_1)\cdots\mathrm{div}(s_e))$ denotes the degree of the $0$-cycle $\mathcal{Y}\cdot\mathrm{div}(s_1)\cdots\mathrm{div}(s_e)$ in the fiber $\mathcal{X}_p$ of $\mathcal{X}$ over $p\in |\mathrm{Spec}(\mathbb{Z})|$.
\end{Thm}
\begin{proof}
	For any effective cycle $\mathcal{Z}$ of $\mathcal{X}$ of pure dimension $e$ we set
	$$H_n(\mathcal{Z})=\left\{\left.(s_1,\dots,s_e)\in\prod_{i=1}^e \widehat{H}^0\left(\mathcal{X},\overline{\mathcal{L}}_i^{\otimes n}\right)~\right|~\dim\left(\mathcal{Z}\cap\bigcap_{i=1}^e \mathrm{Supp}(\mathrm{div}(s_i))\right)=0\right\}.$$
	By induction on $e$ we will show the inequality
	\begin{align}\label{equ_sup-upper-bound}
		(\overline{\mathcal{L}}_1\cdots\overline{\mathcal{L}}_e\cdot(\mathcal{Z},0))\ge\sup_{\gf{n\in\mathbb{Z}_{\ge 1}}{(s_1,\dots,s_e)\in H_n(\mathcal{Z})}}\frac{1}{n^e}\sum_{p\in|\mathrm{Spec}(\mathbb{Z})|}i_p(\mathcal{Z}\cdot\mathrm{div}(s_1)\cdots\mathrm{div}(s_e))\log p
	\end{align}
	for any effective cycle $\mathcal{Z}$ of $\mathcal{X}$ of pure dimension $e$ and the inequality 
	 \begin{align}\label{equ_limit-lower-bound}
	 	&(\overline{\mathcal{L}}_1\cdots\overline{\mathcal{L}}_e\cdot(\mathcal{Y},0))\\
	 	&\le\lim_{n\to\infty}\frac{1}{n^e}\max_{(s_1,\dots,s_e)\in H_n(\mathcal{Y})}\sum_{p\in|\mathrm{Spec}(\mathbb{Z})|}i_p(\mathcal{Y}\cdot\mathrm{div}(s_1)\cdots\mathrm{div}(s_e))\log p\nonumber
	 \end{align}
 	for any generically smooth arithmetic subvariety $\mathcal{Y}\subseteq \mathcal{X}$ of dimension $e$.
 	
 	First, we consider the case $e=1$. If $s_1\in H_n(\mathcal{Z})$, we get by the definition of the arithmetic degree (\ref{equ_degree-definition}) and of the arithmetic intersection product (\ref{equ_intersection-definition}) that
 	\begin{align}\label{equ_case-dim-1}
 	n(\overline{\mathcal{L}}_1\cdot(\mathcal{Z},0))=(\overline{\mathcal{L}}^{\otimes n}_1\cdot(\mathcal{Z},0))=\sum_{p\in|\mathrm{Spec}(\mathbb{Z})|}i_p(\mathcal{Z}\cdot\mathrm{div}(s_1))\log p-\int_{\mathcal{Z}(\mathbb{C})}\log|s_1|_{\overline{\mathcal{L}}_1^{\otimes n}},
 	\end{align}
 	see also \cite[Proposition 5.23 (3)]{Mor14}.
 	Note that the integral is always non-positive as $\|s_1\|_{\sup}\le 1$. Thus, we get inequality (\ref{equ_sup-upper-bound}) in the case $e=1$. To show inequality (\ref{equ_limit-lower-bound}) we recall from Proposition \ref{pro_positive-density} applied to $R_n=\widehat{H}^0(\mathcal{X},\overline{\mathcal{L}}_1^{\otimes n})$ that there is a sequence $(s_{1,n})_{n\in \mathbb{Z}_{\ge1}}$ of sections $s_{1,n}\in \widehat{H}^0(\mathcal{X},\overline{\mathcal{L}}_1^{\otimes n})$ such that
 	$$\lim_{n\to \infty}\frac{1}{n}\int_{\mathcal{Y}(\mathbb{C})}\log |s_{1,n}|_{\overline{\mathcal{L}}_1^{\otimes n}}=0.$$
 	Note that for sufficiently large $n$ we automatically get $\dim\mathcal{Y}\cap \mathrm{Supp}(\mathrm{div}(s_{1,n}))=0$ and hence, $s_{1,n}\in H_n(\mathcal{Y})$. Applying this to Equation (\ref{equ_case-dim-1}) with $\mathcal{Z}=\mathcal{Y}$ we get
 	$$(\overline{\mathcal{L}}_1\cdot(\mathcal{Y},0))=\lim_{n\to \infty} \frac{1}{n}\sum_{p\in|\mathrm{Spec}(\mathbb{Z})|}i_p(\mathcal{Y}\cdot\mathrm{div}(s_{1,n}))\log p.$$
 	Thus, inequality (\ref{equ_limit-lower-bound}) follows for the case $e=1$.
 	
 	Now we assume $e\ge 2$. Let $(s_1,\dots, s_e)\in H_n(\mathcal{Z})$. Recall that by the definition of the arithmetic intersection product (\ref{equ_intersection-definition}) we have
 	$$\overline{\mathcal{L}}_e^{\otimes n}\cdot (\mathcal{Z},0)=(\mathrm{div}(s_e)\cdot \mathcal{Z}, [-\log|s_e|^2]_{\mathcal{Z}(\mathbb{C})}).$$
 	Thus, we can compute
 	\begin{align*}
 		&(\overline{\mathcal{L}}_1\cdots\overline{\mathcal{L}}_e\cdot(\mathcal{Z},0))=\frac{1}{n}(\overline{\mathcal{L}}_1\cdots\overline{\mathcal{L}}_{e-1}\cdot(\mathrm{div}(s_e)\cdot \mathcal{Z}, [-\log|s_e|^2]_{\mathcal{Z}(\mathbb{C})}))\\
 		&=\frac{1}{n}(\overline{\mathcal{L}}_1\cdots\overline{\mathcal{L}}_{e-1}\cdot(\mathrm{div}(s_e)\cdot \mathcal{Z}, 0))-\frac{1}{n}\int_{\mathcal{Z}(\mathbb{C})}\log|s_e|_{\overline{\mathcal{L}}_e^{\otimes n}}c_1(\overline{\mathcal{L}}_1)\cdots c_1(\overline{\mathcal{L}}_{e-1})\\
 		&\ge \frac{1}{n^e}\sum_{p\in|\mathrm{Spec}(\mathbb{Z})|}i_p(\mathcal{Z}\cdot\mathrm{div}(s_1)\cdots\mathrm{div}(s_e))\log p.
 	\end{align*}
 	For the last inequality we applied the induction hypothesis and that $\|s\|_{\sup}\le 1$. Thus, we get inequality (\ref{equ_sup-upper-bound}) also in dimension $e$.
 	
 	Next, we show inequality (\ref{equ_limit-lower-bound}) in dimension $e\ge 2$. It is enough to show that for all sufficiently small $\epsilon>0$ there exists an $n_0\in\mathbb{Z}_{\ge 1}$ such that for all $n\ge n_0$ there is a tuple $(s_1,\dots, s_e)\in H_n(\mathcal{Y})$ such that 
 	$$(\overline{\mathcal{L}}_1\cdots\overline{\mathcal{L}}_e\cdot(\mathcal{Y},0))\le \frac{1}{n^e}\sum_{p\in|\mathrm{Spec}(\mathbb{Z})|} i_p(\mathcal{Y}\cdot\mathrm{div}(s_1)\cdots \mathrm{div}(s_e))\log p+\epsilon.$$
 	To prove this we will apply the vanishing result in Proposition \ref{pro_positive-density}, the arithmetic Bertini theorem and the surjectivity result on the restriction map for small sections by Charles \cite{Cha21}, the arithmetic Hilbert--Samuel formula and the induction hypothesis.
 	
 	We set $c=\deg(\mathcal{L}_{1,\mathbb{C}}\cdots \mathcal{L}_{e-1,\mathbb{C}}\cdot \mathcal{Y}_{\mathbb{C}})$.
 	Replacing $\epsilon$ by a smaller value we can assume that $\overline{\mathcal{L}}_e(-\epsilon/(4c))$ is still arithmetically ample. 
 	If we denote the subset $T_n\subseteq \widehat{H}^0(\mathcal{Y},\overline{\mathcal{L}}_e(-\epsilon/(4c))|_{\mathcal{Y}}^{\otimes n})$ of small sections $s\in \widehat{H}^0(\mathcal{Y},\overline{\mathcal{L}}_e(-\epsilon/(4c))|_{\mathcal{Y}}^{\otimes n})$ such that $\mathrm{div}(s)$ is irreducible and $\mathrm{div}(s)_{\mathbb{Q}}$ is smooth, then we get by Charles \cite[Theorems 1.6 and 1.7]{Cha21} that
 	$$\lim_{n\to \infty}\frac{\#T_n}{\#\widehat{H}^0(\mathcal{Y},\overline{\mathcal{L}}_e(-\epsilon/(4c))|_{\mathcal{Y}}^{\otimes n})}=1.$$
 	Hence, by Proposition \ref{pro_positive-density} we can choose a sequence $(\widetilde{s}'_{e,n})_{n\in\mathbb{Z}_{\ge 1}}$ of sections
    $\widetilde{s}'_{e,n}\in T_n$ satisfying
 	$$\lim_{n\to \infty}\frac{1}{n}\int_{\mathcal{Y}(\mathbb{C})}\log|\widetilde{s}'_{e,n}|_{\overline{\mathcal{L}}_e(-\epsilon/(4c))^{\otimes n}}c_1(\overline{\mathcal{L}}_e)^{e-1}=0.$$
 	Note that the integrand is always non-positive. Thus, the equation stays true after replacing $c_1(\overline{\mathcal{L}}_e)^{e-1}$ by another positive form. Hence, we get
 	$$\lim_{n\to \infty}\frac{1}{n}\int_{\mathcal{Y}(\mathbb{C})}\log|\widetilde{s}'_{e,n}|_{\overline{\mathcal{L}}_e(-\epsilon/(4c))^{\otimes n}}c_1(\overline{\mathcal{L}}_1)\cdots c_1(\overline{\mathcal{L}}_{e-1})=0.$$
 	By Equation (\ref{equ_metric-change}) this is equivalent to
 	\begin{align*}
 	&\lim_{n\to \infty}\frac{1}{n}\int_{\mathcal{Y}(\mathbb{C})}\log|\widetilde{s}'_{e,n}|_{\overline{\mathcal{L}}_e^{\otimes n}}c_1(\overline{\mathcal{L}}_1)\cdots c_1(\overline{\mathcal{L}}_{e-1})\\
 	&=-\epsilon/(4c)\cdot\deg(\mathcal{L}_{1,\mathbb{C}}\cdots \mathcal{L}_{e-1,\mathbb{C}}\cdot \mathcal{Y}_{\mathbb{C}})=-\epsilon/4.
 	\end{align*}
 	Thus, there is an $n_1\in\mathbb{Z}_{\ge 1}$ such that for all $n\ge n_1$ we have
 	$$\lim_{n\to \infty}\frac{1}{n}\int_{\mathcal{Y}(\mathbb{C})}\log|\widetilde{s}'_{e,n}|_{\overline{\mathcal{L}}_e^{\otimes n}}c_1(\overline{\mathcal{L}}_e)^{e-1}>-2\epsilon/4.$$
 	It has been worked out by Charles \cite[Theorem 2.17]{Cha21} that there is an $n_2\ge n_1$ such that for all $n\ge n_2$ the image of the restriction map
 	$$\widehat{H}^0(\mathcal{X},\overline{\mathcal{L}}_e^{\otimes n})\to H^0(\mathcal{Y},\mathcal{L}_e|_{\mathcal{Y}}^{\otimes n})$$ 
 	completely contains the subset $\widehat{H}^0(\mathcal{X},\overline{\mathcal{L}}_e(-\epsilon/(4c))|_{\mathcal{Y}}^{\otimes n})\subseteq H^0(\mathcal{Y},\mathcal{L}_e|_{\mathcal{Y}}^{\otimes n})$.
 	Hence we can choose sections $s'_{e,n}\in \widehat{H}^0(\mathcal{X},\overline{\mathcal{L}}_e^{\otimes n})$ with $s'_{e,n}|_{\mathcal{Y}}=\widetilde{s}'_{e,n}$ for all $n\ge n_2$.
 	
 	Since $\mathcal{L}_{e}$ is ample, there is an $n_3\ge n_2$ such that $\mathcal{L}_{e}^{\otimes n}$ is basepoint-free for all $n\ge n_3$.
 	By the arithmetic Hilbert--Samuel formula for arithmetically ample hermitian line bundles \cite[Proposition 5.41]{Mor14} there exists an $n_4\ge n_3$ such that $\widehat{H^0}(\mathcal{X},\overline{\mathcal{L}}_{e}^{\otimes n})$ contains a basis of $H^0(\mathcal{X},\mathcal{L}_e^{\otimes n})$ for all $n\ge n_4$.
 	
 	For any integer $n\ge n_4$ we write $n=n'n_4+n''$ with $n_4\le n''< 2n_4$.
 	By the definition of the arithmetic intersection product (\ref{equ_intersection-definition}) we have for all $n\ge n_4$
 	\begin{align*}
 	\overline{\mathcal{L}}_e^{\otimes n}\cdot(\mathcal{Y},0)=n'\left(\mathrm{div}(s'_{e,n_4})\cdot\mathcal{Y},\left[-\log|s'_{e,n_4}|^2_{\overline{\mathcal{L}}_e^{\otimes n_4}}\right]_{\mathcal{Y}(\mathbb{C})}\right)+\overline{\mathcal{L}}_e^{\otimes n''}\cdot(\mathcal{Y},0).
 	\end{align*}
 	There is an $n_5\ge n_4$ such that
 	$$\frac{2n_4}{n}(\overline{\mathcal{L}}_1\cdots\overline{\mathcal{L}}_e\cdot(\mathcal{Y},0))<\epsilon/4.$$
 	for all $n\ge n_5$.
	Thus, we can compute for $n\ge n_5$
	\begin{align}\label{equ_bound-by-section}
		&(\overline{\mathcal{L}}_1\cdots\overline{\mathcal{L}}_e\cdot(\mathcal{Y},0))\\
		&\le\frac{n'}{n}\left(\overline{\mathcal{L}}_1\cdots\overline{\mathcal{L}}_{e-1}\cdot\left(\mathrm{div}(s'_{e,n_4})\cdot\mathcal{Y},\left[-\log|s'_{e,n_4}|^2_{\overline{\mathcal{L}}_e^{\otimes n_4}}\right]_{\mathcal{Y}(\mathbb{C})}\right)\right)+\epsilon/4\nonumber\\
		&=\frac{n'}{n}\left(\overline{\mathcal{L}}_1\cdots\overline{\mathcal{L}}_{e-1}\cdot\left(\mathrm{div}(s'_{e,n_4})\cdot\mathcal{Y},0\right)\right)\nonumber\\
		&\quad-\frac{n'}{n}\int_{\mathcal{Y}(\mathbb{C})}\log|s'_{e,n_4}|_{\overline{\mathcal{L}}_e^{\otimes n_4}}c_1(\overline{\mathcal{L}}_1)\cdots c_1(\overline{\mathcal{L}}_{e-1})+\epsilon/4\nonumber\\
		&\le \frac{n'}{n}\left(\overline{\mathcal{L}}_1\cdots\overline{\mathcal{L}}_{e-1}\cdot\left(\mathrm{div}(s'_{e,n_4})\cdot\mathcal{Y},0\right)\right)+3\epsilon/4.\nonumber
	\end{align}	
	By construction $\mathrm{div}(s'_{e,n_4})\cdot \mathcal{Y}=\mathrm{div}(\widetilde{s}'_{e,n_4})$ is generically smooth and irreducible. Hence, by the induction hypothesis there is an $n_6\ge n_5$ such that for every $n\ge n_6$ there is a tuple $(s_{1,n},\dots,s_{e-1,n})\in H_n(\mathcal{Y}\cdot \mathrm{div}(s'_{e,n_4}))$ such that
	\begin{align}\label{equ_induction-hypothesis}
	&\left(\overline{\mathcal{L}}_1\cdots\overline{\mathcal{L}}_{e-1}\cdot\left(\mathrm{div}(s'_{e,n_4})\cdot\mathcal{Y},0\right)\right)\\
	&\le \frac{1}{n^{e-1}}\sum_{p\in|\mathrm{Spec}(\mathbb{Z})|} i_p(\mathcal{Y}\cdot\mathrm{div}(s'_{e,n_4})\cdot\mathrm{div}(s_{1,n})\cdots\mathrm{div}(s_{e-1,n}))\log p+\epsilon/4.\nonumber
	\end{align}
	Since $\overline{\mathcal{L}}_e^{\otimes n}$ is basepoint-free and $\widehat{H}^0(\mathcal{X},\overline{\mathcal{L}}_e^{\otimes n})$ contains a basis of $H^0(\mathcal{X},\mathcal{L}_e^{\otimes n})$ for $n_4\le n<2n_4$, we can choose for all $n\ge n_6$ sections $t_n\in \widehat{H}^0(\mathcal{X},\overline{\mathcal{L}}_e^{\otimes n''})$
	such that
	$$\dim \left( \mathcal{Y}\cap \mathrm{Supp}(\mathrm{div}(t_n))\cap\bigcap_{j=1}^{e-1}\mathrm{Supp}(\mathrm{div}(s_{j,n}))\right)=0.$$
	Now we set
	$$s_{e,n}={s'}_{e,n_4}^{\otimes n'}\otimes t_n\in\widehat{H}^0(\mathcal{X},\overline{\mathcal{L}}_e^{\otimes n}).$$
	Then we get a tuple $(s_{1,n},\dots,s_{e,n})\in H_n(\mathcal{Y})$ for every $n\ge n_6$. By linearity and effectivity we always have
	\begin{align}\label{equ_ip-bound}
	n'\cdot i_p(\mathcal{Y}\cdot \mathrm{div}(s'_{e,n_4})\cdot\mathrm{div}(s_{1,n})\cdots \mathrm{div}(s_{e-1,n}))\le i_p(\mathcal{Y}\cdot\mathrm{div}(s_{1,n})\cdots \mathrm{div}(s_{e,n})).
	\end{align}
	Putting (\ref{equ_bound-by-section}), (\ref{equ_induction-hypothesis}) and (\ref{equ_ip-bound}) together we finally get
	$$(\overline{\mathcal{L}}_1\cdots\overline{\mathcal{L}}_e\cdot(\mathcal{Y},0))\le \frac{1}{n^{e}}\sum_{p\in|\mathrm{Spec}(\mathbb{Z})|}i_p(\mathcal{Y}\cdot \cdot\mathrm{div}(s_{1,n})\cdots \mathrm{div}(s_{e,n}))\log p+ \epsilon.$$
	This proves inequality (\ref{equ_limit-lower-bound}) in dimension $e\ge 2$.

	The formula in the theorem now follows by combining inequalities (\ref{equ_sup-upper-bound}) and (\ref{equ_limit-lower-bound}) with $\mathcal{Z}=\mathcal{Y}$ and using that by Lemma \ref{lem_intersection-Z0} we have
	$$(\overline{\mathcal{L}}_1\cdots\overline{\mathcal{L}}_e\cdot(\mathcal{Y},0))=(\overline{\mathcal{L}}_1|_\mathcal{Y}\cdots\overline{\mathcal{L}}_e|_\mathcal{Y})$$
	for any arithmetic subvariety $\mathcal{Y}\subseteq \mathcal{X}$.
\end{proof}

\section*{Acknowledgments}
I would like to thank Philipp Habegger and Harry Schmidt for useful discussions related to this paper. Further, I gratefully acknowledge support from the Swiss National Science Foundation grant ``Diophantine Equations: Special Points, Integrality, and Beyond'' (n$^\circ$ 200020\_184623).


\begin{thebibliography}{14}
	\bibitem{Ara74} Arakelov, S. Y.: \emph{Intersection theory of divisors on an arithmetic surface}. Izv. Akad. USSR {\bf 8}, no. 6, 1167--1180 (1974).
	\bibitem{BCM20} Bayraktar, T.; Coman, D.; Marinescu, G.: \emph{Universality results for zeros of random holomorphic sections}. Trans. Amer. Math. Soc. {\bf 373} (2020), no. 6, 3765--3791.
	\bibitem{Bil97} Bilu, Y.: \emph{Limit distribution of small points on algebraic tori}. Duke Math. J. {\bf 89} (1997), 465--476.
	\bibitem{BGS94} Bost, J.-B.; Gillet, H.; Soulé, C. \emph{Heights of projective varieties and positive Green forms}. J. Amer. Math. Soc. {\bf 7} (1994), no. 4, 903--1027.
	\bibitem{Cha06} Chambert-Loir, A.: \emph{Mesures et équidistribution sur les espaces de Berkovich}. J. Reine Angew. Math. {\bf 595} (2006), 215--235.
	\bibitem{Cha21} Charles, F.: \emph{Arithmetic ampleness and an arithmetic Bertini theorem}. Ann. Sci. Éc. Norm. Supér. (4) {\bf 54} (2021), no. 6, 1541--1590.
	\bibitem{Che08} Chen, H.: \emph{Positive degree and arithmetic bigness}. Preprint, arXiv:0803.2583 (2008).
	\bibitem{CMM17} Coman, D.; Ma, X; Marinescu, G.: \emph{Equidistribution for sequences of line bundles on normal Kähler spaces}. Geom. Topol. {\bf 21} (2017), no. 2, 923--962.
	\bibitem{ET50} Erdős, P.; Turán, P.: \emph{On the distribution of roots of polynomials}. Ann. Math. {\bf 51} (1950), 105--119.
	\bibitem{FL21} Freyer, A.; Lucas, E.: \emph{Connecting lattice points and volume through successive minima}. Preprint, arXiv:2105.13090 (2021).
	\bibitem{GS90} Gillet, H.; Soulé, C.: \emph{Arithmetic intersection theory}. Inst. Hautes Études Sci. Publ. Math. No. {\bf 72} (1990), 93--174.
	\bibitem{GS92} Gillet, H.; Soulé, C.: \emph{An arithmetic Riemann-Roch theorem}. Invent. Math. {\bf 110} (1992), no. 3, 473--543. 
	\bibitem{Gra07} Granville, A.: \emph{The distribution of roots of a polynomial}. In: Equidistribution in number theory, an Introduction, 93--102, NATO Sci. Ser. II Math. Phys. Chem., {\bf 237}, Springer, Dordrecht, 2007.
	\bibitem{GD60} Grothendieck, A.: \emph{Éléments de géométrie algébrique. I. Le langage des schémas}. Inst. Hautes Études Sci. Publ. Math. No. {\bf 4}, 1960.
	\bibitem{Har77} Hartshorne, R.: \emph{Algebraic geometry}. Graduate Texts in Mathematics, No. {\bf 52}. Springer-Verlag, New York-Heidelberg, 1977. 
	\bibitem{JL21} Jiang, C.; Li, Z.: \emph{Algebraic reverse Khovanskii--Teissier inequality via Okounkov bodies}. Preprint, arXiv:2112.02847 (2021).
	\bibitem{KK12} Kaveh, K.; Khovanskii, A. G.: \emph{Newton--Okounkov bodies, semigroups of integral points, graded algebras and intersection theory}. Ann. of Math. {\bf 176} (2012), no. 2, 925--978.
	\bibitem{KMY02} Kawaguchi, S.; Moriwaki, A.; Yamaki, K.: \emph{Introduction to Arakelov geometry}. In: Algebraic geometry in East Asia (Kyoto, 2001), World Sci. Publ., 1--74, 
	\bibitem{LM09} Lazarsfeld, R.; Mustaţă, M.: \emph{Convex bodies associated to linear series}. Ann. Sci. Éc. Norm. Supér. (4) {\bf 42} (2009), no. 5, 783--835.
	\bibitem{Mor11} Moriwaki, A.: \emph{Free basis consisting of strictly small sections}´. Int. Math. Res. Not. IMRN 2011, no. 6, 1245--1267.
	\bibitem{Mor14}	Moriwaki, A.: \emph{Arakelov geometry}. Translated from the 2008 Japanese original. Translations of Mathematical Monographs, {\bf 244}. American Mathematical Society, Providence, RI, 2014. 
	\bibitem{Oko96} Okounkov, A.: \emph{Brunn-Minkowski inequality for multiplicities}. Invent. Math. {\bf 125} (1996), no. 3, 405--411.
	\bibitem{Oko03} Okounkov, A.: \emph{Why would multiplicities be log-concave?} In: The orbit method in geometry and physics (Marseille, 2000), 329--347, Progr. Math., {\bf 213}, Birkhäuser Boston, Boston, MA, 2003.
	\bibitem{Pri11} Pritsker, I. E.: \emph{Distribution of algebraic numbers}. J. Reine Angew. Math. {\bf 657} (2011), 57--80.
	\bibitem{SZ99} Shiffman, B.; Zelditch, S.: \emph{Distribution of zeros of random and quantum chaotic sections of positive line bundles}. Comm. Math. Phys. {\bf 200} (1999), no. 3, 661--683.
	\bibitem{Tia90} Tian, G.: \emph{On a set of polarized Kähler metrics on algebraic manifolds}. J. Differential Geom. {\bf 32} (1990), no. 1, 99--130. 
	\bibitem{Ull98} Ullmo, E.: \emph{Positivité et discrétion des points algébriques des courbes}. Ann. of Math. (2) {\bf 147} (1998), no. 1, 167--179. 
	\bibitem{YZ21} Yuan, X.; Zhang, S.: \emph{Adelic line bundles over quasi-projective varieties}. Prerpint, arXiv:2105.13587 (2021).
	\bibitem{Zha95} Zhang, S.: \emph{Small points and adelic metrics}. J. Algebraic Geom. {\bf 4} (1995), no. 2, 281--300.
	\bibitem{Zha98} Zhang, S.: \emph{Equidistribution of small points on abelian varieties}. Ann. of Math. (2) {\bf 147} (1998), no. 1, 159--165. 
\end{thebibliography}
\end{document}